\documentclass[12pt]{article}

\usepackage{amsmath,amssymb,amsthm}
\usepackage{tikz-cd}
\usepackage{tkz-euclide}
\usepackage{graphicx}
\usepackage{lscape}
\usepackage{xcolor}
\usepackage{hyperref}
\usetikzlibrary{shapes.geometric}
\newdimen\R
\usepackage[centering, includeheadfoot, hmargin=1.0in, tmargin=1.0in, bmargin=1in, headheight=7pt]{geometry}
 
\newtheorem{definition}{Definition}
\newtheorem*{definition*}{Definition}
\newtheorem{theorem}{Theorem}
\newtheorem*{theorem*}{Theorem}
\newtheorem{corollary}{Corollary}
\newtheorem*{corollary*}{Corollary}
\newtheorem{lemma}{Lemma}
\newtheorem*{lemma*}{Lemma}
\newtheorem{proposition}{Proposition}
\newtheorem*{proposition*}{Proposition}
\newtheorem{conjecture}{Conjecture}
\newtheorem*{conjecture*}{Conjecture}
\theoremstyle{remark}
\newtheorem{example}{Example}
\newtheorem*{example*}{Example}
\newtheorem{remark}{Remark}
\newtheorem*{remark*}{Remark}

\newcommand*{\polyset}[1]{\text{Poly($#1$)}}
\newcommand*{\pathset}[2]{\mathbf{P}_{#1,#2}}

\newcommand*{\poly}[1]{\mathcal{P}_{#1}}
\newcommand*{\qm}[1]{m_{#1-1,#1+1}}
\newcommand*{\fm}[2]{m_{#1,#2}}

\newcommand*{\Jor}{J{\o}rgensen }

\newcommand*{\rom}[1]{\expandafter\@slowromancap\romannumeral #1@}

\newcommand*{\F}{\mathcal{F}}

\newcommand*{\Z}{\mathbb{Z}}
\newcommand*{\D}{\mathcal{D}}

\title{Periodic Infinite Frieze Patterns of Type $\Lambda_{p_1,\ldots,p_n}$ and Dissections on Annuli}
\author{Esther Banaian and Jiuqi Chen}
\date{}

\begin{document}

\maketitle
\begin{abstract}
    Finite frieze patterns with entries in $\mathbb{Z}[\lambda_{p_1},\ldots,\lambda_{p_s}]$ where $\{p_1,\ldots,p_s\} \subseteq \mathbb{Z}_{\geq 3}$ and $\lambda_p = 2 \cos(\pi/p)$ were shown to have a connection to dissected polygons by Holm and J{\o}rgensen. We extend their work by studying the connection between infinite frieze patterns with such entries and dissections of annuli and once-punctured discs. We give an algorithm to determine whether a frieze pattern with entries in $\mathbb{Z}[\lambda_{p_1},\ldots,\lambda_{p_s}]$, finite or infinite, comes from a dissected surface. We introduce quotient dissections as a realization for some frieze patterns unrealizable by an ordinary dissection. We also introduce two combinatorial interpretations for entries of frieze patterns from dissected surfaces. These interpretations are a generalization of matchings introduced by Broline, Crowe, and Isaacs for finite frieze patterns over $\mathbb{Z}$. 
\end{abstract}
\tableofcontents

\section{Introduction}

Frieze patterns, defined in the finite case by Coxeter \cite{10} and in the infinite case by Tschabold \cite{20}, have a rich connection with the geometry of triangulated and dissected surfaces. For a beautiful survey on frieze patterns and their relation to both geometry and algebra see \cite{16}. The goal of this paper is to delve further into the geometric connections, principally in the case of dissected surfaces.

We define finite and infinite frieze patterns simultaneously.

\begin{definition}
A (finite) frieze pattern $\F$ over a ring $R$ is a (finite) array $(m_{i,j})_{i,j\in \Z, j\geq i}$ of shifted infinite rows such that  
\begin{enumerate}
    \item for all $i,j$, $m_{i,j} \in R$;
    \item $\fm{i}{i} = 0$ and $\fm{i}{i+1} = 1$ for all $i\in \Z$; and
    \item every diamond in $\F$ of the form $\begin{array}{ccc} & \fm{i}{j} & \\ \fm{i-1}{j} & & \fm{i}{j+1} \\ & \fm{i+1}{j+1} & \end{array}$ satisfies the unimodular rule: $\fm{i-1}{j}\fm{i}{j+1}-\fm{i}{j}\fm{i+1}{j+1} = 1$.
\end{enumerate}
\end{definition}

We illustrate the indexing of a frieze pattern. Notice that all entries of the form $m_{i,j}$ for fixed $i$ lie along a SE diagonal while all entries with fixed $j$ lie along a SW diagonal. 

\begin{center}
$\begin{array}{ccccccccccccccccccccccccc}
 \ldots&0&&0&&0&&0&&0&&\ldots\\
 \ldots& &1&&1&&1&&1&&&\ldots\\
 \ldots&m_{\color{red}{-1},\color{black}{1}}&&m_{0,2}&&m_{1,3}&& m_{2,4}&&m_{3,\color{blue}{5}}&&\ldots\\
  \ldots&&m_{\color{red}{-1},\color{black}{2}}&&m_{0,3}&&m_{1,4}&&m_{2,\color{blue}{5}}&&&\ldots\\
 \ldots  & m_{-2,2}&&m_{\color{red}{-1},\color{black}{3}}&&m_{0,4}&&m_{1,\color{blue}{5}}&&m_{2,6}&&\ldots\\
 \ldots &&m_{-2,3}&&m_{\color{red}{-1},\color{black}{4}}&&m_{0,\color{blue}{5}}&&m_{1,6}&&&\ldots\\
 &&&.^{.^.}&&m_{\color{red}{-1},\color{blue}{5}}&&\ddots&&&&&&
 \end{array}$
 \end{center}
 \vspace{0.2cm}
 
  We define several terms related to frieze patterns.
 
 \begin{definition}
 \begin{enumerate}
\item The first nontrivial row is called the \emph{quiddity row}. We will simply refer to this as the first row.  
\item The number of nontrivial rows of a finite frieze pattern is the \emph{width}.
\item If the quiddity row $\ldots, a_0, a_1,\ldots,a_n, a_{n+1},\ldots$  of a frieze pattern is $n$-periodic, so that $a_i = a_{i+n}$ for all $i \in \Z$, we call $(a_1,\ldots,a_n)$ a \emph{quiddity cycle}.
\end{enumerate}
\end{definition}

By the unimodular rule, the quiddity sequence of a frieze pattern determines the entire frieze pattern. Hence a quiddity cycle also determines a frieze pattern.

Soon after Coxeter defined frieze patterns in \cite{11}, Conway and Coxeter showed that finite frieze patterns of positive integers are in bijection with triangulated polygons \cite{10}; see Section \ref{subsec:ConwayCoxeter}. When Caldero and Chapoton demonstrated a connection between cluster algebras and frieze patterns \cite{6}, there was renewed interest in frieze patterns. One resulting generalization was from Baur, Parsons, and Tschabold, who showed infinite frieze patterns of positive integers are in bijection with triangulated annuli and once-punctured discs \cite{1}. See Section \ref{subsec:BPT} for more details.

Holm and J{\o}rgensen generalized Conway and Coxeter's result in a different direction and investigated finite frieze patterns from dissections of polygons \cite{14}. In place of positive integers, they work over the ring of algebraic integers of the field $\mathbb{Q}(\lambda_{p_1},\ldots,\lambda_{p_s})$ where $p_1,\ldots,p_s$ is a list of sizes of subgons in the dissection and $\lambda_p = 2\cos(\pi/p)$. More specifically, their frieze patterns have quiddity cycles where each entry is of the form $\sum_{i=1}^s c_i \lambda_{p_i}$ for $c_i \in \mathbb{Z}_{\geq 0}$. We say that such frieze patterns are of \emph{Type $\Lambda_{p_1,\ldots,p_s}$}. We give more details about this work in Section \ref{subsec:HJ}.

We combine these directions of generalization by investigating infinite frieze patterns of Type $\Lambda_{p_1,\ldots,p_s}$. Some of these frieze patterns correspond to dissections of annuli and once-punctured discs.  In Section \ref{sec:FPType} we discuss some preliminary results about frieze patterns of Type $\Lambda_{p_1,\ldots,p_s}$ while in Section \ref{sec:Dissections} we discuss the details of dissections of annuli and once-punctured discs.

In Section \ref{sec:FPFromDissection}, we build towards determining which infinite frieze patterns of Type $\Lambda_{p_1,\ldots,p_s}$ arise from dissected annuli and once-punctured discs. As a first step, we give an explicit description of some frieze patterns which do not arise in this way (i.e., are ``unrealizable'') in Proposition \ref{proposition:unrealizability1}. When we restrict to the subclass of \emph{skeletal} frieze patterns of Type $\Lambda_{p_1,\ldots,p_s}$, we give a direct characterization of which frieze patterns which are realizable by a dissection of an annulus or once-punctured disc (Proposition \ref{prop:realizability}). See Definition \ref{def:SkeletalDissection} where skeletal dissections are described.

There is a class of skeletal frieze patterns which are neither unrealizable by Proposition \ref{proposition:unrealizability1} nor realizable by Proposition \ref{prop:realizability}. We demonstrate these frieze patterns in the table below. We consider some 2-periodic frieze patterns with quiddity cycle $(a,b)$. All resulting frieze patterns are skeletal. We write \checkmark when the frieze pattern generated by quiddity cycle $(a,b)$ is realizable by Proposition \ref{prop:realizability}, we write $\times$ when the resulting frieze pattern is unrealizable by Proposition \ref{proposition:unrealizability1}, and otherwise we write ?.

\begin{center}
\begin{tabular}{c|c|c|c|c}
$a$ \textbackslash $b$ & 3 & $2 + \sqrt{2}$ & $1 + 2 \sqrt{2}$ & $3\sqrt{2}$ \\\hline
3 & \checkmark & \checkmark & ? & $\times $\\\hline
$2 + \sqrt{2}$ & \checkmark & \checkmark & \checkmark & ?\\\hline
$1 + 2\sqrt{2}$ & ? & \checkmark & \checkmark& \checkmark\\\hline
$3\sqrt{2}$ & $\times$ & ? & \checkmark & \checkmark\\
\end{tabular}
\end{center}

The presence of these frieze patterns between the realizable and unrealizable frieze patterns motivated the definition of \emph{quotient dissections}, which provide a realization for all skeletal frieze patterns which are neither unrealizable by Proposition \ref{proposition:unrealizability1} nor realizable by Proposition \ref{prop:realizability}. See Section \ref{subsec:Dumpling} for details on quotient dissections. Then, Section \ref{subsec:RealizabilityAlgorithm} provides an algorithm to determine whether an arbitrary (not necessarily infinite) frieze pattern of Type $\Lambda_{p_1,\ldots,p_s}$ is realizable by a dissection or a quotient dissection.

In Section 6, we provide two combinatorial interpretations of entries of frieze patterns from dissections of surfaces. These interpretations are sums over sets of assignments of subgons to vertices with two different rules for weighting the matchings. Theorem \ref{thm:LocalWeight} shows that the \emph{local weighting} provides a combinatorial interpretation for the entries of a frieze pattern from a dissection or quotient dissection. We show that summing over the two weightings (local and traditional) always gives the same result in Theorem \ref{thm:EqualWeights}.

Section 7 discusses a third type of weighting, \emph{annulus weighting}, which gives the \emph{growth coefficients} of infinite frieze patterns from dissections. See Theorem \ref{thm:CombinatorialGrowthC}. Growth coefficients are a sequence of numbers from a periodic infinite frieze pattern which measure the growth rate of its entries and were first defined in \cite{2}.

In Section 8, we compare our weighted matchings to weak T-paths recently defined in \cite{8}. We conclude in Section 9 with a discussion about which infinite frieze patterns of Type $\Lambda_{p_1,\ldots,p_s}$ have all positive entries. We cannot guarantee that all frieze patterns from quotient dissections have positive entries, but we conjecture this is true in Conjecture \ref{conj:Positive}.
\section{Background}

\subsection{Finite Frieze Patterns of Positive Integers}\label{subsec:ConwayCoxeter}

The correspondence between triangulated polygons and finite frieze patterns over $\mathbb{Z}_{\geq 0}$ is a model for the way we will connect other surfaces to frieze patterns.
Given a triangulated $n$-gon, label the vertices $v_1,\ldots,v_n$ in counterclockwise order. Suppose that $v_i$ is incident to $a_i$ triangles. Then, we associate the frieze pattern with quiddity cycle $(a_1,\ldots,a_n)$ to the triangulated polygon. See Figure \ref{fig:CCBijection} for an example. Conway and Coxeter prove that this map is a bijection. 

\begin{theorem}[\cite{10}]
Finite frieze patterns of positive integers with width $n$ are in bijection with triangulated $(n+3)$-gons. 
\end{theorem}

\begin{center}
\begin{figure}\label{fig:CCBijection}
\centering
\raisebox{-.5\totalheight}{\begin{tikzpicture}[scale = 1]
\node[left, blue](A) at (0:1.5\R){};
\node[left, blue](B) at (45:1.5\R){};
\node[below,blue](C) at (90:1.5\R){};
\node[right,blue](D) at (135:1.5\R){};
\node[right,blue](E) at (180:1.5\R){};
\node[above, blue](C) at (-45:1.5\R){};
\node[above, blue](D) at (135+135:1.5\R){};
\node[right,blue](E) at (180+45:1.5\R){};
  \draw (0:1.5\R) node[right]{5} --  (45:1.5\R) node[right]{1} --   (90:1.5\R) node[above]{3}  --   (135:1.5\R) node[above]{1} -- (135 + 45:1.5\R) node[left]{3} -- (180 + 45:1.5\R) node[left]{2}-- (270 :1.5\R) node[below]{2} -- (270 + 45:1.5\R) node[below]{1}--(0:1.5\R);
  \draw(0:1.5\R) -- (90:1.5\R);
  \draw(0:1.5\R) -- (180:1.5\R);  
   \draw(0:1.5\R) -- (270:1.5\R);
   \draw(90:1.5\R) -- (180:1.5\R);
   \draw(0:1.5\R) -- (225:1.5\R);
  \end{tikzpicture}} 
 \scalebox{0.75}{$\begin{array}{cccccccccccccccccccc}
   &&0&&  0&&  0&&  0&&  0&&  0&&  0&&  0&&0&\\
  &&&  1&&  1&&  1&&  1&&  1&&  1&&  1&&1&&1\\
 &&1&&3&&1&&  3&&  2&&  2&&  1&&5&&1&\\
  &&&2&&  2&&  2&& 5&&  3&&  1&&4&&4&&2\\
   &&7&&1&& 3&&  3&& 7&&  1&&3&&3&&7&\\
  &&&3&&1&&  4&&  4&&  2&&2&&2&&5&&3\\
   &&1&&2&&1&&5&&  1&&3&&1&&3&&1&\\
  &&&1&&1&&1&&  1&&1&&1&&1&&1&&1\\
   &&0&&0&&0&&0&&0&&0&&0&&0&&0&\\
 \end{array}$}
 \caption{A triangulated octagon and the corresponding frieze pattern over $\mathbb{Z}_{\geq 0}$ of width 5}
 \end{figure}
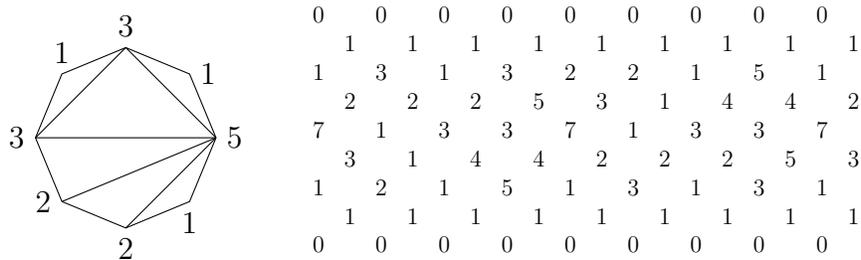
 \end{center}

\subsection{Infinite Frieze Patterns of Positive Integers}\label{subsec:BPT}

Tschabold produced infinite frieze patterns by using the same algorithm as Conway and Coxeter but starting with a triangulated once-punctured disc \cite{20}. Soon after, Tschabold along with Baur and Parsons extended this work by also considering triangulated annuli \cite{1}. With these surfaces, we count adjacent triangles by drawing a small circle around each vertex, then counting the number of triangles this circle passes through. When the triangulation includes a self-folded triangle, the small circle will see the two sides of the triangle as two separate triangles.  When forming a quiddity cycle from a triangulated annulus, one picks one boundary to work with respect to. As a convention, we will work with respect to the outer boundary. However, in Section \ref{subsec:InAndOutFP} we will discuss how the frieze patterns from the outer boundary and inner boundary are related. See Figure \ref{fig:BPTBijection} for an example of a frieze pattern from a triangulated annulus.

It turns out that all periodic infinite frieze patterns of positive integers arise from a triangulation of either a once-punctured disc or annulus. 

\begin{theorem}[\cite{1}]
Periodic infinite frieze patterns of positive integers are in bijection with triangulated once-punctured discs and annuli. 
\end{theorem}

\begin{figure}
    \centering
    \begin{center}
\raisebox{-.4\totalheight}{\begin{tikzpicture}[scale = .9]
\draw (0,0) circle (2\R);
\draw (0,0) circle (\R);
\coordinate(A) at (90:2\R);
\coordinate(B) at (0:\R);
\coordinate(C) at (-60:2\R);
\coordinate(D) at (210:2\R);
\coordinate(E) at (90:\R);
\coordinate(F) at (90:1.5\R);
\coordinate(G) at (270:1.5\R);
\node[above, scale = 0.8] at (A) {4};
\node[right, scale = 0.8] at (C) {4};
\node[left, scale = 0.8] at (200:2\R){1};
\node[circle, fill = black, scale = 0.3] at (200: 2\R){};
\draw(90:2\R) -- (90:\R);
\draw(A) to [out=-20,in=30] (B);
\draw(C) to [out = 80, in = -30](B);
\draw(C) to [out = 180, in = 200, looseness = 1.8](A);
\draw(C) to [out = 180, in = 180, looseness = 1.9](E);
  \end{tikzpicture}} 
 \scalebox{0.8}{$\begin{array}{cccccccccccccccccccc}
  0&&0&&0&&0&&0&&0&&0&&0\\
 &1& &1&&1&&1&&1&&1&&1&\\
 1&&4&&4&&1&&4&&4&&1&&4\\
&3& &15&&3&&3&&15&&3&&3&\\
 8&&11&&11&&8&&11&&11&&8&&11\\
   &29&&8&&29&&29&&8&&29&&29&\\
   105&&21&&21&&105&&21&&21&&105&&21\\
  & &&&&&&\vdots&&&&&\\
 \end{array}$}
 \end{center}
    \caption{A triangulated annulus and the corresponding infinite frieze pattern over $\mathbb{Z}_{\geq 0}$}
    \label{fig:BPTBijection}
\end{figure}
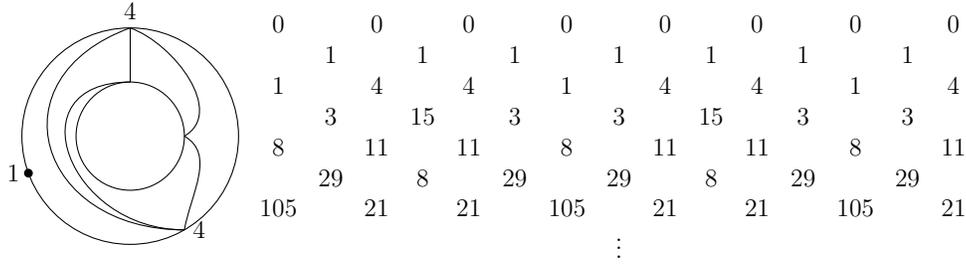

\subsection{Finite Frieze Patterns from Dissected Polygons}\label{subsec:HJ}

Holm and \Jor defined a class of frieze patterns arising from dissected polygons \cite{14}.

\begin{definition}
A \emph{dissection} $\mathcal{D}$ of a surface with marked points $(S,M)$ is a set of non-crossing arcs which divide $S$ into polygons. We refer to these polygons as \emph{subgons}. 
A \emph{$p$-angulation} is a dissection which breaks $(S,M)$ into $p$-gons. 
\end{definition}

Note that a $n$-gon can only be $p$-angulated if $n+2$ is divisible by $p$. 


Given a polygon $P$ with dissection $\mathcal{D}$, let $v_1,\ldots,v_n$ be the vertices of P, labeled in counterclockwise order. Let $A_i$ be the set of subgons in the dissection incident to $v_i$. For any $p \geq 3$, let $\lambda_p = 2\cos(\pi/p)$. Then, we set $m_{i-1,i+1} = \sum_{\mathcal{P} \in A_i} \lambda_{\vert \mathcal{P} \vert}$, where $\vert \mathcal{P} \vert$ is the number of sides of $\mathcal{P}$. For example, in Figure \ref{fig:HJInjection}, if the top left vertex of the hexagon is $v_1$, $A_1 = \{3,3,4\}$. 

\begin{remark}
The number $\lambda_p$ is the ratio of the length of a 1-diagonal (a diagonal which skips 1 vertex) and the length of a side of a regular $p$-gon. The first few values of $\lambda_p$ are familiar numbers.
\begin{center}
    \begin{tabular}{|c|c|}\hline
        $p$ & $\lambda_p$  \\\hline
        3 & 1 \\
        4 & $\sqrt{2}$ \\
        5 & $\frac{1 + \sqrt{5}}{2} = \phi$ \\
        6 & $\sqrt{3}$\\\hline
    \end{tabular}
\end{center}
In particular, since $\lambda_3 = 1$, the Holm and J{\o}rgensen algorithm reduces to that of Conway and Coxeter when the dissection is a triangulation. 

Lemma \ref{lem:ChebyshevGeometric} gives expressions for all diagonals in a regular polygon. 
\end{remark}

\begin{theorem}[\cite{14}]
 Using the above algorithm, there is an injection from the set of dissections of an $(n+3)$-gon to width $n$ frieze patterns. 
\end{theorem}

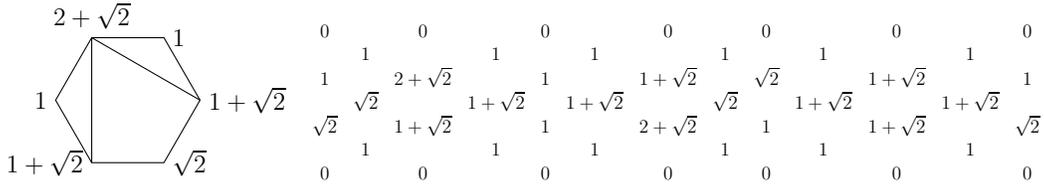
\begin{figure}
    \centering
    \raisebox{-.4\totalheight}{\begin{tikzpicture}[scale = 0.8]
\draw (0:1.5\R) node[right, scale = 0.8]{$1+\sqrt{2}$} --  (60:1.5\R) node[right, scale = 0.8]{$1$} --   (120:1.5\R) node[above, scale = 0.8]{$2+\sqrt{2}$}  --   (180:1.5\R) node[left, scale = 0.8]{$1$} -- (240:1.5\R) node[left, scale = 0.8]{$1 + \sqrt{2}$} -- (300:1.5\R) node[right, scale = 0.8]{$\sqrt{2}$}-- (0:1.5\R);
\draw(0:1.5\R) -- (120:1.5\R);
\draw(120:1.5\R) -- (240:1.5\R);
 \end{tikzpicture}} 
 \scalebox{0.6}{$\begin{array}{cccccccccccccccccccc}
  0&&0&&0&&0&&0&&0&& 0\\
 &1& &1&&1&&1&&1&&1&\\
 1&&2+\sqrt{2}&&1&&1+\sqrt{2}&&\sqrt{2}&&1 + \sqrt{2} && 1\\
&\sqrt{2}& &1 + \sqrt{2}&&1+\sqrt{2}&&\sqrt{2}&&1 + \sqrt{2}& &1 + \sqrt{2}&\\
  \sqrt{2}&&1 + \sqrt{2}&&1&&2+\sqrt{2}&&1&&1+\sqrt{2} && \sqrt{2}\\
 &1& &1&&1&&1&&1&& 1 & \\
  0&&0&&0&&0&&0&&0&&0\\
 \end{array}$}
    \caption{A dissected hexagon and the corresponding frieze pattern of width 3}
    \label{fig:HJInjection}
\end{figure}

 Holm and \Jor do not provide a description of which frieze patterns are in bijection with all dissected polygons. However, when narrowing the focus to $p$-angulations, they have a stronger result.

\begin{definition}
Let $p\in \Z \geq 3$. A frieze pattern is of \emph{Type $\Lambda_p$} if the quiddity row consists of
(positive) integral multiples of $\lambda_p = 2\cos(\frac{\pi}{p})$.
\end{definition}

See Figure \ref{fig:HJBijection} for an example of a frieze pattern of Type $\Lambda_4$ and the corresponding 4-angulation. 

\begin{theorem}[\cite{14}]
There is a bijection between $p$-angulated $(n+3)$-gons and frieze patterns of Type $\Lambda_p$.
\end{theorem}

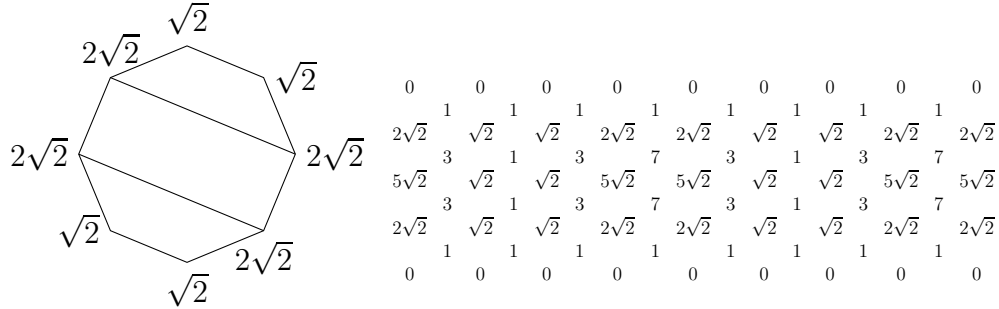
\begin{figure}
    \centering
    \raisebox{-.4\totalheight}{\begin{tikzpicture}[scale = 1.2]
 \draw (0:1.5\R) node[right]{$2\sqrt{2}$} --  (45:1.5\R) node[right]{$\sqrt{2}$} --   (90:1.5\R) node[above]{$\sqrt{2}$}  --   (135:1.5\R) node[above]{$2\sqrt{2}$} -- (135 + 45:1.5\R) node[left]{$2\sqrt{2}$} -- (180 + 45:1.5\R) node[left]{$\sqrt{2}$}-- (270 :1.5\R) node[below]{$\sqrt{2}$} -- (270 + 45:1.5\R) node[below]{$2\sqrt{2}$}--(0:1.5\R);
  \draw (0:1.5\R) -- (135:1.5\R);
  \draw (-45:1.5\R) -- (180:1.5\R);
 \end{tikzpicture}} 
 \scalebox{0.6}{$\begin{array}{cccccccccccccccccc}
0 && 0 && 0 && 0 && 0 && 0 && 0 && 0 && 0 \\
& 1 && 1 && 1 && 1 && 1 && 1 && 1 && 1 &\\
2\sqrt{2} && \sqrt{2} &&  \sqrt{2} &&2 \sqrt{2} && 2\sqrt{2} && \sqrt{2} && \sqrt{2} && 2 \sqrt{2} && 2\sqrt{2} \\
& 3 && 1 && 3 && 7 && 3 && 1 && 3 && 7 &\\
5\sqrt{2} && \sqrt{2} &&  \sqrt{2} &&5 \sqrt{2} && 5\sqrt{2} && \sqrt{2} && \sqrt{2} && 5 \sqrt{2} && 5\sqrt{2} \\
& 3 && 1 && 3 && 7 && 3 && 1 && 3 && 7 &\\
2\sqrt{2} && \sqrt{2} &&  \sqrt{2} &&2 \sqrt{2} && 2\sqrt{2} && \sqrt{2} && \sqrt{2} && 2 \sqrt{2} && 2\sqrt{2} \\
& 1 && 1 && 1 && 1 && 1 && 1 && 1 && 1 &\\
0 && 0 && 0 && 0 && 0 && 0 && 0 && 0 && 0 \\
\end{array}$}
    \caption{A 4-angulated octagon and the corresponding frieze pattern of Type $\Lambda_4$}
    \label{fig:HJBijection}
\end{figure}

\section{Frieze Patterns of Type $\Lambda_{p_1,\ldots,p_s}$ }\label{sec:FPType}

We further study frieze patterns whose quiddity sequence consists of sums of $\lambda_{p}$ for possibly multiple values of $p$. Inspired by Holm and J{\o}rgensen's notation, we introduce a term for shorthand.

\begin{definition}
Let $p_1, \ldots, p_s \in \Z_{\geq 3}$. An infinite frieze pattern is of \emph{Type $\Lambda_{p_1,\ldots,p_s}$} if each entry in its quiddity row is the sum of positive integral multiples of $\lambda_{p_1},\ldots,\lambda_{p_s}$.
\end{definition}

For example, the frieze pattern in Figure \ref{fig:HJInjection} is Type $\Lambda_{3,4}$.
Given a frieze pattern $\mathcal{F}$ of Type $\Lambda_{p_1,\ldots,p_s}$, we can write each entry in the quiddity row in the form of \[
\qm{i} = \sum_{p\in A_i}\lambda_{p}
\]
where $A_i$ is a multiset. For example, given the quiddity cycle $(3, 2+\sqrt{2}, \sqrt{2} + \sqrt{3})$, we have that $A_{3k} = \{3,3,3\}, A_{3k+1} = \{ 3,3,4 \},$ and $A_{3k+2} = \{4,6\}$ for all $k \in \Z$.

We record some useful results about frieze patterns. First, we study 1-periodic frieze patterns. These results are most conveniently stated in terms of Chebyshev polynomials. 

\begin{definition}\label{def:Chebyshev}
Let $U_k(x)$ denote the \emph{normalized Chebyshev polynomials of the second kind}, determined by initial conditions $U_{-1}(x) = 0, U_0(x) = 1$ and the recurrence,\[
U_k(x)= U_1(x)U_{k-1}(x) - U_{k-2}(x)
\]
\end{definition}

For example, $U_1(x) = x$ and $U_2(x) = x^2 - 1$. These normalized Chebyshev polynomials can be obtained by evaluating the ordinary Chebyshev polynomials of the second kind at $\frac{x}{2}$. 

\begin{lemma}\label{lem:OnePeriodicFrieze}
Let $\mathcal{F}$ be the 1-periodic frieze pattern determined by the quiddity cycle $(x)$, where $x$ is an indeterminate. Then, $\mathcal{F}$ is an infinite frieze pattern, with the $k$-th row given by $(U_k(x))$. 
\end{lemma}

This lemma is quickly proven given the following well-known result, which shows that rows of frieze patterns satisfy a recurrence similar to the Chebyshev recurrence. 

\begin{lemma}\label{lem:DivisionFreeFrieze}
Let $\mathcal{F} = \{m_{i,j}\}$ be a frieze pattern. Then, for all $m_{i,j}$ with $\vert j - i \vert \geq 2$,\[
m_{i,j} = m_{i,i+2}m_{i+1,j} - m_{i+2,j}
\]
and dually,\[
m_{i,j} = m_{j-2,j}m_{i,j-1} - m_{i,j-2}.
\]
\end{lemma}

Thus, we can study one-periodic frieze patterns by studying evaluations of Chebyshev polynomials. Lemma \ref{lem:ChebyPeriodic} easily follows from knowledge of roots of Chebyshev polynomials.

\begin{lemma}\label{lem:ChebyPeriodic}
Evaluations of Chebyshev polynomials at $\lambda_p$ are periodic, in the sense that  $U_{k+p}(\lambda_p) = -U_{k}(\lambda_p)$. In particular, $U_{p-1}(\lambda_p) = -U_{-1}(\lambda_p) = 0$. 
\end{lemma}

From Lemma \ref{lem:ChebyPeriodic}, we see for every $n \geq 0$, there is exactly one 1-periodic frieze pattern of width $n$, and it is given by quiddity cycle $(\lambda_{n+3})$. Moreover, this provides us with a necessary criterion for infinite frieze patterns of Type $\Lambda_{p_1,\ldots,p_s}$. We will see a similar statement about unrealizable frieze patterns in Proposition \ref{proposition:unrealizability1}. 

\begin{lemma}\label{lemma:1}
If a sequence of numbers contains more than $p-2$ consecutive entries $\lambda_p$,  then it cannot be a quiddity sequence of an infinite frieze pattern of positive entries.
\end{lemma}

\begin{proof}
Suppose we have $m_{0,2} = m_{1,3} = \cdots = m_{p -2,p} = \lambda_p$. . Then, by Lemma \ref{lem:DivisionFreeFrieze} and Lemma \ref{lem:ChebyPeriodic}, $m_{0,p} = 0$. Thus, either this is finite frieze pattern of width $p-3$, from an empty dissection of a $p$-gon, or it is an infinite frieze pattern with non-positive entries.
\end{proof}





\section{Dissections of Once-Punctured Discs and Annuli}\label{sec:Dissections}

Our main goal is to study frieze patterns from dissections of once-punctured discs and annuli. Accordingly, we must describe these surfaces and their dissections more carefully. 

We let $A_{n,m}$ denote an annulus with $n$ vertices on the outer boundary and $m$ vertices on the inner boundary. We require $n, m \geq 1$. We let $S_n$ denote a once-punctured disc with $n \geq 1$ vertices on the boundary. For convenience, we will refer to the boundary of $S_n$ as the outer boundary. As a convention, we will let $v_i$ denote a vertex on the outer boundary of either surface; if working with $A_{n,m}$, $w_j$ will denote a vertex on the inner boundary. 

There are two types of arcs in a dissection of $A_{n,m}$ or $S_n$.

\begin{definition}\label{def:bridgingandperipheralarcs}
Let $\D$ be a dissection of the annulus $A_{n,m}$. An arc connecting an outer vertex and an inner vertex is called \emph{bridging arc}. If $\D$ is instead a dissection of $S_n$, an arc from the boundary to the puncture will be called \emph{bridging}. For either surface, an arc connecting two vertices on the same boundary is called a \emph{peripheral arc}. 
\end{definition}

Arcs in $S_n$ from the boundary to the puncture are instead called ``central'' in \cite{1}.

As a convention, we will assume there are never peripheral arcs which are incident to the inner boundary. Since our frieze patterns are determined by the number of subgons incident to each vertex, we will consider dissections up to Dehn twist.

We also describe several special types of subgons in a dissection. In general, a dissection can have many subgons which are not one of these types. 

\begin{definition}\label{def:typesofsubgons}
\begin{itemize}
    \item A subgon of size $p$ in a dissection with $p-1$ edges on the same boundary is called an \emph{ear}.
    \item A subgon in a dissection which has at least one edge on the outer boundary and at least one vertex on the inner boundary is called an \emph{outer subgon}. 
\end{itemize}
\end{definition}

We will see in Section \ref{sec:FPFromDissection} that working with skeletal dissections is much simpler than working with general dissections. Baur, \c{C}anak\c{c}i, Jacobsen, Kulkarni, and Todorov introduced the notion of a skeletal triangulation in \cite{3}. 

\begin{definition}\label{def:SkeletalDissection}
A dissection of $A_{n,m}$ or $S_n$ is said to be \emph{skeletal} if it contains only bridging arcs.
\end{definition}

Importantly, a skeletal dissection will not contain any ears.

\subsection{Infinite Strips}\label{subsec:InfStrip}

Our matching formulas in Section \ref{sec:Match} for frieze patterns from dissections of $A_{n,m}$ or $S_n$ will require looking at the universal cover of these surfaces.

We describe first the universal cover of a dissected annulus. Let $\D$ be a dissection of $A_{n,m}$. Every dissection of $A_{n,m}$ must contain at least one bridging arc since the dissection divides the surface into polygons. Pick a bridging arc and call it $e$. Call the vertex of $e$ on the outer boundary $v_1$ and the vertex on the inner boundary $w_1$. Label the remainder of the vertices in counterclockwise order: $v_2,\ldots,v_n$ on the outer boundary and $w_2,\ldots,w_m$ on the inner boundary. Cut the annulus along $e$, producing a dissected $(n+m+2)$-gon, with two copies of the edge $e$ and two copies of the vertices $v_1,w_1$. Call this dissected polygon $F$.

In this $(n+m+2)$-gon, one set of vertices $v_1,w_1$ are neighbors to $v_2$ and $w_2$ and the other are neighbors to $v_n$ and $w_m$. Call the vertices $v_1,w_1$ which are next to $v_2,w_2$ $v_1^0$ and $w_1^0$ respectively. For all $2 \leq i \leq n$ and $2 \leq j \leq m$ relabel $v_i$ and $w_j$ as $v_i^0$ and $w_j^0$. Finally, relabel the vertices $v_1,w_1$ which are neighbors to $v_n,w_m$ as $v_1^1$ and $w_1^1$. 

\begin{center}
\begin{tabular}{cc}
\begin{tikzpicture}[scale = 1]
 \draw (0,0) circle (2\R);
\draw (0,0) circle (.5\R);
\node[scale = 0.8, left] at (120:2\R){$v_1$};
\node[scale = 0.8, left] at (180:2\R){$v_2$};
\node[scale = 0.8,left] at (240:2\R){$v_3$};
\node[scale = 0.8, below] at (90:0.5\R) {$w_1$};
\node[scale = 0.8, above] at (270:0.5\R) {$w_2$};
\node[scale = 0.8, right] at (0:0.5\R) {$w_3$};
 \coordinate(A) at (120:2\R);
 \coordinate(B) at (240:2\R);
 \coordinate(C) at (90:.5\R);
 \coordinate(D) at (270:.5\R);
\draw (B) to [out = 100, in = -100] (A);
\draw[very thick] (A)--(C);
\draw(B) -- (D);
\draw(B) to [out = 30, in = 0, looseness = 3] (C);
\node[circle, fill = black, scale = 0.3] at (180:2\R){};
\node[circle, fill = black, scale = 0.3] at (0:0.5\R){};
\node[] at (100:1.4\R){$e$};
  \end{tikzpicture}
  &
  \begin{tikzpicture}[scale = 1.2]
 \draw(1,0) -- (4,0);
 \draw(1,2) -- (4,2);
 \draw[very thick](1,0) -- (1,2);
 \node[] at (.7,1){$e$};
 \draw(3,0) -- (3,2);
 \draw(3,0) -- (4,2);
 \draw[very thick](4,0) -- (4,2);
 \node[] at (4.3,1){$e$};
 \node[circle, fill = black, scale = 0.3] at (3.5,2){};
 \node[circle, fill = black, scale = 0.3] at (2,0){};
 \coordinate(A) at (1,0);
 \coordinate(B) at (3,0);
 \draw (A) to [out = 60, in = 120, looseness = 2] (B);
 \node[below, scale = 0.8] at (1,0) {$v_1^0$};
 \node[below, scale = 0.8] at (2,0) {$v_2^0$};
 \node[below, scale = 0.8] at (3,0) {$v_3^0$};
 \node[below, scale = 0.8] at (4,0) {$v_1^1$};
 \node[above, scale = 0.8] at (1,2) {$w_1^0$};
  \node[above, scale = 0.8] at (3,2) {$w_2^0$};
 \node[above, scale = 0.8] at (3.5,2) {$w_3^0$};
 \node[above, scale = 0.8] at (4,2) {$w_1^1$};
 \node[] at (-1,1) {$\rightarrow$};
  \end{tikzpicture}
  \end{tabular}
  \end{center}

We draw $F$ as a rectangle, with the boundary edges labeled $e$ vertical and all other edges horizontal so that all vertices $v_i^0$ and $v_1^1$ are on the bottom and all vertices $w_j^0$ and $w_1^1$ are on the top. We create a periodic dissection of the infinite strip by translating $F$ infinitely many times horizontally to both the right and the left. We glue these copies of $F$ along the edge $e$. In the copy of $F$ which is $k$ shifts to the right, we label the vertices on the bottom $v_1^k,\ldots,v_n^k,v_1^{k+1}$ and similarly on the top. We label the vertices in the copy of $F$ $k$ shifts to the left with $-k$.

There is a covering map $\rho$ from this dissected infinite strip to $\D$; for all $k \in \mathbb{Z}$, $\rho(v_i^k) = v_i$ and $\rho(w_j^k) = w_j$. This will result in a map on edges: if $g_k$ is an edge from $v_i^k$ to $w_j^{k'}$ (note $k'$ could be $k$ or $k+1$), then $\rho(g_k) = g$. We also extend this map to the subgons in the dissection.

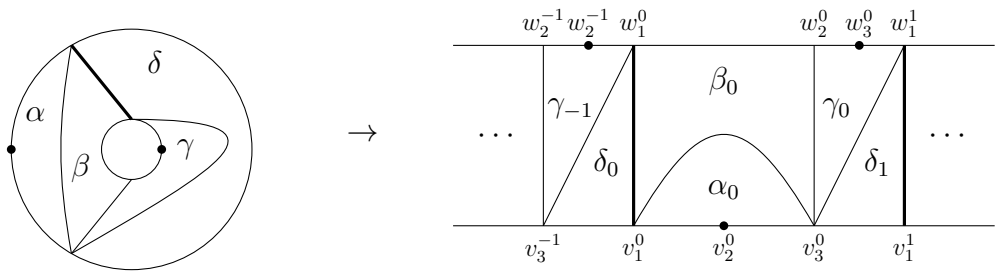
\begin{center}
\begin{tabular}{cc}
\begin{tikzpicture}[scale = 1]
 \draw (0,0) circle (2\R);
\draw (0,0) circle (.5\R);
 \coordinate(A) at (120:2\R);
 \coordinate(B) at (240:2\R);
 \coordinate(C) at (90:.5\R);
 \coordinate(D) at (270:.5\R);
\draw (B) to [out = 100, in = -100] (A);
\draw[very thick] (A)--(C);
\draw(B) -- (D);
\draw(B) to [out = 30, in = 0, looseness = 3] (C);
\node[circle, fill = black, scale = 0.3] at (180:2\R){};
\node[circle, fill = black, scale = 0.3] at (0:0.5\R){};
\node[] at (160:1.7\R){$\alpha$};
\node[] at (200:.9\R){$\beta$};
\node[] at (0:.9\R){$\gamma$};
\node[] at (75:1.5\R){$\delta$};
  \end{tikzpicture}
  &
  \begin{tikzpicture}[scale = 1.2]
 \draw(-1,0) -- (5,0);
 \draw(-1,2) -- (5,2);
 \draw (0,0) -- (1,2);
 \draw[very thick](1,0) -- (1,2);
 \draw(3,0) -- (3,2);
 \draw(3,0) -- (4,2);
 \draw[very thick](4,0) -- (4,2);
 \node[circle, fill = black, scale = 0.3] at (3.5,2){};
 \node[circle, fill = black, scale = 0.3] at (0.5,2){};
 \node[circle, fill = black, scale = 0.3] at (2,0){};
 \draw(0,0) -- (0,2);
 \coordinate(A) at (1,0);
 \coordinate(B) at (3,0);
 \draw (A) to [out = 60, in = 120, looseness = 2] (B);
 \node[below, scale = 0.8] at (0,0) {$v_3^{-1}$};
 \node[below, scale = 0.8] at (1,0) {$v_1^0$};
 \node[below, scale = 0.8] at (2,0) {$v_2^0$};
 \node[below, scale = 0.8] at (3,0) {$v_3^0$};
 \node[below, scale = 0.8] at (4,0) {$v_1^1$};
  \node[above, scale = 0.8] at (1,2) {$w_1^0$};
  \node[above, scale = 0.8] at (3,2) {$w_2^0$};
 \node[above, scale = 0.8] at (3.5,2) {$w_3^0$};
 \node[above, scale = 0.8] at (4,2) {$w_1^1$};
   \node[above, scale = 0.8] at (0.5,2) {$w_2^{-1}$};
\node[above, scale = 0.8] at (0,2) {$w_2^{-1}$};
 \node[] at (2,0.4){$\alpha_0$};
 \node[] at (2,1.6){$\beta_0$};
 \node[] at (0.7,0.7) {$\delta_0$};
 \node[] at (3.7,0.7) {$\delta_1$};
 \node[] at (3.25,1.3) {$\gamma_0$};
 \node[] at (0.3,1.3) {$\gamma_{-1}$};
 \node[] at (-0.5,1) {$\cdots$};
 \node[] at (4.5,1) {$\cdots$};
 \node[] at (-2,1) {$\rightarrow$};
  \end{tikzpicture}
  \end{tabular}
  \end{center}

Next, suppose that $\D$ is a dissection of $S_n$. Label the vertices $v_1,\ldots,v_n$ in counterclockwise order. Draw an infinite strip with vertices $v_i^k$ for all $1 \leq i \leq n$ and all $k \in \mathbb{Z}$ on the lower boundary, in order $\ldots,v_1^i,\ldots,v_n^i,v_1^{i+1},\ldots$ and no vertices on the upper boundary.

Following Lemma 3.6 in \cite{1}, from $\D$ we will construct a dissection of the infinite strip using \emph{asymptotic arcs}. On the infinite strip, an asymptotic arc has one endpoint at a marked point then goes infinitely far to the right or left. Our asymptotic arcs will travel to the right, and we consider that they meet at $+\infty$. 

 For each $i$ such that there is an arc from $v_i$ to the puncture in $\D$ and for all   $k \in \mathbb{Z}$, draw an asymptotic arc with one endpoint at $v_i^k$ which travels infinitely far to the right. Note that we can draw all of these asymptotic arcs so that they do not cross. 

In our dissection of $S_n$, suppose $v_i$ and $v_j$ are such that each are connected to the puncture by an edge in $\D$ and, sweeping counterclockwise from $v_i$ to $v_j$, there are no other arcs to the puncture. Then, these two arcs to the puncture cut out a dissected polygon from $S_n$. Let $k' = k$ if $j > i$ and $k' = k+1$ if $j \leq i$. In the infinite strip, for all $k \in \mathbb{Z}$ we have a polygon of the same size between $v_i^k$ and $v_j^{k'}$, using $+\infty$ as one of the endpoints. We dissect each of these polygons in the infinite strip so that they have the same dissection as the polygon in $S_n$. That is, if there is a peripheral arc between $v_a$ and $v_b$, where $v_a$ and $v_b$ are weakly between $v_i$ and $v_j$ traveling counterclockwise, we connect $v_a^{k_1}$ and $v_b^{k_2}$ between $v_i^k$ and $v_j^{k'}$ where $k_1$ and $k_2$ are $k$ or $k+1$.
If there are at least two bridging arcs in $\D$, we do this between each consecutive pair of bridging arcs. If there is exactly one bridging arc, we do this process once using $v_i = v_j$. 

We define a map $\rho$ from this asymptotic dissection of the infinite strip to $\D$ such that $\rho(v_i^k) = v_i$. If, for some $i$, each $v_i^k$ is incident to an asymptotic arc, $e_k$, then $\rho(e_k) = e$ where $e$ is an arc from $v_i$ to the puncture. Let $g$ be an arc between $v_a^k$ and $v_b^{k'}$, where either $a < b$ and $k' = k$ or $a \geq b$ and $k' = k+1$. In particular, $v_a^k$ appears to the left of $v_b^{k'}$ with our convention of drawing the infinite strip. Then, $\rho(g)$ is the arc between $v_a$ and $v_b$ which follows the boundary in the counterclockwise direction from $v_a$ to $v_b$. The image of $\rho$ on vertices and edges will extend to a map on subgons. A subgon in the infinite strip with a vertex at $+\infty$ will be mapped to a subgon in $S_n$ with a vertex at the puncture.

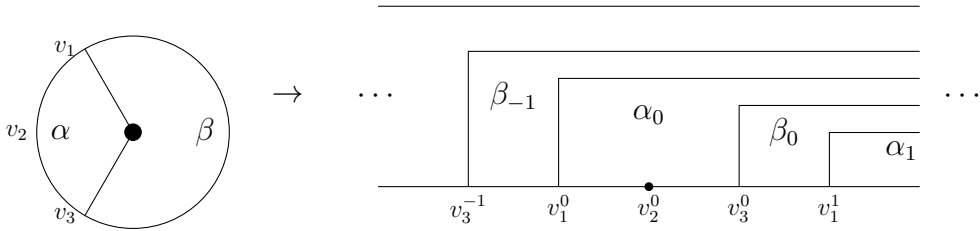
\begin{center}
\begin{tabular}{cc}
\begin{tikzpicture}[scale = 0.8]
 \draw (0,0) circle (2\R);
\node[circle, fill = black, scale = 0.6] at (180:0\R){};
\node[scale = 0.8, left] at (120:2\R){$v_1$};
\node[scale = 0.8, left] at (180:2\R){$v_2$};
\node[scale = 0.8,left] at (240:2\R){$v_3$};
\draw (120:2\R) to (0:0);
\draw (240:2\R) to (0:0);
\node[] at (180:1.5\R){$\alpha$};
\node[] at (0:1.5\R){$\beta$};
  \end{tikzpicture}
  &
  \begin{tikzpicture}[scale = 1.2]
 \draw(-1,0) -- (5,0);
 \draw(-1,2) -- (5,2);
 \draw (0,0) -- (0,1.5) -- (5,1.5);
 \draw(1,0) -- (1,1.2) -- (5,1.2);
 \draw(3,0) -- (3,0.9) -- (5,0.9);
 \draw(4,0) -- (4,0.6) -- (5,0.6);
 \node[circle, fill = black, scale = 0.3] at (2,0){};
 \node[below, scale = 0.8] at (0,0) {$v_3^{-1}$};
 \node[below, scale = 0.8] at (1,0) {$v_1^0$};
 \node[below, scale = 0.8] at (2,0) {$v_2^0$};
 \node[below, scale = 0.8] at (3,0) {$v_3^0$};
 \node[below, scale = 0.8] at (4,0) {$v_1^1$};
 \node[] at (2,0.8){$\alpha_0$};
 \node[] at (0.5,1){$\beta_{-1}$};
 \node[] at (3.5,0.6){$\beta_0$};
  \node[] at (4.8,0.4){$\alpha_1$};
 \node[] at (-1,1) {$\cdots$};
 \node[] at (5.5,1) {$\cdots$};
  \node[] at (-2,1) {$\rightarrow$};
  \end{tikzpicture}
  \end{tabular}
  
  \end{center}

\section{Frieze Patterns from dissected Once-Punctured Discs and Annuli}\label{sec:FPFromDissection}

For the remainder of this paper, we will study the relationship between periodic infinite frieze patterns and dissections of annuli and once-punctured discs. From here on, all infinite frieze patterns will be assumed to be periodic. One could study non-periodic infinite frieze patterns of Type $\Lambda_{p_1,\ldots,p_s}$ by studying non-periodic dissections of the infinite strip.

 We use Holm and J{\o}rgensen's algorithm to form a quiddity cycle, and thus a frieze pattern, from a dissection of one of these surfaces. To be explicit, given a dissection $\D$ of $S_n$ or $A_{n,m}$ with vertices $v_1,\ldots,v_n$ on the outer boundary, let $\polyset{v_i}$ be the set of subgons incident to $v_i$. As in the triangulation case, we will determine this set by drawing a small half circle around each vertex and looking at the subgons this half circle intersects. This can result in counting a subgon twice, as in the instance of a self-folded triangle, or more generally, a self-folded subgon.

For each subgon $\mathcal{P}$ in $\D$, we will use $|\mathcal{P}|$ denote the number of edges of this subgon. We set $A_i = \{ \vert \mathcal{P} \vert : \mathcal{P} \in \polyset{v_i}\}$, as was defined for the polygon case in Section \ref{subsec:HJ}.
 Then, we form a quiddity cycle $(m_{0,2},\ldots,m_{n-1,n+1})$ by setting $m_{i-1,i+1} = \sum_{p \in A_i} \lambda_p$.

\begin{example}\label{ex:FPfromDissection}
For example, in the dissection of $A_{3,3}$ below, $A_1 = \{3,3,4\}$, $A_2 = \{3\}$, and $A_3 = \{3,3,4,4\}$.

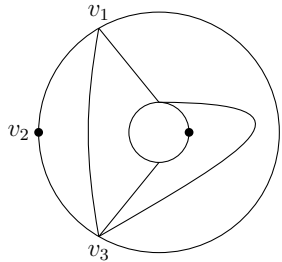
\begin{center}
\begin{tikzpicture}[scale = 1]
 \draw (0,0) circle (2\R);
\draw (0,0) circle (.5\R);
 \coordinate(A) at (120:2\R);
 \node[above, scale = 0.8] at (120:2\R){$v_1$};
 \node[left, scale = 0.8] at (180:2\R){$v_2$};
  \node[below, scale = 0.8] at (240:2\R){$v_3$};
 \coordinate(B) at (240:2\R);
 \coordinate(C) at (90:.5\R);
 \coordinate(D) at (270:.5\R);
\draw (B) to [out = 100, in = -100] (A);
\draw(A)--(C);
\draw(B) -- (D);
\draw(B) to [out = 30, in = 0, looseness = 3] (C);
\node[circle, fill = black, scale = 0.3] at (180:2\R){};
\node[circle, fill = black, scale = 0.3] at (0:0.5\R){};
\node[] at (160:1.7\R){};
\node[] at (200:\R){};
\node[] at (0:.9\R){};
\node[] at (75:1.5\R){};
  \end{tikzpicture}
  \end{center}
  
 We give the first few rows of the frieze pattern from this dissection. This is the frieze pattern with quiddity cycle $(2 + \sqrt{2},1,2 + 2\sqrt{2})$.
 
 \[
 \begin{array}{cccccccccccccccccccc}
  0&&0&&0&&0&&0&&\\
 &1& &1&&1&&1&\\
 1&&2+\sqrt{2}&&2+2\sqrt{2}&&1&&2+\sqrt{2}&&\\
&1+\sqrt{2}& &7 + 6\sqrt{2}&&1+2\sqrt{2}&&1+\sqrt{2}&\\
  4+3\sqrt{2}&&5+4\sqrt{2}&&5+5\sqrt{2}&&4+3\sqrt{2}&&5+4\sqrt{2}&&\\
&21+7\sqrt{2}& &4 + 3\sqrt{2}&&13+9\sqrt{2}&&21+7\sqrt{2}&\\
  & &&&\vdots&&&&&\\
 \end{array}
 \]
\end{example}

Our goal in this section is determining when an infinite frieze pattern of Type $\Lambda_{p_1,\ldots,p_s}$  comes from a dissection of an annulus or once-punctured disc.

\begin{definition}
An infinite frieze pattern of period $n$ is called $\emph{realizable}$ if there exists a dissection $\D$ on an annulus $A_{n,m}$ or a once-punctured disc $S_n$ such that each entry $m_{i-1,i+1}$ in the quiddity sequence is given by
\[m_{i-1,i+1} = \sum_{p\in A_i}\lambda_p\]

where $A_i$ is the multiset of sizes of subgons incident to  $v_i$. If a frieze pattern is not realizable, we call it \emph{unrealizable}.

Similarly, we say a quiddity cycle $(m_{0,2},\ldots,m_{n-1,n+1})$ is \emph{realizable} if the frieze pattern it generates is realizable, and otherwise we call it \emph{unrealizable}.
\end{definition}

\begin{remark}
We will also call any frieze pattern with a negative entry or a zero outside the boundary row(s) of zeroes unrealizable. See Section \ref{sec:Positivity} for more discussion on when the entries of frieze pattern may not be positive. 
\end{remark}

We begin by determining some frieze patterns which are unrealizable.

\begin{proposition}[Realizability test] \label{proposition:unrealizability1}
Let $\{m_{i-1,i+1}\}$ be such that $m_{i-1,i+1} = \sum_{p\in A_i} \lambda_p$ for all $i$. Then, the frieze pattern generated by $\{m_{i-1,i+1}\}$ is unrealizable if at least one of the following is true.
\begin{enumerate}
    \item There exists $i$ such that $A_i\cap A_{i+1} = \emptyset$.
    \item For some $p \in \mathbb{Z}_{\geq 3}$, there exists more than $p-2$ consecutive indices $i$ such that $A_i = \{p\}$, and there also exists some $j$ such that $A_j \neq \{p\}$.
\end{enumerate}
\end{proposition}

\begin{proof}
In a dissection, each pair of adjacent vertices are incident to a common subgon. Therefore, a quiddity sequence from a dissection must have $A_i \cap A_{i+1} \neq \emptyset$ for all $i$.
Similarly, it is impossible for the quiddity sequence from a realizable frieze pattern to have more than $p-2$ consecutive entries $\lambda_p$ unless the sequence comes from an empty dissection of a $p$-gon. 
\end{proof}

We say that a frieze pattern which does not satisfy either condition of Proposition \ref{proposition:unrealizability1} ``passes the realizability test''.
Recall that by Lemma \ref{lemma:1}, if a quiddity sequence satisfies 2 above, the frieze pattern it generates contains negative entries.

\subsection{Realizability of Skeletal Frieze Patterns}\label{subsec:RealizabilitySkeletal}

In order to describe which frieze patterns are realizable by dissections of $A_{n,m}$ or $S_n$, it will be useful to first restrict to a certain class of frieze patterns.

\begin{definition}
If $\F$ is an infinite frieze pattern with quiddity sequence $\{m_{i-1,i+1}\}$ such that the length any string of consecutive entries $\lambda_p$ in the quiddity sequence is smaller than $p-2$, we call $\F$ a \emph{skeletal frieze pattern}.
\end{definition}

Recall by Definition \ref{def:SkeletalDissection} that a dissection is skeletal if it contains no peripheral  arcs. The following statement follows from the fact that if a dissection has any peripheral arcs, it will have ears. 

\begin{lemma}\label{lem:SkeletalAgree}
A skeletal dissection will always produce a skeletal frieze pattern.
\end{lemma}

We can directly characterize realizability of skeletal frieze patterns. 

\begin{proposition}[Skeletal Realizability Criteria]\label{prop:realizability}
Let $\mathcal{F} = \{m_{i,j}\}$ be a skeletal frieze pattern of Type $\Lambda_{p_1,\ldots,p_s}$ with quiddity cycle $(m_{0,2},\ldots,m_{n-1,n+1})$. Let $A_i$ be the multiset of integers, all at least 3, such that $m_{i-1,i+1} = \sum_{p \in A_i} \lambda_p$. Then, $\mathcal{F}$ is realizable by a dissection of $A_{n,m}$ or $S_n$ if and only if 
\begin{enumerate}
    \item[i)] the sequence $\{m_{i-1,i+1}\}$ does not satisfy either condition of Proposition \ref{proposition:unrealizability1}, and 
    \item[ii)] we can choose $p_{i,i+1} \in A_i \cap A_{i+1}$ such that, if $\vert A_i \vert > 1$ and $p_{i-1,i} = p_{i,i+1}$, then $\{p_{i-1,i},p_{i,i+1}\} \subseteq A_i$ as multisets. 
\end{enumerate}
\end{proposition}

\begin{proof}
We will first prove that if $\F$ is a skeletal frieze pattern realizable by dissection $\D$, then the quiddity cycle $(m_{0,2},\ldots,m_{n-1,n+1})$ of $\F$ satisfies i and ii. 

Since $\F$ is skeletal, the quiddity sequence will not satisfy part 2 of Proposition \ref{proposition:unrealizability1}.  In any dissection, each pair of consecutive vertices $v_i, v_{i+1}$ will be vertices of a common subgon, $\mathcal{P}_{i,i+1}$. Hence, $\vert \mathcal{P}_{i,i+1}\vert \in A_i \cap A_{i+1}$, and the quiddity cycle does not satisfy 1 in Proposition \ref{proposition:unrealizability1}. Moreover, if $\vert \mathcal{P}_{i-1,i}\vert = \vert \mathcal{P}_{i,i+1}\vert$ and $\vert A_i \vert > 2$, then $\vert \mathcal{P}_{i-1,i}\vert$ will appear with multiplicity at least 2 in $A_i$ since $v_i$ is adjacent to at least two distinct polygons of size $\vert \mathcal{P}_{i-1,i}\vert$. This shows item ii.

For the other direction, given a skeletal frieze pattern $\F$  generated by a quiddity cycle $(m_{0,2},\ldots,m_{n-1,n+1})$ which satisfies the conditions in the statement of the Proposition, we will construct a dissection which is the realization of $\F$. We start by determining the shape of the outer subgons.

Choose a sequence $p_{i,i+1}$, for $1 \leq i \leq n$ which satisfies condition ii. For some $i$, there may be multiple choices of $p_{i,i+1}$; choose one arbitrarily. 
Let $\{i_1,\ldots,i_\ell\} \subseteq \{1,\ldots,n\}$ be the set of indicies such that $\vert A_{i_j}\vert > 1$, with the ordering $i_1 < i_2 < \cdots < i_{\ell - 1} < i_\ell $. Since for any indices $i_j < s < i_{j+1}$, $\vert A_s \vert = 1$, we have that $p_{i_j, i_j+1} = p_{i_j+1,i_j+2} = \cdots = p_{i_{j+1}-1,i_{j+1}}$. We define this number to be $p_{i_j,i_{j+1}}$. Since we work modulo $n$, we similarly define $p_{i_\ell,i_1}.$

First, if all $i_j$ are such that $\vert A_{i_j} \vert = 2$ and for all consecutive pairs $(i_j,i_{j+1})$, modulo $n$, $p_{i_j,i_{j+1}} -2  = i_{j+1} - i_j$, then this frieze pattern will correspond to a dissection of $S_n$.  We label the vertices of $S_n$ $v_1,\ldots,v_n$ in counterclockwise order, and we draw an arc from each $v_{i_j}$ to the puncture.  This dissection of $S_n$ will correspond to the original frieze pattern $\F$. 

Suppose then that there exists $i_j$ such that $\vert A_{i_j} \vert > 2$ or a pair $(i_j, i_{j+1})$ such that $p_{i_j,i_{j+1}} -2  > i_{j+1} - i_j$. Then, $\F$ will correspond to a dissection of $A_{n,m}$, where $m$ will be determined in the following construction. We label the vertices on the outer boundary $v_1,\ldots, v_n$ in counterclockwise order. For each $i_j$, we draw two vertices on the inner boundary, $w_{i_j}^-$ and  $w_{i_j}^+$ where, in counterclockwise order $w_{i_j}^+$ follows $w_{i_j}^-$ and $w_{i_{j+1}}^{\pm}$ both follow $w_{i_{j}}^\pm$. If $\vert A_{i_j} \vert = 2$, we identify $w_{i_j}^-$ and  $w_{i_j}^+$. If $p_{i_j,i_{j+1}} -2  = i_{j+1} - i_j$, then we identify $w_{i_j}^+ $ and $w_{i_{j+1}}^-$. If $p_{i_j,i_{j+1}} -2  > i_{j+1} - i_j$,  we include $p_{i_j,i_{j+1}} -3 -( i_{j+1} - i_j)$ vertices on the inner boundary between $w_{i_j}^+ $ and $w_{i_{j+1}}^-$. Once we have made these identifications and drawn these vertices, we draw arcs from each $v_{i_j}$ to $w_{i_j}^-$ and $w_{i_j}^+$. If $\vert A_{i_j} \vert = 2$, then since $w_{i_j}^-$ and $w_{i_j}^+$ are identified we only draw one arc from $v_{i_j}$.   

At this point, we have established all outer subgons; we add the remaining subgons of the dissection. We do not need to add anything to vertices $v_{i_j}$ such that $\vert A_{i_j} \vert \leq 3$. Let $v_{i_j}$ be a vertex such that $\vert A_{i_j} \vert > 3$. We draw $\vert A_{i_j} \vert - 3$ additional arcs from $v_{i_j}$ to distinct vertices on the inner boundary, between $w_{i_j}^-$ and $w_{i_j}^+$. We add additional vertices between the vertices on the inner boundary connected to $v_{i_j}$ so that $A_{i_j}$ records the sizes of polygons incident to $v_{i_j}$.  
\end{proof}

We used the following fact in the proof of Proposition \ref{prop:realizability}. We record it for use in the Realizability Algorithm.

\begin{lemma}\label{lem:SkeletalSn}
A skeletal frieze pattern $\F$ with quiddity cycle $(m_{0,2},\ldots,m_{n-1,n+1})$ which satisfies Proposition \ref{prop:realizability} is realizable by a dissection of $S_n$ if and only if all $A_i$ are such that $\vert A_i \vert \in \{1,2\}$, and for each $p \in \mathbb{Z}$, any string of entries, $m_{i-1,i+1} = \cdots = m_{i+\ell - 1, i+\ell+1} = \lambda_p$ in the quiddity cycle has length $p-3$. 
\end{lemma}

As an example of Proposition \ref{prop:realizability}, consider the quiddity cycle $(2+\sqrt{2}, \sqrt{2} + \sqrt{3}, \sqrt{3}, 1 + \sqrt{3})$. This quiddity cycle satisfies all conditions of the Proposition. Here, $A_1 = \{3,3,4\}, A_2 = \{4,6\}, A_3 = \{6\},$ and $A_4 = \{3,4\}$. Our subsequence of indices where $A_i$ is larger than a singleton is $1,2,4$. We only have one choice for the sequence of $p_{i,i+1}$: $p_{1,2} = 4, p_{2,3} = p_{3,4} = 6, p_{4,1} = 3$. Thus, we set $p_{2,4} = 6$. Since we have $\vert A_1 \vert > 2$, $p_{1,2} - 2 > 2 - 1$, and $p_{2,4} - 2 > 4 - 2$, this frieze pattern will come from a dissected annulus. We place vertices $v_1,v_2,v_3,v_4$ in counterclockwise order on the outer boundary. 

We start with adding 6 vertices, $w_{1,2}^\pm, w_{2,4}^\pm, w_{4,1}^\pm$ to the inner boundary, with counterclockwise order given by $w_{1,2}^-,w_{1,2}^+,w_{2,4}^-,\ldots$. Since $p_{4,1} -2 = 5-4$, we identify $w_{1,2}^-$ and $w_{4,1}^+$. Since $\vert A_2\vert =2$, we identify $w_{2,4}^-$ and $w_{2,4}^+$; similarly, we identify $w_{4,1}^-$ and $w_{4,1}^+$. Finally, we add $p_{2,4} - 3 - (4-2) = 1$ vertices on the inner boundary between $w_{2,4}^+$ and $w_{4,1}^-$.  Since $p_{1,2} - 3 - (2 - 1) = 0$, we do not add any vertice between $w_{1,2}^+$ and $w_{2,4}^-$. Now, we connect each $v_{i_j}$ to $w_{i_j}^\pm$, with these identifications in place.

At this point, we have established all outer subgons of the dissection. From here, we check whether we need additional arcs from any $v_{i_j}$ so that each $A_i$ correctly records all incident subgons at $v_i$. However, in this case we do not need to add any additional arcs or vertices on the inner boundary. We draw this dissection below. 

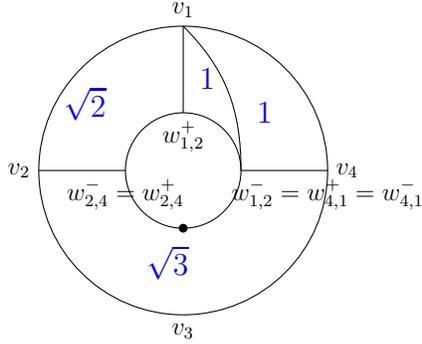
\begin{center}
\begin{tikzpicture}[scale = 1.2]
 \draw (0,0) circle (2\R);
\draw (0,0) circle (0.8\R);
 \coordinate(A) at (120:2\R);
 \node[above, scale = 0.8] at (90:2\R){$v_1$};
 \node[left, scale = 0.8] at (180:2\R){$v_2$};
 \node[below, scale = 0.8] at (270:2\R){$v_3$};
 \node[right,  scale = 0.8] at (0:2\R){$v_4$};
  \node[below, scale = 0.8] at (90:0.8\R){$w_{1,2}^+$};
 \node[below, scale = 0.8] at (180:0.8\R){$w_{2,4}^- = w_{2,4}^+$};
 \node[below, scale = 0.8] at (270:2\R){};
 \node[below, scale = 0.8] at (0:2\R){$w_{1,2}^- = w_{4,1}^+ = w_{4,1}^-$};
 \node[circle, fill = black, scale = 0.3] at (270:0.8\R){};
 \draw (90:2\R) -- (90:0.8\R);
 \draw (180:0.8\R) -- (180:2\R);
 \draw(0:0.8\R) -- (0:2\R);
 \coordinate(A) at (90:2\R);
 \coordinate(B) at (0:0.8\R);
\draw (A) to [out = -45, in = 90] (B);
\node[left, blue] at (30:1.6\R){$1$}; 
\node[left, blue] at (65:1.4\R){$1$}; 
\node[left, blue] at (135:1.3\R){$\sqrt{2}$}; 
\node[left, blue] at (280:1.3\R){$\sqrt{3}$}; 
  \end{tikzpicture}
  \end{center}

Given a frieze pattern which satisfies the conditions of Proposition \ref{prop:realizability}, the corresponding dissection is not unique. First, since the quiddity sequence is $n$-periodic for some $n$, one could take a longer quiddity cycle $(m_{0,2},\ldots,m_{kn-1,kn+1})$ for any $k \in \Z_{>0}$. We call the resulting dissection the \emph{kth power} of the dissection from $(m_{0,2},\ldots,m_{n-1,n+1})$, as is described in Definition \ref{def:kthpower}.

There is a more substantial way that these dissections can be non-unique, which only appears in some dissections with mixed sizes of subgons. Given a quiddity cycle $(m_{0,2},\ldots,m_{n-1,n+1})$ there might be several distinct sequences $p_{1,2}, p_{2,3},...,p_{n,n+1}$ that satisfy the conditions of Proposition \ref{prop:realizability}. For instance, see the two dissections with the same quiddity cycle in Example \ref{ex:NonUniqueDissection}. 

\begin{example}\label{ex:NonUniqueDissection}
We show two distinct dissections of annuli, each of which corresponds to the quiddity cycle $(1+2\sqrt{2},2+2\sqrt{2})$.

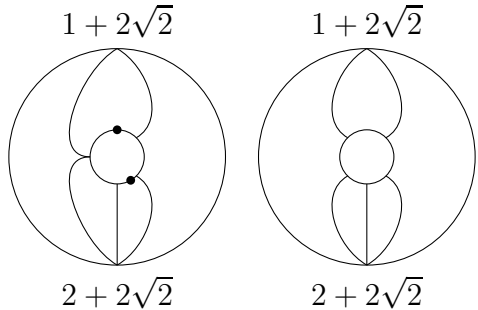
\begin{center}
\begin{tabular}{c c}
\begin{tikzpicture}[scale = 0.9]
\draw (0,0) circle (2\R);
\draw (0,0) circle (0.5\R);
\coordinate(A) at (90:2\R);
\coordinate(B) at (-90:2\R);
\coordinate(a) at (180:0.5\R);
\coordinate(b) at (45:0.5\R);
\coordinate(c) at (-45:0.5\R);
\coordinate(d) at (-90:0.5\R);
\node[above, black] at (90:2\R){$1+2\sqrt{2}$}; 
\node[below, black] at (-90:2\R){$2+2\sqrt{2}$}; 
\node[circle, fill = black, scale = 0.3] at (90:0.5\R){};
\node[circle, fill = black, scale = 0.3] at (-60:0.5\R){};
\draw(A) to [out = -150, in = 180](a);
\draw(B) to [out = 150, in = 180](a);
\draw(A) to [out = -30, in = 30](b);
\draw(B) to [out = 30, in = -30](c);
\draw(B)--(d);
\end{tikzpicture}

&

\begin{tikzpicture}[scale = 0.9]
\draw (0,0) circle (2\R);
\draw (0,0) circle (0.5\R);
\coordinate(A) at (90:2\R);
\coordinate(B) at (-90:2\R);
\coordinate(a) at (135:0.5\R);
\coordinate(d) at (-135:0.5\R);
\coordinate(b) at (45:0.5\R);
\coordinate(c) at (-45:0.5\R);
\coordinate(e) at (-90:0.5\R);
\node[above, black] at (90:2\R){$1+2\sqrt{2}$}; 
\node[below, black] at (-90:2\R){$2+2\sqrt{2}$}; 

\draw(A) to [out = -150, in = 150](a);
\draw(B) to [out = 150, in = -150](d);
\draw(A) to [out = -30, in = 30](b);
\draw(B) to [out = 30, in = -30](c);
\draw(B)--(e);
\end{tikzpicture} 
\\
     
\end{tabular}
\end{center}
\end{example}

\subsection{Quotient Dissections}\label{subsec:Dumpling}

There are skeletal frieze patterns that pass the realizability test but do not satisfy the requirements in Proposition \ref{prop:realizability}. One example is the frieze pattern generated by the quiddity cycle $(1+\phi,1+\sqrt{2},1+\phi,\sqrt{2} + \phi)$; recall $\lambda_3 = 1, \lambda_4 = \sqrt{2},$ and $\lambda_5 = \phi$. We see that $A_1 = \{3,5\}, A_2 = \{3,4\}$, and $A_3 = \{3,4\}$. Thus, $A_1 \cap A_2 = A_2 \cap A_3 = \{3\}$, but 3 only appears with multiplicity 1 in $A_2$. We introduce another type of geometric construction for such frieze patterns. 

\begin{definition}\label{def:QuotientDissection}

Given $\D$, a dissection of $A_{n,m}$, we can form a \emph{quotient dissection}, $\overline{\D}$, by iteratively identifying pairs of subgons of the same size which share a vertex on the outer boundary. For our purposes, we forbid identifying two subgons which share an edge or identifying subgons which have all vertices on outer boundary. We may end up identifying many subgons of the same size as one depending on which pairs we identify. 
\end{definition}

Given a quotient dissection of $A_{n,m}$, we construct a quiddity cycle in the same way as for a ordinary dissection. That is, for all vertices $v_i$, let $A_{i}$ be the (multi)set of sizes of subgons incident to $v_i$, using the identifications. Then, we form a quiddity sequence $(m_{0,2},\ldots, m_{n-1,n+1})$ by setting $m_{i-1,i+1} = \sum_{p \in A_i} \lambda_p$.

\begin{example}\label{ex:quotientdissection}
 
 Below, we identify the 4-gons labeled $a_1$ and $a_2$. This changes the quiddity cycle from $(2 +2\sqrt{2},3\sqrt{2})$ to $(2 + \sqrt{2}, 2\sqrt{2}$)
 
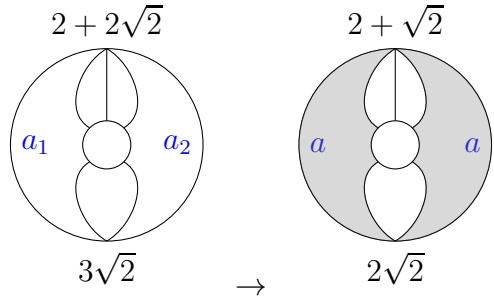
\begin{center}
\begin{tabular}{ccc}
\begin{tikzpicture}[scale = 0.8]
\draw (0,0) circle (2\R);
\draw (0,0) circle (0.5\R);
\coordinate(A) at (90:2\R);
\coordinate(B) at (-90:2\R);
\coordinate(a) at (-90:0.5\R);
\coordinate(b) at (90:0.5\R);
\coordinate(c) at (45:0.5\R);
\coordinate(d) at (-45:0.5\R);
\coordinate(e) at (135:0.5\R);
\coordinate(f) at (-135:0.5\R);
\node[above] at (90:2\R){$2 + 2\sqrt{2}$};
\node[below] at (-90:2\R){$3\sqrt{2}$};
\node[right, blue] at (180:2\R){$a_1$};
\node[left, blue] at (0:2\R){$a_2$};
\draw(90:2\R) -- (90:0.5\R);
\draw(B) to [out = 150, in = -150](f);
\draw(B) to [out = 30, in = -30](d);
\draw(A) to [out = -30, in = 30](c);
\draw(A) to [out = -150, in = 150](e);

\end{tikzpicture}
&$\rightarrow$&
\begin{tikzpicture}[scale = 0.8]
\draw (0,0) circle (2\R);
\draw (0,0) circle (0.5\R);
\coordinate(A) at (90:2\R);
\coordinate(B) at (-90:2\R);
\coordinate(a) at (-90:0.5\R);
\coordinate(b) at (90:0.5\R);
\coordinate(c) at (45:0.5\R);
\coordinate(d) at (-45:0.5\R);
\coordinate(e) at (135:0.5\R);
\coordinate(f) at (-135:0.5\R);
\node[above] at (90:2\R){$2 + \sqrt{2}$};
\node[below] at (-90:2\R){$2\sqrt{2}$};
\node[right, blue] at (180:2\R){$a$};
\node[left, blue] at (0:2\R){$a$};
\draw(90:2\R) -- (90:0.5\R);
\draw(B) to [out = 150, in = -150](f);
\draw(B) to [out = 30, in = -30](d);
\draw(A) to [out = -30, in = 30](c);
\draw(A) to [out = -150, in = 150](e);
\filldraw[fill = gray, opacity = 0.3]    (A) arc [radius=2\R, start angle=90, delta angle=180]-- (B) to [out = 150, in = -150] (f) -- (f) to [out = 135, in = -135] (e)-- (e) to [out = 150, in = -150](A) ;
\filldraw[fill = gray, opacity = 0.3]  (A) arc [radius=2\R, start angle=90, delta angle=-180]-- (B) to [out = 30, in = -30] (d) -- (d) to [out = 45, in = -45] (c)-- (c) to [out = 30, in = -30](A);

\end{tikzpicture}
\end{tabular}
\end{center}
\end{example}

\begin{remark}
\begin{enumerate}
\item We do not allow identifying subgons which share an edge to avoid constructing quotient dissections which realize quiddity cycles that do not pass the realizability test. In particular, we could realize a quiddity cycle with more than $p-2$ consecutive entries $\lambda_p$ by gluing many $p$-gons together along shared edges. See Figure \ref{fig:QuotientAlongEdge}.
\item We also do not allow identifying subgons with all vertices on one boundary to avoid constructing realizations of unrealizable frieze patterns. As an example, see Figure \ref{fig:QuotientByEars} where we identify two ears to produce the quiddity cycle  $(2\sqrt{2},\sqrt{2},1,1+2\sqrt{2})$. This quiddity cycle is unrealizable, which is verified by the Realizability Algorithm in Section \ref{subsec:RealizabilityAlgorithm}. 
\item Due to these restrictions, we will never form quotient dissections starting with a dissection of $S_n$. In a once-punctured disc, given two subgons which share a vertex in $S_n$ either these subgons share an edge or at least one subgon has all vertices on the outer boundary. 
\end{enumerate}
\end{remark}

\begin{figure}
    \centering
\begin{tikzpicture}[scale = 1]

\draw (0,0) circle (2\R);
\draw (0,0) circle (0.5\R);
\coordinate(A) at (90:2\R);
\coordinate(B) at (90:0.5\R);
\coordinate(D) at (0:2\R);
\node[circle, fill = black, scale = 0.3] at (0:2\R){};
\coordinate(C) at (-60:2\R);
\coordinate(E) at (90:0.5\R);
\coordinate(F) at (-60:0.5\R);
\coordinate(G) at (0:0.5\R){};
\coordinate(H) at (240:2\R){};
\coordinate(I) at (240:0.5\R);
\coordinate(J) at (150:2\R){};
\coordinate(K) at (150:0.5\R){};
\node[above] at (90:2\R) {$2\sqrt{2}$};
\node[below] at (240:2\R) {$\sqrt{2}$};
\node[below] at (-60:2\R) {$\sqrt{2}$};
\node[left] at (150:2\R){$2\sqrt{2}$};
\node[right] at (0:2\R) {$\sqrt{2}$};
\draw(90:2\R) -- (90:0.5\R);
\draw(J) -- (K);
\draw(H) -- (I);
\draw(C) -- (F);
\draw(0:2\R) -- (0:0.5\R);
\filldraw[fill = gray, opacity = 0.3]    (H) arc[radius=2\R, start angle = 240, delta angle = -90] -- (J) -- (K) arc[radius = 0.5\R, start angle = 150, delta angle = 90] -- (I) -- (H);
\filldraw[fill = gray, opacity = 0.3]    (C) arc[radius=2\R, start angle = -60, delta angle = 60] -- (D) -- (G) arc[radius = 0.5\R, start angle = 0, delta angle = -60] -- (F) -- (C);
\filldraw[fill = gray, opacity = 0.3]    (C) arc[radius=2\R, start angle = -60, delta angle = -60] -- (H) -- (I) arc[radius = 0.5\R, start angle = 240, delta angle = 60] -- (F) -- (C);
\filldraw[fill = gray, opacity = 0.3]    (A) arc[radius=2\R, start angle = 90, delta angle = -90] -- (D) -- (G) arc[radius = 0.5\R, start angle = 0, delta angle = 90] -- (B) -- (A);
\end{tikzpicture}
    \caption{If we identify three consecutive pairs of 4-gons, which pairwise share edges, the resulting quiddity cycle does not pass the realizability test. We forbid quotient dissections which glue two subgons that share an edge.}
    \label{fig:QuotientAlongEdge}
\end{figure}
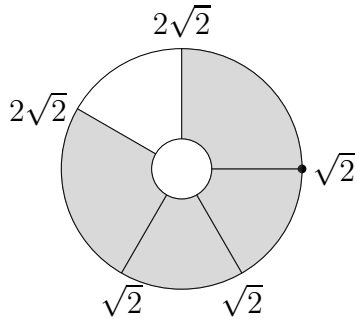

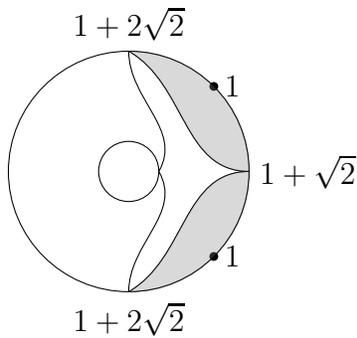
\begin{figure}
    \centering
    \begin{tikzpicture}[scale = 1]
\draw (0,0) circle (2\R);
\draw (0,0) circle (0.5\R);
\coordinate(A) at (90:2\R);
\node[circle, fill = black, scale = 0.3] at (45:2\R){};
\node[circle, fill = black, scale = 0.3] at (-45:2\R){};
\coordinate(B) at (0:2\R);
\coordinate(C) at (-90:2\R);
\coordinate(D) at (0:0.5\R);
\node[above] at (90:2\R) {$1 + 2\sqrt{2}$};
\node[below] at (270:2\R) {$1 + 2\sqrt{2}$};
\node[right] at (0:2\R) {$1 + \sqrt{2}$};
\node[right] at (45:2\R) {$1$};
\node[right] at (-45:2\R) {$1$};
\draw(A) to [out = -90, in = 60, looseness = 1](D);
\draw(C) to [out = 90, in = -60, looseness = 1](D);
\draw(A) to [out = -30, in = 180] (B);
\draw(C) to [out = 30, in = 180] (B);
\filldraw[fill = gray, opacity = 0.3]    (A) to [out = -30, in = 180] (B) arc[radius=2\R, start angle = 0, delta angle = 90] -- (A);
\filldraw[fill = gray, opacity = 0.3]    (C) to [out = 30, in = 180] (B) arc[radius=2\R, start angle = 0, delta angle = -90] -- (C);
\end{tikzpicture}
    \caption{A dissection where we identify two ears. The resulting quiddity sequence is a $3$-gluing to the unrealizable sequence $(2\sqrt{2}, \sqrt{2}, 1, 1 + 2\sqrt{2})$. We forbid such dissections.}
    \label{fig:QuotientByEars}
\end{figure}


\begin{definition}\label{def:DumplingFrieze}
We say that a frieze pattern $\F = \{m_{i,j}\}$ has a \emph{quotient dissection realization} if there exists some quotient dissection $\overline{\D}$ on $A_{n,m}$ such that, if $A_i$ is the (multi)set of sizes of subgons incident to $v_i$, for all $m_{i-1,i+1}$ we have
\[m_{i-1,i+1} =  \sum_{p\in A_i}\lambda_p.\]
\end{definition}

For example, in Example \ref{ex:quotientdissection} we see that the frieze pattern generated by $(2 + \sqrt{2},2\sqrt{2})$ has a quotient dissection realization. There are skeletal frieze patterns which satisfy Proposition \ref{prop:realizability} and also can be realized by a quotient dissection. We will focus on frieze patterns which have a quotient dissection realization and are not realizable by an ordinary dissection. 

We now show that quotient dissections are the missing piece in the sense that the realizability test (Proposition \ref{proposition:unrealizability1}) is sufficient to determine if a skeletal frieze pattern is realizable by a dissection or quotient dissection.

\begin{theorem}\label{thm:QuotientDissections} 
Every skeletal frieze pattern which passes the realizability test (Proposition \ref{proposition:unrealizability1}) is realizable by a dissection of a once-punctured disc or annulus, or a quotient dissection of an annulus. 
\end{theorem}

\begin{proof}
Let $\F = \{m_{i,j}\}$ be a frieze pattern which does not satisfy either condition of Proposition \ref{proposition:unrealizability1}. If moreover, $\F$ satisfies Proposition  \ref{prop:realizability}, then $\F$ is realized by a dissection of $S_n$ or $A_{n,m}$. 

Now, suppose that $\F$ does not satisfy item ii in Proposition  \ref{prop:realizability}. Then, it is not possible to pick a sequence $p_{1,2},p_{2,3},\ldots,p_{1,n}$ such that $p_{i-1,i} \in A_i$, $p_{i,i+1} \in A_i$, and if $p_{i-1,i} = p_{i,i+1}$, then this number appears with at least multiplicity two in $A_i$. 

However, since $\mathcal{F}$ passes the realizability test, we know that for all $i$, $A_i \cap A_{i+1} \neq \emptyset$.
Therefore, it is possible to pick a sequence $p_{i,i+1} \in A_i \cap A_{i+1}$. Let $j_1,\ldots,j_\ell$ be the subsequence of $[n]$ where $p_{j_i}:= p_{j_i -1,j_i} = p_{j_i,j_i+1}$, but this number appears only with multiplicity 1 in $A_{j_i}$. We know that for all choices of $p_{i,i+1}$ this set of indices must be nonempty; otherwise, this frieze pattern would be realizable by a dissection of $A_{n,m}$ or $S_n$ already. 

Let $(m_{0,2},\ldots,m_{n-1,n+1})$ be the quiddity cycle for $\F$. We introduce a modified quiddity cycle, $(\widehat{m}_{0,2},\ldots,\widehat{m}_{n-1,n+1})$. If there exists either at least two indices $i$ or no indices $i$ such that $\vert A_i \vert > 1$ and $i \notin \{j_1,\ldots,j_\ell\}$, define $\widehat{m}_{i-1,i+1}$ for $1 \leq i \leq n$ as 
\[
\widehat{m}_{i-1,i+1} = \begin{cases} m_{i-1,i+1} & i \notin \{j_1,\ldots,j_\ell\} \\ m_{i-1,i+1} + \lambda_{p_{j_k}} & i = j_k
 \\ \end{cases}
\]

If there exists exactly one index $i$ such that $\vert A_i \vert > 1$ and $i \neq j_k$,  then it must be that $p_{j_1} = p_{j_2} = \cdots = p_{j_\ell}$. In this case, we set $\widehat{m}_{i-1,i+1} = m_{i-1,i+1} + \lambda_{p_1}$ for all $1 \leq i \leq n$ 

By construction, $\widehat{F}$, the frieze pattern generated by the quiddity cycle $(\widehat{m}_{0,2},\ldots,\widehat{m}_{n-1,n+1})$ is realizable. We claim that the corresponding dissection, $\widehat{D}$, is on an annulus $A_{n,m}$ for some $m > 0$, and not on $S_n$. 

First, note that it would be impossible for $\widehat{F}$ to be non-skeletal. A frieze pattern is non-skeletal if its quiddity cycle has at least one string of $p-2$ entries $\lambda_p$. We produce $\widehat{F}$ by beginning with a skeletal frieze pattern $\mathcal{F}$ which passes the realizability test, then we add values $\lambda_p$ to several entries of the quiddity cycle; thus, we will not produce longer strings of entries $\lambda_p$.  

Now, note that the quiddity cycle from a skeletal dissection of $S_n$ will be of the form \begin{equation}\label{eq:punctureddiscquid}
(\ldots,\lambda_{p_1},\ldots, \lambda_{p_1},\lambda_{p_1} +\lambda_{p_2}, \lambda_{p_2},\ldots,\lambda_{p_2},\lambda_{p_2}+\lambda_{p_3},\ldots) \end{equation} with a string of $p_i-3$ entries $\lambda_{p_i}$ between each $\lambda_{p_i} + \lambda_{p_{i+1}}$. The entries $\lambda_{p_i}$ could not have been the result of adding $\lambda_{p_i}$ to the entry in the original quiddity cycle, or else the original entry would be zero. Similarly, if one of the entries $\lambda_{p_i} + \lambda_{p_{i+1}}$ resulted from adding either $\lambda_{p_i}$ or $\lambda_{p_{i+1}}$ to the original quiddity sequence, then the original frieze pattern would have failed the realizability test; that is, it would have satisifed one of the conditions of Proposition \ref{proposition:unrealizability1}. If $\lambda_{p_i} = \lambda_{p_{i+1}}$, then the original quiddity cycle would have had a string of over $p-2$ entries $\lambda_p$. If $\lambda_{p_i} \neq \lambda_{p_{i=1}}$, then the original quiddity cycle would have had adjacent entries $m_{j-1,j+1},m_{j,j+2}$ such that $A_j \cap A_{j+1} = \emptyset$. 

Thus, $\widehat{\D}$ is a dissection of $A_{n,m}$, for some $m \geq 1$. At each vertex $j_k$, $v_{j_k}$ is incident to two distinct outer $p_{j_k}$-subgons, one of which is shared with $v_{j_k}-1$ and the other with $v_{j_k}+1$. At each $j_k$, we identify these two subgons, forming a quotient dissection $\overline{\D}$. By construction, for all $i$, $m_{i-1,i+1}$ is the weighted sum of subgons incident to $v_i$ in $\overline{\D}$.
\end{proof}




We end this section by describing the universal cover of quotient dissections. As in the case of ordinary dissections, we will use the universal cover in our matching formulas in Section \ref{sec:Match}.

\begin{definition}
Let $\overline{\D}$ be a quotient dissection of $A_{n,m}$, and let $\D$ be the dissection of $A_{n,m}$ obtained by considering all identified pairs of subgons as distinct. Let $I$ be the dissected infinite strip, which is the universal cover of $\D$. 

Let $\poly{1}$ and $\poly{2}$ be two subgons in $\D$ which are both incident to $v_i$ and are identified in $\overline{\D}$. In $I$, for all $k \in \mathbb{Z}$ there are subgons $\poly{1}^k$ and $\poly{2}^k$ incident to $v_i^k$ such that $\rho(\poly{j}^k) = \poly{j}$ for $j = 1,2$. We identify all pairs $\poly{1}^k$ and $\poly{2}^k$, and repeat this process for all other pairs of subgons identified in $\overline{\D}$. Call this dissection of the infinite strip with identifications $\overline{I}$. We call $\overline{I}$ the universal cover of  $\overline{\D}$. 
\end{definition}

By construction, if $\poly{i}$ and $\poly{j}$ are two subgons of $\overline{I}$ which are identified, then $\rho(\poly{i})$ and $\rho(\poly{j})$ are identified in $\overline{\D}$. 

\subsection{Cutting}\label{subsec:Cut}

The conditions of Proposition \ref{prop:realizability} are not sufficient for non-skeletal frieze patterns. For example, the quiddity cycle $(1+\sqrt{2}, 1, 2, 1+\sqrt{2})$ satisfies all conditions, but we claim it is not realizable. Since this frieze pattern is not skeletal, we can \emph{cut} it in order to make it simpler. Baur, Parsons, and Tschabold introduced the notion of cutting a frieze pattern when proving one of their main results \cite{1}. They were working with frieze patterns of positive integers and triangulations of surfaces; we extend the notion of cutting to general dissections. 

\begin{definition}\label{def:cut}
Let $\mathcal{F}$ be a frieze pattern with quiddity cycle $(m_{0,2},\ldots,m_{n-1,n+1})$ where, for some $p \in \mathbb{Z}_{\geq 3}$, there exists $1 \leq i \leq n-(p-3)$ such that $m_{i-1,i+1} = m_{i,i+2} = \cdots = m_{i+(p-3)-1,i+(p-3)+1} = \lambda_p$. We can \emph{cut} at the interval $[i,i+p-3]$ and produce a new quiddity cycle $(m'_{0,2},\ldots,m'_{n-(p-2)-1,n-(p-2)+1})$. If $n > p-1,$ then we assume without loss of generality $1 < i < n-(p-3)$, and define\[
m'_{j-1,j+1} = \begin{cases} m_{j-1,j+1} & j < i-1\\
m_{i-1,i+1} - \lambda_p & j = i-1\\
m_{i+(p-2)-1,i+(p-2) + 1} - \lambda_p & j = i\\
m_{j + (p-2) -1, j+ (p-2) + 1} & j > i \\
\end{cases}
\]

If $n = p-1$, assume without loss of generality that $i = 2$. Then, the new quiddity cycle is 1-periodic, with $m_{0,2}' = m_{0,2} - 2\lambda_p$.

If $n = p-2$, then this is the 1-periodic frieze pattern with quiddity cycle $(\lambda_p)$, which corresponds to an empty dissection of a $p$-gon. Such a frieze pattern is finite by results of Holm and J{\o}rgensen  \cite{14}. We do not define the cut of such a frieze pattern. \end{definition}

As an example, we look again at our example of the quiddity sequence $(1 + \sqrt{2},1,2,1+\sqrt{2})$. We can cut at the interval $[2,2]$, producing $(\sqrt{2},1,1+\sqrt{2})$. The resulting quiddity cycle does not pass the realizability test. We will see in Lemma \ref{lem:CutRealizable} that this implies the original quiddity cycle is also not realizable.

\begin{remark}
\begin{enumerate}
    \item It may be necessary to cyclically shift a quiddity cycle to cut it using Definition \ref{def:cut}.
    \item If a quiddity cycle does not pass the realizability test, Proposition \ref{proposition:unrealizability1}, then the result of cutting could give negative entries in the adjusted cycle. For example, if our quiddity cycle is $(1,\sqrt{2},\sqrt{2},\sqrt{3})$, then cutting at interval $[2,3]$ results in $(1-\sqrt{2},\sqrt{3}-\sqrt{2})$.
\end{enumerate}
\end{remark}

We define $p$-\emph{gluing}, an inverse operation to cutting an interval of length $p-2$. Our definition follows that of Holm and J{\o}rgensen in \cite{14}. Unlike in the case of cutting, any frieze pattern admits a $p$-gluing for any $p \in \mathbb{Z}_{\geq 3}$ and at any interval. 

\begin{definition}\label{def:gluing}
Let $\F$ be a frieze pattern with quiddity cycle $(m_{0,2},\ldots,m_{n-1,n+1})$, and let $p \in \mathbb{Z}_{\geq 3}$. For some $1 \leq i \leq n$, we $p$-\emph{glue} between $i$ and $i+1$, producing a new quiddity cycle $(m'_{0,2},\ldots,m'_{n-1+(p-2),n+1+(p+2)})$ defined by \[
m_{j-1,j+1}' = \begin{cases}
m_{j-1,j+1} & j < i\\
m_{i-1,i+1} + \lambda_p & j = i\\
\lambda_p & i < j < i+(p-1)\\
m_{i,i+2} + \lambda_p & j = i+(p-1)\\
m_{j - 1 - (p-1), j+1 - (p-1)} & j > i + (p-1)\\
\end{cases}
\]
\end{definition}

For example, we can 6-glue between $1$ and $2$ on the quiddity cycle $(2\sqrt{2},\sqrt{2})$, producing $(2\sqrt{2} + \sqrt{3}, \sqrt{3}, \sqrt{3}, \sqrt{3}, \sqrt{3}, \sqrt{2} + \sqrt{3})$.
For any $p \in \mathbb{Z}_{\geq 3}$, $p$-gluing between $i$ and $i+1$ and cutting at the interval $[i,i+p-3]$ are inverse operations. 

\begin{lemma}\label{lem:CutRealizable}
Suppose a frieze pattern $\mathcal{F}$ with quiddity cycle $(m_{0,2},\ldots,m_{n-1,n+1})$ is realizable by a dissection of $A_{n,m}, S_n$, or an $n$-gon, $P_n$. 
\begin{enumerate}
    \item If we can cut this quiddity cycle at $[i,i+p-3]$, then resulting frieze pattern is realizable by a dissection of $A_{n-(p-2),m}, S_{n-(p-2)}$, or $P_{n-(p-2)}$. 
    \item For any $p \in \mathbb{Z}_{\geq 3}$ and any $1 \leq i \leq n$, the frieze pattern resulting from  $p$-gluing between $i$ and $i+1$ is realizable by a dissection of $A_{n+(p-2),m}, S_{n+(p-2)}$, or $P_{n+(p-2)}$. 
\end{enumerate}
\end{lemma}

\begin{proof}
(1) Since we can cut $(m_{0,2},\ldots,m_{n-1,n+1})$ at $[i,i+p-3]$, it must be that $m_{i-1,i+1} = \cdots = m_{i+(p-3) -1, i+(p-3)+1} = \lambda_p$; since $\mathcal{F}$ is realizable, $m_{i-2,i} = \sum_{q \in A_{i-1}} \lambda_q$ and $m_{i+p-3,i+p-1} = \sum_{q \in A_{i+p-2}} \lambda_q$ where $p \in A_{i-1}$, $p \in A_{i+p-2}$, $\vert A_{i-1} \vert > 1$, and $\vert A_{i+p-2} \vert > 1$ . In the realization of $\mathcal{F}$, vertices $v_i,\ldots,v_{i+p-3}$ are all adjacent to an ear. This ear's unique non-boundary edge is between $v_{i-1}$ and $v_{i+p-2}$. Performing the algebraic operation of cutting $(m_{0,2},\ldots,m_{n-1,n+1})$ at $[i,i+p-3]$ is equivalent to removing this ear from the dissection. If the realization of $\mathcal{F}$ is a dissection of an $n$-gon, then removing this ear produces an $(n - (p-2))$-gon and the resulting frieze pattern is width $n - (p-2)-3 = n-p-5$. If the realization of $\mathcal{F}$ is a dissection of $S_n$ or $A_{n,m}$, then removing this ear produces $S_{n-(p-2)}$ or $A_{n - (p-2),m}$ respectively. 

Note that, if $m_{i-1,i+1} = \lambda_p$ for all $i$, then this is a finite frieze pattern of width $p-3$ which corresponds to an empty dissection of a $p$-gon. We do not define the cut of such a frieze pattern in Definition \ref{def:cut}.

(2) A $p$-gluing between $i$ and $i+1$ corresponds to attaching an ear of size $p$ to the realization of $\mathcal{F}$ such that the edge between vertices $v_i$ and $v_{i+1}$ is the unique non-boundary edge of the ear.
\end{proof}

\begin{center}
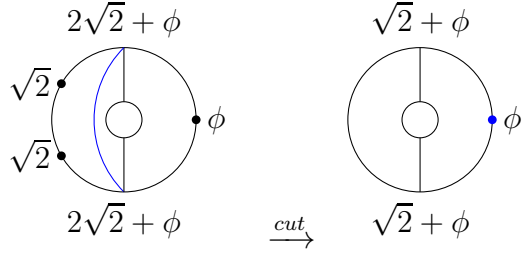
\begin{figure}
\centering
\begin{tabular}{c c c c }
     
\begin{tikzpicture}[scale = 0.6]

\draw (0,0) circle (2\R);
\draw (0,0) circle (0.5\R);
\coordinate(A) at (90:2\R);
\node[circle, fill = black, scale = 0.3] at (150:2\R){};
\node[circle, fill = black, scale = 0.3] at (210:2\R){};
\node[circle, fill = black, scale = 0.3] at (0:2\R){};
\coordinate(C) at (-90:2\R);
\coordinate(B) at (-90:0.5\R);
\coordinate(E) at (90:0.5\R);
\node[left] at (150:2\R) {$\sqrt{2}$};
\node[left] at (210:2\R) {$\sqrt{2}$};
\node[above] at (90:2\R) {$2\sqrt{2} + \phi$};
\node[below] at (270:2\R) {$2\sqrt{2} + \phi$};
\node[right] at (0:2\R) {$\phi$};
\draw(90:2\R) -- (90:0.5\R);
\draw[blue](C) to [out = 135, in = -135](A);
\draw(C) to [out = 90, in = -90](B);

\end{tikzpicture}

&$\xrightarrow{cut}$&
\begin{tikzpicture}[scale = 0.6]

\draw (0,0) circle (2\R);
\draw (0,0) circle (0.5\R);
\coordinate(A) at (90:2\R);
\coordinate(C) at (-90:2\R);
\coordinate(B) at (-90:0.5\R);
\coordinate(E) at (90:0.5\R);
\node[circle, fill = blue, scale = 0.3] at (0:2\R){};
\node[above] at (90:2\R) {$\sqrt{2} + \phi$};
\node[below] at (270:2\R) {$\sqrt{2} + \phi$};
\node[right] at (0:2\R) {$\phi$};
\draw(90:2\R) -- (90:0.5\R);
\draw(C) to [out = 90, in = -90](B);
\end{tikzpicture}
\end{tabular}
    \caption{The algebraic operation of cutting a frieze pattern corresponds to removing an ear. Gluing a frieze pattern corresponds to adding an ear.}
\end{figure}
\end{center}

When a frieze pattern $\F$ is only realizable by a quotient dissection, cutting and gluing have the same geometric meaning.

\begin{lemma}\label{lem:DumplingGlue}
Suppose a frieze pattern $\mathcal{F}$ with quiddity cycle $(m_{0,2},\ldots,m_{n-1,n+1})$ is realizable by a quotient dissection $\overline{\D}$ of $A_{n,m}$. 
\begin{enumerate}
    \item If we can cut this quiddity cycle at $[i,i+p-3]$, then the resulting frieze pattern is realizable by a quotient dissection of $A_{n-(p-2),m}$. 
    \item For any $p \in \mathbb{Z}_{\geq 3}$ and any $1 \leq i \leq n$, the frieze pattern resulting from  $p$-gluing between $i$ and $i+1$ is realizable by a quotient dissection of $A_{n+(p-2),m}$. 
\end{enumerate}
\end{lemma}

\begin{proof}
Recall we do not allow identifying subgons with all vertices on the same boundary. Thus, if we can cut $\F$ at the interval $[i,i+p-3]$, then in the quotient dissection, $\overline{\D}$, vertices $v_i,\ldots,v_{i+p-3}$ are incident to an ear which has not been identified to any other subgon. Therefore, we can remove this ear and this produces a valid quotient dissection of $A_{n-(p-2),m}$. Similarly, we can always add a ear of size $p$ to a quotient dissection, which would produce a quotient dissection of $A_{n+(p-2),m}$. 
\end{proof}

\subsection{Realization Algorithm}\label{subsec:RealizabilityAlgorithm}

Since we have direct realizability criteria for skeletal frieze patterns in Proposition \ref{prop:realizability}, we can use cutting to characterize the realizability of arbitrary frieze patterns of Type $\Lambda_{p_1,\ldots,p_s}$. 

Let $\F_0$ be any frieze pattern of Type $\Lambda_{p_1,\ldots,p_s}$ with a period $n$ quiddity cycle.If $\F_0$ is finite, of width $n-3$, we must consider the quiddity row of $\F_0$ as $n$-periodic. If $\F_0$ is infinite, we can consider any valid period of its quiddity row. We run the following algorithm to determine whether $\F_0$ is realizable; moreover, if $\F_0$ is realizable, we will be able to determine which surface can be dissected to realize $\F_0$ or whether $\F_0$ is only realizable by a quotient dissection.

Let $i \geq 0$.
\begin{enumerate}
    \item If $\F_i$ satisfies at least one condition of Proposition \ref{proposition:unrealizability1} or has any entries in the quiddity sequence which are less than or equal to 0, it is unrealizable.
    \item If $\F_i$ is generated by the quiddity cycle $(\lambda_p)$, then $\F_i$ is realizable by an empty dissection of a $p$-gon. $\F_0$ is realizable by a dissection of an $n$-gon.
    \item If $\F_i$ passes the realizability test, is not generated by the quiddity cycle $(\lambda_p)$, and is not skeletal, let $\F_{i+1}$ be the result of cutting $\F_i$ at one interval. Return to step 1.
    \item If $\F_i$ passes the realizability test and is skeletal, but does not satisfy Proposition \ref{prop:realizability}, then $\F_i$ is realizable by a quotient dissection. $\F_0$ is realizable by a quotient dissection of $A_{n,m}$.    
    \item If $\F_i$ satisfies Proposition \ref{prop:realizability}, implying that it passes the realizability test and is skeletal, then $\F_i$ is realizable by a dissection of an annulus or once-punctured disc. $\F_0$ is realizable by a dissection of $A_{n,m}$ (for some $m$) or $S_n$.
    \begin{enumerate}
        \item Lemma \ref{lem:SkeletalSn} provides one way to determine whether $\F_i$ is realizable by $S_n$ or $A_{n,m}$ 
        \item Proposition \ref{prop:GrCo2} gives another method of checking: $\F_i$ is realizable by $S_n$ if only if the \emph{principal growth coefficient} (see Theorem \ref{thm:GrowthCoefficient}) of $\F_i$ is 2. 
    \end{enumerate}
\end{enumerate}

When we end in cases 2,4, or 5, Lemma \ref{lem:CutRealizable} and Lemma \ref{lem:DumplingGlue} give the guarantee that $\F_0$ is also realizable by the same type of dissection as $\F_i$. In each case, these realizations can be constructed by geometrically gluing subgons onto the skeletal realization of $\F_i$ in the reverse order of the cuts performed to get from $\F_0$ to $\F_i$.

\begin{example}
\begin{enumerate}
    \item Consider $\F_0$ with quiddity cycle  $(1+\sqrt{2}+\sqrt{3},\sqrt{3},\sqrt{3},\sqrt{3},1+ \sqrt{3},1,1 +\sqrt{2}+\sqrt{3}, 1+\sqrt{2}) $. This quiddity cycle passes the realizability test, but is not skeletal. We cut at $[6,6]$ to produce $\F_1$ with quiddity cycle  $(1+\sqrt{2}+\sqrt{3},\sqrt{3},\sqrt{3},\sqrt{3},\sqrt{3},\sqrt{2}+\sqrt{3}, 1+\sqrt{2})$. This new quiddity cycle still passes the realizability test and still is not skeletal. We cut at $[2,5]$ to produce $\F_2$ with quiddity cycle $(1 + \sqrt{2},\sqrt{2},1+\sqrt{2})$. By Lemma \ref{lem:SkeletalSn}, we see that $\F_2$ is realizable by a skeletal dissection of $S_3$. Thus, $\F_0$ is realizable by a dissection of $S_8$.
    \item Consider $\F_0$ with quiddity cycle $(1+\sqrt{2},1,2,1+\sqrt{2})$. We cut at $[2,2]$ to produce $(\sqrt{2},1,1+\sqrt{2})$. Since this quiddity cycle does not pass the realizability test, $\F_0$ is not realizable. 
\end{enumerate}
\end{example}

\section{Matchings}\label{sec:Match}

\subsection{Background}

We now introduce combinatorial interpretations of the entries of frieze patterns of Type $\Lambda_{p_1,\ldots,p_n}$ from dissected surfaces. Our interpretations will hold in both infinite and finite cases. Each entry in a frieze pattern will correspond to to a sum of weighted \emph{matchings} on the surface. We will in fact introduce two distinct weightings for matchings; these two weightings shed light on different features of these frieze patterns. 

\begin{definition}\label{def:Matching}
Let $S$ be a dissected surface, with vertices $v_1,\ldots, v_n$ on one boundary, labeled in counterclockwise order. A \emph{matching} or \emph{path} $w$ between $v_i$ and $v_j$ is a sequence of subgons $w = \poly{i+1},\ldots,\poly{j-1},$ (working modulo $n$) such that $\poly{k}$ is incident to $v_k$ for all $i+1 \leq k \leq j-1$.
\end{definition}

We will use $\pathset{i}{j}$ to denote the set of all paths from $v_i$ to $v_j$. We will also set the convention that $\pathset{i}{i+1} = \{ \emptyset\}$ while $\pathset{i}{i} = \emptyset$. These conventions will allow us to extend our combinatorial interpretation to the upper boundary rows of 0's and 1's in a frieze pattern.

Traditionally, a matching $w = \poly{i+1},\ldots,\poly{j-1}$ in a triangulated surface is only admissible if all $\poly{k}$ are pairwise distinct. For purposes of this section, we denote the subset of matchings with this property $\pathset{i}{j}^d$, with $d$ for ``distinct''. Later, we will rephrase this property by weighting matchings which use the same triangle twice with weight 0 (see Definition \ref{def:TradWeight}). 

Broline, Crowe and Isaacs were the first to connect matchings to frieze patterns \cite{5}.

\begin{theorem}[\cite{5}]\label{thm:BCI}
Let $\F = \{m_{i,j}\}$ be a finite frieze pattern of positive integers of width $n-3$, and let $P$ be the triangulated $n$-gon which is the realization of $\F$. Label the vertices of $P$ in counterclockwise order $v_1,\ldots,v_n$, treating these labels modulo $n$. Then, \[
m_{i,j} = \sum_{w \in \pathset{i}{j}^d} 1 = \vert  \pathset{i}{j}^d\vert
\]

\end{theorem}

 Baur, Parsons, and Tschabold extended this correspondence to infinite frieze patterns of positive integers in \cite{1}. In these cases, one must work in the universal cover of the triangulated surface which realizes the frieze pattern. Even though this surface is infinite, given any two vertices $v_i^{k}$ and $v_j^\ell$ on the infinite strip with a dissection, we can work in a finite piece (a polygon) which contains $v_i^k$ and $v_j^\ell$.

 In the statement of the theorem, we use the following notation. Let $i \in \mathbb{Z}$, $n \in \mathbb{Z}_{>0}$. Then, set \[
 i_n = \begin{cases} i \pmod{n} & n \nmid i \\
 n & n \mid i \end{cases}.
 \]
 Recall our notation for infinite strips established in Section \ref{subsec:InfStrip}. 

\begin{theorem}[\cite{1}]\label{thm:BPTMatchings}
Let $\F = \{m_{i,j}\}$ be an infinite, $n$-periodic frieze pattern of positive integers. Let $S$ be the triangulated surface which is the realization of $\F$. Let $I$ be the triangulated infinite strip which is the universal cover of $S$ . Let $\{v_i^k\}$ be the infinite family of vertices of $I$ which map to the vertices on the outer boundary of $S$ under the covering map. Then,\[
m_{i,j} = \sum_{w \in \pathset{i}{j}^d} 1 = \vert  \pathset{i}{j}^d\vert
\]
where $\pathset{i}{j}$ is the set of matchings between $v_{i_n}^{\lfloor \frac{i}{n} \rfloor}$ and $v_{j_n}^{\lfloor \frac{j}{n} \rfloor}$.
\end{theorem}


Our weighting scheme will use evaluations of normalized Chebyshev polynomials at $\lambda_p$. See Section \ref{sec:FPType} for the definition of these Chebyshev polynomials and some basic results.

 In a polygon, we say a \emph{$k$-diagonal} is one which skips $k$ vertices counterclockwise. A side of a polygon is a 0-diagonal. The following result comes from work in \cite{15} and adds motivation for the presence of the terms $U_k(\lambda_p)$ in our weightings.

\begin{lemma}\label{lem:ChebyshevGeometric}
In a regular $p$-gon, the ratio of the length of a $k$-diagonal and a side is $U_k(\lambda_p)$. In particular,\begin{enumerate}
    \item since we could equivalently look at vertices skipped in clockwise order, $U_k(\lambda_p) = U_{p - 2 - k}(\lambda_p)$, and 
    \item since $k$-diagonals in a polygon will be at least as long as the sides, for $0 \leq k \leq p-2, U_k(\lambda_p) \geq 1$.
\end{enumerate}
\end{lemma}

We calculate some values of $U_k(\lambda_p)$. Recall that by Lemma \ref{lem:ChebyPeriodic}, $U_k(\lambda_p) = -U_{k+p}(\lambda_p)$.

\begin{center}
\begin{tabular}{c||c|c|c|c}
k \textbackslash p & 3 & 4 & 5 & 6\\\hline\hline
0 & 1 & 1 & 1 & 1\\
1 & 1 & $\sqrt{2}$& $\phi$ & $\sqrt{3}$\\
2 & 0 & 1 & $\phi$ & 2 \\
3 & -1 & 0 & 1 & $\sqrt{3}$\\
4 & -1 & -1 & 0 & 1 \\
5 & 0 & $-\sqrt{2}$ & -1 & 0\\ 
\end{tabular}
\end{center}

\subsection{Local Weighting}

We now define our first method to weight paths in dissected polygons. We will use $w[a:b] = \poly{a},\ldots,\poly{b}$ to denote a submatching of $w = \poly{i+1},\ldots,\poly{j-1}$, where $i+1\leq a\leq b\leq j-1$. If $a > b$, then $w[a:b] = \emptyset$.

\begin{definition}[Local Weighting]\label{def:LocalWeight}
Given a matching $w = \poly{i+1},\ldots,\poly{j-1}$ in a dissected polygon, we define the local weight $wt_L(w)$ recursively. Suppose $k$ is the largest positive integer such that the first $k$ subgons in $w$ are the same, i.e. $\poly{i+1} = \poly{i+2} = \cdots = \poly{i+k} \neq \poly{i+k+1}$, for some $1 \leq k \leq j-i-1$. Then, we set \[
wt_L(w) = U_k(\lambda_{\vert \poly{i+1}\vert}) wt_L(w[i+k+1:j-1]).
\]
We set $wt_L(\emptyset) = 1$.
\end{definition}

This weighting is ``local'' in the sense that it only depends on consecutive occurrences of the same subgon. See Example \ref{ex:Matchings} and Figure \ref{fig:WeightingExample} for examples of calculating $wt_L$.

Lemma \ref{lem:EquivalentLocalWeight} gives an equivalent definition of the local weighting by recursively checking if two consecutive subgons in a matching are equal. 

\begin{lemma}\label{lem:EquivalentLocalWeight}
Given a matching  $w = \poly{i+1},\ldots,\poly{j-1}$ in a dissected polygon, we define $wt_{L'}$ recursively. If $\poly{i+1} = \poly{i+2},$ set \[
wt_{L'}(w) = \lambda_{\vert \poly{i+1}\vert} wt_{L'}(w[i+2:j-1]) - wt_{L'}(w[i+3:j-1]),
\]
and otherwise we set \[
wt_{L'}(w) = \lambda_{\vert \poly{i+1}\vert} wt_{L'}(w[i+2:j-1]).
\]
We set $wt_{L'}(\emptyset) = 1$.
Then, for any $w$, $wt_L(w) = wt_{L'}(w)$.
\end{lemma}

\begin{proof}
Let $w \in \pathset{i}{j}$. We will induct on $j - i$. 

If $j = i$, there are no matchings in $\pathset{i}{j}$. If $j = i+1$, then $\pathset{i}{j} = \{\emptyset\}$, and $wt_L(\emptyset) = wt_{L'}(\emptyset) = 1$. If $j = i+2$, then each matching $w \in \pathset{i}{j}$ is of the form $w = \poly{i+1}$. In this case as well, $wt_L(\poly{i+1}) = wt_{L'}(\poly{i+1}) = \lambda_{\vert \poly{i+1}\vert}$.

Now, suppose we have established the claim for all $w \in \pathset{i}{j}$ with $j \leq i + \ell - 1$, $\ell \geq 3$ and all possible values of $k$, the number of consecutive subgons at the beginning of a matching which are the same. Let $j = i + \ell$ and let $w \in \pathset{i}{j}$. Suppose the first $k \geq 1$ subgons of $w$ are the same. If $k = 1$, then by induction we immediately have \[
wt_{L'}(w) = \lambda_{\vert \poly{i+1}\vert} wt_{L'}(w[i+2,j-1]) = wt_L(w)
\]

Now let $k > 1$. Then,\[
wt_{L'}(w) = \lambda_{\vert \poly{i+1}\vert} wt_{L'}(w[i+2:j-1]) - wt_{L'}(w[i+3:j-1])
\]
By induction, we can rewrite the righthand side in terms of $wt_L$,\begin{align*}
wt_{L'}(w) = \lambda_{\vert \poly{i+1}\vert} wt_{L'}(w[i+2:j-1]) - wt_{L'}(w[i+3:j-1])\\
= \lambda_{\vert \poly{i+1}\vert} U_{k-1}(\lambda_{\vert \poly{i+1}\vert})wt_{L}(w[i+k+1:j-1]) - U_{k-2}(\lambda_{\vert \poly{i+1}\vert})wt_{L}(w[i+k+1:j-1])
\\= U_{k}(\lambda_{\vert \poly{i+1}\vert})wt_{L}(w[i+k+1:j-1]).\\
\end{align*}
\end{proof}

The formulation of the local weight in Lemma \ref{lem:EquivalentLocalWeight} resembles the recurrences for frieze pattern entries in Lemma \ref{lem:DivisionFreeFrieze}.
This fact is crucial for verifying that matchings with the local weighting give a combinatorial interpretation of frieze pattern entries.

\begin{theorem}\label{thm:LocalWeight}
\begin{enumerate}
    \item Let $\mathcal{F} = \{m_{i,j}\}$ be a finite frieze pattern of Type $\Lambda_{p_1,\ldots,p_s}$ and of width $n-3$ which is realizable by a dissected polygon, $P$. Let $v_1,\ldots,v_n$ be the vertices of $P$.  Then for all $i,j$ with $0 \leq j - i \leq n$ \[
    m_{i,j} = \sum_{w \in \pathset{i}{j}} wt_L(w)
    \]
    \item Let $\mathcal{F} = \{m_{i,j}\}$ be an infinite, $n$-periodic frieze pattern of Type $\Lambda_{p_1,\ldots,p_s}$ which is realizable by a dissection of $S_n$ or $A_{n,m}$ or a quotient dissection. Let $I$ be the dissected infinite strip which is the universal cover.  Let $\{v_i^k\}$ be the infinite family of vertices of $I$ which map to the vertices on the outer boundary of $S$ under the covering map. Then,\[
    m_{i,j} = \sum_{w \in \pathset{i}{j}} wt_L(w)
    \]
    where $\pathset{i}{j}$ is the set of matchings between $v_{i_n}^{\lfloor \frac{i}{n} \rfloor}$ and $v_{j_n}^{\lfloor \frac{j}{n} \rfloor}$.
\end{enumerate}
\end{theorem}

\begin{proof}
We will consider finite and infinite frieze patterns simultaneously. We prove this theorem by induction. We have that \[
\sum_{w \in \pathset{i}{i}} wt_L(w) = 0 = m_{i,i}
\]
since $\pathset{i}{i} = \emptyset$.
Similarly, $\pathset{i}{i+1} = \{\emptyset\}$, giving \[
\sum_{w \in \pathset{i}{i+1}} wt_L(w) = wt_L(\emptyset) = 1 = m_{i,i+1}.
\]
By definition, we also know that every entry in the quiddity row satisfies the formula
\[
m_{i-1,i+1} = \sum_{w\in \pathset{i-1}{i+1}}wt_L(w) = \sum_{\mathcal{P}\in \polyset{v_i}}\lambda_{|\mathcal{P}|}.
\]

Suppose we have verified the claim for all $m_{i,j}$ with $j - i < k$. If our frieze pattern is finite of width $k-4$, then for all $i$, $m_{i,i+(k-1)} = 0$. Hence, we have checked all rows of the frieze pattern and we are done. 

Now, suppose our frieze pattern is infinite or finite with width larger than $k-4$, and consider $m_{i,j}$ with $j = i + k$. By Lemma \ref{lem:DivisionFreeFrieze}, we have that $m_{i,j} = m_{i,i+2}m_{i+1,j} - m_{i+2,j}$. Thus, it suffices to prove \[
\sum_{w \in \pathset{i}{j}} wt_L(w) = \bigg(\sum_{w \in \pathset{i}{i+2}} wt_L(w)\bigg) \bigg( \sum_{w \in \pathset{i+1}{j}} wt_L(w) \bigg) - \sum_{w \in \pathset{i+2}{j}} wt_L(w)
\]
where, by induction, we already know that the right hand side is equal to $ m_{i,i+2}m_{i+1,j} - m_{i+2,j}$.

Let $\hat{\mathcal{P}}$ be the unique subgon which is incident to both $v_{i+1}$ and $v_{i+2}$. There is a bijection between the set $\pathset{i+2}{j}$ and the subset of $\pathset{i}{j}$ consisting of matchings $w = \poly{i+1},\ldots,\poly{j-1}$ with $\poly{i+1} = \poly{i+2} = \hat{\mathcal{P}}$. This bijection is given by sending $h \in \pathset{i+2}{j}$ to $w_h := \hat{\mathcal{P}}, \hat{\mathcal{P}}, h \in \pathset{i}{j}$. Let $\pathset{i}{j}' \subseteq \pathset{i}{j}$ be the complement of the set of matchings of the form $w_h$. It is possible that $\pathset{i}{j}'$ is empty. 

We break up our sum based on this partition of $\pathset{i}{j}$\[
\sum_{w \in \pathset{i}{j}} wt_L(w) = \sum_{h \in \pathset{i+2}{j}} wt_L(w_h) + \sum_{w \in \pathset{i}{j}'} wt_L(w).
\]

Lemma \ref{lem:EquivalentLocalWeight} transforms the two terms on the right hand side as \begin{equation}\label{eq:localweight1}
\sum_{h \in \pathset{i+2}{j}} wt_L(w_h) = \sum_{h \in \pathset{i+2}{j}} (\lambda_{\vert \hat{\mathcal{P}} \vert} wt_L(\hat{\mathcal{P}},h) - wt_L(h))
\end{equation}
and
\begin{equation}\label{eq:localweight2}
\sum_{w \in \pathset{i}{j}'} wt_L(w) = \bigg(\sum_{\mathcal{P} \in \pathset{i}{i+2}} wt_L(\mathcal{P}) \bigg)\bigg( \sum_{w' \in \pathset{i+1}{j}} wt_L(w')\bigg) - \sum_{h \in\pathset{i+2}{j}} \lambda_{\vert \hat{\mathcal{P}} \vert} wt_L(\hat{\mathcal{P}},h).
\end{equation}
The subtracted term in Equation \ref{eq:localweight2} is removing the contribution of any matchings of the form $w_h$ for $h \in \pathset{i+2}{j}$.

From equations \ref{eq:localweight1} and \ref{eq:localweight2}, we conclude that \[
\sum_{w \in \pathset{i}{j}} wt_L(w) = \bigg(\sum_{\mathcal{P} \in \pathset{i}{i+2}} wt_L(\mathcal{P}) \bigg)\bigg( \sum_{w' \in \pathset{i+1}{j}} wt_L(w')\bigg) - \sum_{h \in \pathset{i+2}{j}} wt_L(h).
\]
\end{proof}

In the proof of Theorem \ref{thm:LocalWeight}, if we were considering a finite frieze pattern, we stopped the induction once we reached the lower boundary row of 0's. However, if we use the recurrence in Lemma \ref{lem:DivisionFreeFrieze} instead of the diamond relation, we can extend a finite frieze pattern to an infinite frieze pattern with some negative rows, as in Example \ref{ex:InfinitePolygon}. Such infinite frieze patterns are related to \emph{$SL_2$-tilings}, as discussed in \cite{7}. 
This allows us to consider extending the inductive argument in Theorem \ref{thm:LocalWeight} in the case of a frieze pattern from a polygon.

\begin{corollary}\label{cor:InfinitePolygonMatching}
Let $\mathcal{F} = \{m_{i,j}\}$ be an \emph{infinite} frieze pattern of Type $\Lambda_{p_1,\ldots,p_s}$ from a dissection of an $n$-gon with vertices $v_1,\ldots,v_n$. Then,\[
m_{i,j} = \sum_{w \in \pathset{i}{j}} wt_L(w)
\]

where if $\vert j - i \vert \geq n$, then $\pathset{i}{j}$ is the set of matchings between $v_{i_n}$ and $v_{j_n}$ which visits vertices $v_{(i+1)_n},\ldots,v_{(j-1)_n}$.
\end{corollary}

\begin{example}\label{ex:InfinitePolygon}
See below for an example of an extension of a finite frieze pattern to an infinite frieze pattern. On the left we give the dissected polygon which realizes this frieze pattern.

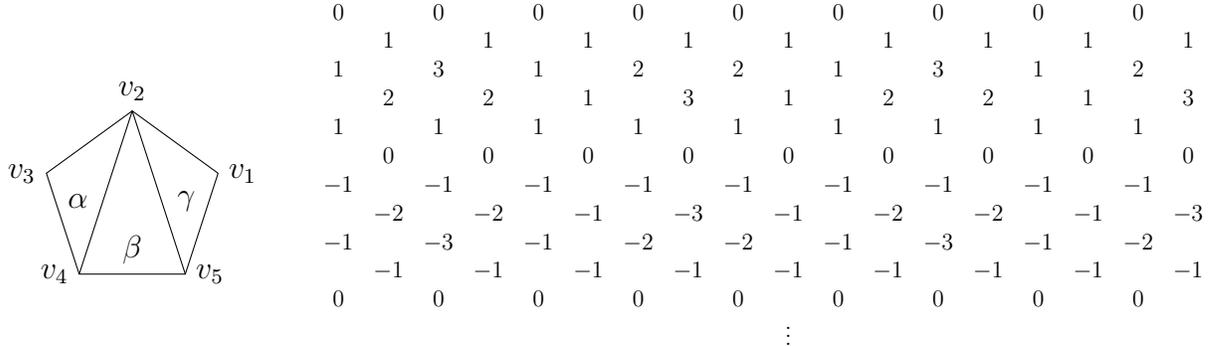
\raisebox{-.5\totalheight}{\begin{tikzpicture}[scale = 1]
  \draw (90-72:1.5\R) node[right]{$v_1$} --  (90:1.5\R) node[above]{$v_2$} --   (90+72:1.5\R) node[left]{$v_3$}  --   (90+144:1.5\R) node[left]{$v_4$} -- (90+216:1.5\R) node[right]{$v_5$} -- (18:1.5\R);
 \node[] at (270:0.8\R){$\beta$};
 \node[] at (0:0.9\R){$\gamma$};
 \node[] at (180:0.9\R){$\alpha$};
  \draw(90+144:1.5\R) -- (90:1.5\R);
  \draw(90:1.5\R) -- (90+216:1.5\R);  
  \end{tikzpicture}} 
 \scalebox{0.75}{$\begin{array}{cccccccccccccccccccc}
   &&0&&  0&&  0&&  0&&  0&&  0&&  0&&  0&&0&\\
  &&&  1&&  1&&  1&&  1&&  1&&  1&&  1&&1&&1\\
 &&1&&3&&1&&  2&&  2&&  1&&  3&&1&&2&\\
  &&&2&&  2&&  1&& 3&&  1&&  2&&2&&1&&3\\
  &&1&&  1&&  1&&  1&&  1&&  1&&  1&&1&&1&\\
 &&&0&&  0&&  0&&  0&&  0&&  0&&  0&&  0&&0\\
  &&-1&&  -1&&  -1&&  -1&&  -1&&  -1&&  -1&&-1&&-1&\\
  &&&-2&&-2&&-1&&  -3&&-1&&-2&&-2&&-1&&-3\\
   &&-1&&-3&&-1&&-2&&-2&&-1&&-3&&-1&&-2&\\
  &&&-1&&  -1&&  -1&&  -1&&  -1&&  -1&&  -1&&-1&&-1\\ 
  &&0&&  0&&  0&&  0&&  0&&  0&&  0&&  0&&0&\\
   &&&&& && && && \vdots&& && &&&&\\  
 \end{array}$}
 
 As an example of Corollary \ref{cor:InfinitePolygonMatching}, note that all entries $m_{i,i+6} = -1$. We look at the case $i=2$; that is, $m_{2,8} = \sum_{w \in \pathset{2}{8}} wt_L(w)$.  The set $\pathset{2}{8}$ consists of matchings of the form $\poly{3},\poly{4},\poly{5},\poly{6},\poly{7}$, where $\poly{i}$ is incident to $v_{i_n}$. Amongst the 12 matchings in $\pathset{2}{8}$, $wt_L(\alpha\beta\gamma\gamma\gamma) = -1$ and all other matchings have weight 0. Recall that $\lambda_3 = 1$, $U_2(1) = 0$ and $U_3(1) = -1$. 
 
 For an example with more nonzero matchings, we compute $\sum_{w \in \pathset{4}{10}} wt_L(w) = -1$.
 
 \begin{center}
 \begin{tabular}{c|c||c|c}
     $w$ & $wt_L(w)$& $w$ & $wt_L(w)$ \\\hline
    $\beta\gamma\gamma\alpha\alpha$  & 0 & $\gamma\gamma\gamma\alpha\alpha$  & 0\\
    $\beta\gamma\beta\alpha\alpha$  & 0 & $\gamma\gamma\beta\alpha\alpha$  & 0\\
    $\beta\gamma\gamma\alpha\beta$  & 0 & $\gamma\gamma\gamma\alpha\beta$  & $-1$\\ $\beta\gamma\alpha\alpha\alpha$  & $-1$ & $\gamma\gamma\alpha\alpha\alpha$  & 0\\ 
    $\beta\gamma\beta\alpha\beta$  & 1 & $\gamma\gamma\beta\alpha\beta$  & 0\\
    $\beta\gamma\alpha\alpha\beta$  & 0 & $\gamma\gamma\alpha\alpha\beta$  & 0\\ 
 \end{tabular}
 \end{center}
\end{example}

\subsection{Traditional Weighting}

We now have one combinatorial interpretation for entries of a realizable frieze pattern of Type $\Lambda_{p_1,\ldots,p_s}$. However, there are some matchings $w$ such that $wt_L(w) < 0$, even though the frieze pattern entries appear to be all positive. We introduce another weighting, $wt_T$, which satisfies $wt_T(w) \geq 0$ for all matchings $w$.

\begin{definition}[Traditional Weighting]\label{def:TradWeight}
Consider a dissected polygon, with vertices $v_1,\ldots,v_n$. Let $w = (\poly{i+1},\ldots,\poly{j-1})$ be a matching between vertices $v_i$ and $v_j$. For all $\mathcal{P}$ in the dissection, let $k_{\mathcal{P}}$ be the number of times that $\mathcal{P}$ is used in $w$. If for any $\mathcal{P}$, $k_{\mathcal{P}} > \vert \mathcal{P} \vert - 2$, then $wt_T(w) = 0$. Otherwise, we set\[
wt_T(w) = \prod_{\mathcal{P} \in w} U_{k_{\mathcal{P}}} (\lambda_{\vert \mathcal{P} \vert})
\]
\end{definition}

We define this weighting on a dissected polygon; when working with a dissection of $A_{n,m}$ or $S_n$, we will calculate the weight of a matching from the universal cover. 

This weighting scheme is ``traditional'' in the sense that, in the case of triangulations, it forbids using a triangle more than one time in a matching, as was the case in \cite{5} and \cite{1}. In the case of dissections, Bessenrodt also only permitted using a $p$-gon up to $p-2$ times in a matching in \cite{4}. This weighting differs from Bessenrodt's though, as she used formal variables and did not use Chebyshev polynomials to track using a subgon multiple times. 

\begin{remark}
We do not define traditional weighting on a quotient dissection. We do not yet know a consistent way to weight matchings in a quotient dissection, other than with $wt_L$, which would give an interpretation of the entries of the corresponding frieze pattern. In particular, we cannot verify that all frieze patterns from quotient dissections have positive entries; see Conjecture \ref{conj:Positive}.
\end{remark}

See Example \ref{ex:Matchings} and Figure \ref{fig:WeightingExample} for sample calculations of $wt_T$. 

\begin{example}\label{ex:Matchings}
In Figure \ref{fig:WeightingExample}, we exhibit the three weightings, $wt_L, wt_T, wt_A$ on matchings in $\pathset{0}{4}$ from the following dissection. See Definitions \ref{def:LocalWeight}, \ref{def:TradWeight}, and \ref{def:AnnulusWeight} respectively. 

\begin{center}
\begin{tikzpicture}[scale = 1]
 \draw (0,0) circle (2\R);
\draw (0,0) circle (.5\R);
\node[scale = 0.8, left] at (120:2.1\R){$v_1$};
\node[scale = 0.8, left] at (180:2\R){$v_2$};
\node[scale = 0.8,left] at (240:2\R){$v_3$};
 \coordinate(A) at (120:2\R);
 \coordinate(B) at (240:2\R);
 \coordinate(C) at (90:.5\R);
 \coordinate(D) at (270:.5\R);
\draw (B) to [out = 100, in = -100] (A);
\draw(A)--(C);
\draw(B) -- (D);
\draw(B) to [out = 30, in = 0, looseness = 3] (C);
\node[circle, fill = black, scale = 0.3] at (180:2\R){};
\node[circle, fill = black, scale = 0.3] at (0:0.5\R){};
\node[] at (160:1.7\R){$\alpha$};
\node[] at (200:.9\R){$\beta$};
\node[] at (0:.9\R){$\gamma$};
\node[] at (75:1.5\R){$\delta$};
  \end{tikzpicture}
\end{center}

We draw the universal cover of this dissection in Figure \ref{fig:WeightingExample}. As an example of our calculations, compare how we calculate $wt_L(\beta\alpha\beta)$, \[
wt_L(\beta\alpha\beta) = U_1(\lambda_{\vert \beta \vert}) U_1(\lambda_{\vert \alpha \vert}) U_1(\lambda_{\vert \beta \vert}) = \sqrt{2} \cdot 1 \cdot \sqrt{2} = 2
\]
with how we compute $wt_T(\beta\alpha\beta)$
\[wt_T(\beta\alpha\beta) = U_2(\lambda_{\vert \beta \vert}) U_1(\lambda_{\vert \alpha \vert}) = (2 - 1)\cdot 1 = 1
\]

Note that, while individual matchings have different weights, $\sum_{w \in \pathset{0}{4}} wt_L(w) = \sum_{w\in\pathset{0}{4}} wt_T(w)$.  We prove in Theorem \ref{thm:EqualWeights} that this will always occur.

 For an example where $wt_T(w)$ and $wt_A(w)$ are distinct, consider the matching $w = \delta\alpha\delta$. We see that $wt_T(\delta\alpha\delta) = U_1(\lambda_{\vert \delta \vert}) U_1(\lambda_{\vert \delta \vert}) U_1(\lambda_{\vert \alpha \vert}) = 1$ since $v_1^0$ and $v_3^0$ are incident to different lifts of $\delta$ while $wt_A(\delta\alpha\delta) = 0$. 

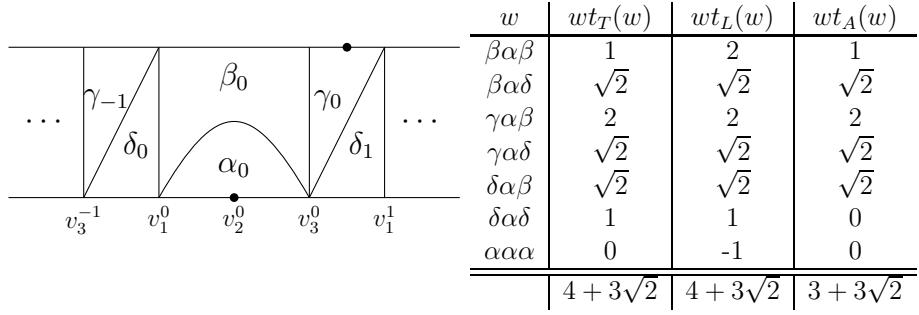
\begin{figure}
    \centering
    \raisebox{-.4\totalheight}{\begin{tikzpicture}[scale = 1]
 \draw(-1,0) -- (5,0);
 \draw(-1,2) -- (5,2);
 \draw (0,0) -- (1,2);
 \draw(1,0) -- (1,2);
 \draw(3,0) -- (3,2);
 \draw(3,0) -- (4,2);
 \draw(4,0) -- (4,2);
 \node[circle, fill = black, scale = 0.3] at (3.5,2){};
 \node[circle, fill = black, scale = 0.3] at (2,0){};
 \draw(0,0) -- (0,2);
 \coordinate(A) at (1,0);
 \coordinate(B) at (3,0);
 \draw (A) to [out = 60, in = 120, looseness = 2] (B);
 \node[below, scale = 0.8] at (0,0) {$v_3^{-1}$};
 \node[below, scale = 0.8] at (1,0) {$v_1^0$};
 \node[below, scale = 0.8] at (2,0) {$v_2^0$};
 \node[below, scale = 0.8] at (3,0) {$v_3^0$};
 \node[below, scale = 0.8] at (4,0) {$v_1^1$};
 \node[] at (2,0.4){$\alpha_0$};
 \node[] at (2,1.6){$\beta_0$};
 \node[] at (0.7,0.7) {$\delta_0$};
 \node[] at (3.7,0.7) {$\delta_1$};
 \node[] at (3.25,1.3) {$\gamma_0$};
 \node[] at (0.3,1.3) {$\gamma_{-1}$};
 \node[] at (-0.5,1) {$\cdots$};
 \node[] at (4.5,1) {$\cdots$};
  \end{tikzpicture}} 
 \scalebox{0.85}{\begin{tabular}{c|c|c|c}
        $w$  & $wt_T(w)$ & $wt_L(w)$ & $wt_A(w)$  \\ \hline
        $\beta \alpha \beta$  & 1 & 2 & 1\\
        $\beta \alpha \delta$  & $\sqrt{2}$  & $\sqrt{2}$ & $\sqrt{2}$\\
        $\gamma \alpha \beta$  & 2 & 2 & 2 \\
        $\gamma \alpha \delta$  & $\sqrt{2}$ & $\sqrt{2}$& $\sqrt{2}$\\
        $\delta \alpha \beta$  & $\sqrt{2}$ & $\sqrt{2}$& $\sqrt{2}$ \\
        $\delta \alpha \delta$  & 1 & 1 & 0 \\
        $\alpha \alpha \alpha$  & 0 & -1 & 0 \\\hline\hline
        & $4 + 3\sqrt{2}$ & $4 + 3\sqrt{2}$ & $3 + 3\sqrt{2}$\\
     \end{tabular}}
    \caption{Exhibiting the three weightings of matchings for the set of matchings between $v_0$ and $v_4$}
    \label{fig:WeightingExample}
\end{figure}

\end{example}

 The following well-known identity of Chebyshev polynomials will be useful for proving our next main result. 

\begin{lemma}\label{lem:ChebyshevProduct}
For all $m,n \geq -1$,
$U_{m+1}(x)U_{n+1}(x) - U_m(x)U_n(x) = U_{m+n+2}(x)$
\end{lemma}

We frame Theorem \ref{thm:EqualWeights} in the context of dissected polygons. If we are working with an annulus or once-punctured disc, the set of subgons between any two vertices in the universal cover forms a dissected polygon. 

\begin{theorem}\label{thm:EqualWeights}
Let $P$ be a polygon with dissection $\D$. Let $v_i,v_j$ be two vertices of $P$, with $v_j$ following $v_i$ in counterclockwise order. Then,
\[
\sum_{w \in \pathset{i}{j}} wt_L(w) = \sum_{w \in \pathset{i}{j}} wt_T(w).
\]
\end{theorem}

\begin{proof}
Let $v_i,v_j$ be vertices of our polygon, $P$, with $i < j$ so that $v_j$ follows $v_i$ in counterclockwise order. We will induct on the number of arcs with both end points strictly between $v_i$ and $v_j$ when traveling counterclockwise. Call the number of such arcs $k$. 

First, consider the case where $k = 0$. Then, we claim that for all matchings $w \in \pathset{i}{j}$, $wt_L(w) = wt_T(w)$. In this case, it is not possible that we use a $p$-subgon $p$ times since there are no subgons with all vertices strictly between $i$ and $j$. It is possible we could choose a $p$-subgon $p-1$ times, but such a matching would have weight 0 under both weightings. Moreover, since there is no subgon contained strictly between $i$ and $j$,  we cannot have a matching which uses the same subgon non-contiguously. Thus, for any matching $w \in \pathset{i}{j}$, $wt_L(w) = wt_T(w)$.

Now suppose we have $k > 0$ arcs which have both end points between $i$ and $j$. In this case, there will be at least one ear between $i$ and $j$. Pick an ear and call it $Q_0$. Let $(a_0,b_0)$ be the unique non-boundary edge from $\D$ which is a side of $Q_0$. Let $Q_1$ be the other subgon which has $(a_0,b_0)$ as one of its sides. 

 
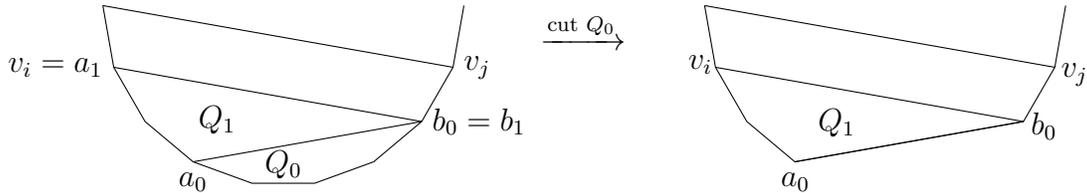
\begin{center}
\begin{tikzpicture}[scale = 2]
\draw (180:1.5\R) -- (200:1.5\R) -- (220:1.5\R) -- (240:1.5\R) -- (260:1.5\R) -- (280:1.5\R) -- (300:1.5\R) -- (320:1.5\R) -- (340:1.5\R) -- (0:1.5\R);
\draw (180:1.5\R) -- (340:1.5\R);
\draw (240:1.5\R) -- (320:1.5\R);
\draw(200:1.5\R) -- (320:1.5\R);
\node[] at (270:1.32\R){$Q_0$};
\node[] at (240:1.1\R){$Q_1$};
\node[left] at (200:1.5\R){$v_i = a_1$};
\node[right] at (340:1.5\R) {$v_j $};
\node[below] at (240:1.5\R) {$a_0$};
\node[right] at (320:1.5\R){$b_0 = b_1$};
\draw[xshift = 5\R] (180:1.5\R) -- (200:1.5\R) -- (220:1.5\R) -- (240:1.5\R) -- (320:1.5\R) -- (340:1.5\R) -- (0:1.5\R);
\draw[xshift = 5\R] (240:1.5\R) -- (320:1.5\R);
\draw[xshift = 5\R] (200:1.5\R) -- (320:1.5\R);
\draw[xshift = 5\R] (180:1.5\R) -- (340:1.5\R);
\node[] at (-12:4.7\R){$Q_1$};
\node[] at (-8:3.5\R){$v_i$};
\node[] at (-19:4.5\R) {$a_0$};
\node[] at (-9:6.4\R){$b_0$};
\node[] at (-5:6.6\R) {$v_j$};
\node[] at (-5:2.5\R) {$\xrightarrow[]{\text{cut } Q_0}$};
\end{tikzpicture}
\end{center}

We partition $\pathset{i}{j}$ based on the choices at $a_0$ and $b_0$. Let $\pathset{i}{j}^{x,y}$ be the subset of $\pathset{i}{j}$ with $\poly{a_0} = Q_x$ and $\poly{b_0} = Q_y$ for $x,y \in \{0,1\}$. We let $\pathset{i}{j}'$ be the complement of $\pathset{i}{j}^{0,0} \cup \pathset{i}{j}^{1,0} \cup \pathset{i}{j}^{0,1} \cup \pathset{i}{j}^{1,1}$ in $\pathset{i}{j}$; in some cases, $\pathset{i}{j}' = \emptyset$. 

Let $\overline{\pathset{i}{J}}$ be the set of matchings in $\overline{P}$, the dissected polygon resulting from cutting $Q_0$ from $P$. 

First, we claim that \begin{equation}\label{eq:CutTradWeight} \sum_{w \in \pathset{i}{j}} wt_T(w) = \sum_{w \in \overline{\pathset{i}{j}}} wt_T(w)
\end{equation}

Since $Q_0$ has $\vert Q_0 \vert -2$ vertices which are not incident  to any other subgon, a matching $w = \poly{i+1}, \ldots, \poly{j-1} \in \pathset{i}{j}$ has $wt_T(w) \neq 0$ only if $w \in \pathset{i}{j}^{1,1}$ or $w \in \pathset{i}{j}'$. Thus, the contribution of $Q_0$ to any nonzero matching is $U_{\vert Q_0 \vert - 2}(\lambda_{\vert Q_0 \vert}) = 1$. This observation provides a weighted (under $wt_T$) bijection between matchings in $\pathset{i}{j}$ which use $Q_0$ exactly $\vert Q_0 \vert - 2$ times and matchings in $\overline{\pathset{i}{j}}$. 

Second, we claim that \begin{equation}\label{eq:CutLocalWeight} \sum_{w \in \pathset{i}{j}} wt_L(w) = \sum_{w \in \overline{\pathset{i}{j}}} wt_L(w)
\end{equation}

 In the following, we will conflate a vertex $v_k$ with its index $k$. For further ease of notation, we will assume that $1 \leq i < j \leq \vert P \vert$. Let $a_1'$ (respectively $b_1'$) be the  first vertex clockwise (counterclockwise) of $a_0$ ($b_0$) which is incident to an arc from $\D$ besides the arc $(a_0,b_0)$. Note that if $a_0$ is incident to multiple arcs in $\D$, then $a_0 = a_1'$, and similarly for $b_1'$. We set $a_1 = a_1'$ if $a_1'$ counterclockwise from $i$ (that is, $a_1' \geq i$); otherwise we set $a_1 = i$. Let $a = a_1 - a_0$. Similarly, set $b_1 = b_1'$ if $b_1'$ is clockwise from $j$ ($b_1' \leq j$) and set $b_1 = j$ otherwise. Let $b = b_1 - b_0$. Note that we must choose $Q_1$ at all vertices strictly between $a_1$ and $a_0$ and at all vertices strictly between $b_0$ and $b_1$, if any such vertices exist.

In order for $w \in \pathset{i}{j}$ to have $wt_L(w) \neq 0$, it must be that $w \in \pathset{i}{j}^{0,0} \cup \pathset{i}{j}^{1,1} \cup \pathset{i}{j}'$. We sum over matchings in each of these subsets. First we just consider matchings in $\pathset{i}{j}^{0,0}$ and $\pathset{i}{j}^{1,1}$.  

\begin{align*}
\sum_{w \in \pathset{i}{j}^{0,0}} wt_L(w) = U_{\vert Q_0 \vert}(\lambda_{\vert Q_0 \vert}) U_a(\lambda_{\vert Q_1 \vert})U_b(\lambda_{\vert Q_1 \vert})\sum_{w \in \pathset{i}{a_1+1}} wt_L(w) \sum_{w \in \pathset{b_1-1}{j}} wt_L(w)\\
= -U_a(\lambda_{\vert Q_1 \vert})U_b(\lambda_{\vert Q_1 \vert})\sum_{w \in \pathset{i}{a_1+1}} wt_L(w) \sum_{w \in \pathset{b_1-1}{j}} wt_L(w)
\end{align*}

\begin{align*}
\sum_{w \in \pathset{i}{j}^{1,1}} wt_L(w) = U_{\vert Q_0 \vert - 2 }(\lambda_{\vert Q_0 \vert}) U_{a+1}(\lambda_{\vert Q_1 \vert})U_{b+1}(\lambda_{\vert Q_1 \vert})\sum_{w \in \pathset{i}{a_1+1}} wt_L(w) \sum_{w \in \pathset{b_1-1}{j}} wt_L(w)\\
= U_{a+1}(\lambda_{\vert Q_1 \vert})U_{b+1}(\lambda_{\vert Q_1 \vert})\sum_{w \in \pathset{i}{a_1+1}} wt_L(w) \sum_{w \in \pathset{b_1-1}{j}} wt_L(w)
\end{align*}

From Lemma \ref{lem:ChebyshevProduct}, we conclude that \[
\sum_{w \in \pathset{i}{j}^{0,0} \cup \pathset{i}{j}^{1,1}} wt_L(w) = U_{a+b+2}(\lambda_{\vert Q_1 \vert})\sum_{w \in \pathset{i}{a_1+1}} wt_L(w) \sum_{w \in \pathset{b_1-1}{j}} wt_L(w) = \sum_{w \in \overline{\pathset{i}{j}^{1,1}}} wt_L(w)
\]
where $\overline{\pathset{i}{j}^{1,1}}$ is the set of matchings between $v_i$ and $v_j$ in $\overline{P}$ which have $\poly{a_0} = \poly{b_0} = Q_1$. 

If $w \in \pathset{i}{j}'$, then $w[a_0:a_0] \neq Q_0$ and $w[b_0:b_0] \neq Q_0$, so $k_{Q_0} = \vert Q_0 \vert - 2$, and $Q_0$ contributes $1$ to $wt_L(w)$. Moreover, since $w[a_0:a_0] \neq Q_1$ and $w[b_0:b_0] \neq Q_1$, $w[a_0:a_0]$ and $w[b_0:b_0]$ cannot be the same subgon (when viewing the matching in the infinite strip). Thus, the local weight of the image of $w$ after cutting $Q_0$ will be the same as the local weight of $w$. We can conclude that \[
\sum_{w \in \pathset{i}{j}'} wt_L(w) = \sum_{w \in \overline{\pathset{i}{j}'}} wt_L(w)
\]

By proving equations \ref{eq:CutTradWeight} and \ref{eq:CutLocalWeight}, we have shown that summing over all matchings with either weighting is unaffected by cutting ears. Thus, we can always reduce to the case where there is no ear between $v_i$ and $v_j$, and in this case, we know the sums of the two weightings are equal. 
\end{proof}

An immediate corollary of Theorem \ref{thm:EqualWeights} is that our frieze pattern entries are given by summing over matchings using the traditional weight.

\begin{corollary}\label{cor:TraditionalWeight}
\begin{enumerate}
    \item Let $\mathcal{F} = \{m_{i,j}\}$ be a finite frieze pattern of Type $\Lambda_{p_1,\ldots,p_s}$ and of width $n+3$ which is realizable by a dissected polygon, $P$. Let $v_1,\ldots,v_n$ be the vertices of $P$.  Then,\[
    m_{i,j} = \sum_{w \in \pathset{i}{j}} wt_T(w)
    \]
    \item Let $\mathcal{F} = \{m_{i,j}\}$ be an infinite, $n$-periodic frieze pattern of Type $\Lambda_{p_1,\ldots,p_s}$ which is realizable by a dissection of $S_n$ or $A_{n,m}$. Let $I$ be the dissected infinite strip which is the universal cover.  Let $\{v_i^k\}$ be the infinite family of vertices of $I$ which map to the vertices on the outer boundary of $S$ under the covering map. Then,\[
    m_{i,j} = \sum_{w \in \pathset{i}{j}} wt_T(w)
    \]
    where $\pathset{i}{j}$ is the set of matchings between $v_{i_n}^{\lfloor \frac{i}{n} \rfloor}$ and $v_{j_n}^{\lfloor \frac{j}{n} \rfloor}$.
\end{enumerate}
\end{corollary}

\section{Growth Coefficients}\label{sec:GrCo}

Baur, Fellner, Parsons, and Tschabold studied the rate of growth of entries in an infinite frieze pattern \cite{2}. Specifically, they found that in an infinite frieze pattern with rows of period $n$, for every $k \in \mathbb{Z}_{\geq 1}$ the difference between row $kn$ and row $kn - 2$ along the columns is constant. 

\begin{theorem}[Growth Coefficients \cite{2}] \label{thm:GrowthCoefficient}
Given an infinite frieze pattern $\mathcal{F} = \{m_{i,j}\}$ with period $n$, for any $k \geq 1$, $m_{i,i+kn + 1} - m_{i+1, i+kn}$ is constant for all $i$. 
\end{theorem}

Let $s_k = m_{0,kn-1} - m_{1,kn}$ be the \emph{$k$-th growth coefficient}. For example, the first growth coefficient of the frieze pattern in Figure \ref{fig:BPTBijection} is $8 -1 = 11 - 4 = 7$ and the first growth coefficient for the frieze pattern in Example \ref{ex:FPfromDissection} is $3 + 3\sqrt{2}$. By the following result, all $s_k$ are determined by  $s_1$; hence, we sometimes refer to $s_1$ as the \emph{principal growth coefficient}. 

\begin{theorem}[\cite{2}]\label{thm:GrowthCoefficientRecurrence}
Let $s_0 = 2$. The growth coefficients, $\{s_k\}$, of an infinite frieze pattern $\mathcal{F}$ satisfy \[
s_{k+1} = s_1s_{k} - s_{k-1}
\]
\end{theorem}

We define a third weighting for matchings on $A_{n,m}$ or $S_n$, working specifically in the case of matchings of length $n$ ( in $\pathset{i}{i+n+2}$); such matchings choose one subgon incident to each vertex on the outer boundary of the surface. 

\begin{definition}[Annulus Weighting]\label{def:AnnulusWeight}
Let $S$ be an annulus or once-punctured disc with dissection $\D$. Let $I$ be the dissected infinite strip with dissection $\widetilde{\D}$ which is the universal cover of $S$. Let $w = \poly{i+1},\ldots,\poly{i+n} \in \pathset{i}{i+n+1}$. For each subgon $\mathcal{P}$ in $\D$, let $N_{\mathcal{P}}$ be the number of times $\poly{k}$, $i+1 \leq k \leq i+n$, is a lift of $\mathcal{P}$.
 If any $N_{\mathcal{P}} > \vert \mathcal{P} \vert -2$, we set $wt_A(w) = 0$. Otherwise, set  \[
wt_A(w) = \prod_{\mathcal{P} \in \D} U_{N_\mathcal{P}}(\lambda_{\vert \mathcal{P} \vert})
\]
\end{definition}

See Example \ref{ex:Matchings} for sample calculations of $wt_A$. We call this weighting the ``annulus'' weighting since it is equivalent to using the rules of the Traditional weighting but working with respect to the surface $S$ corresponding to the frieze pattern instead of the infinite strip. Even though we name this the annulus weighting, it is also valid in a once-punctured disc. The following result shows that the sum of matchings of length $n$ under $wt_A$ is invariant under cyclically shifting the starting point. For the following proof, recall our notation, 
\[
 i_n = \begin{cases} i \pmod{n} & n \nmid i \\
 n & n \mid i \end{cases}.
 \]

\begin{lemma}\label{lem:AnnulusWeightInvariant}
Let $\F = \{m_{i,j}\}$ be an $n$-periodic infinite frieze pattern of Type $\Lambda_{p_1,\ldots,p_s}$ from a dissection $\D$ of $A_{n,m}$ or $S_n$. Let $w = \poly{i+1},\ldots,\poly{i+n} \in \pathset{i}{i+n+1}$, and let $u = \poly{j+1},\ldots,\poly{j+n} \in \pathset{j}{j+n+1}$ be the result of cyclically shifting $w$. Then, $wt_A(w) = wt_A(u)$.

In particular, for any $i,j \in \mathbb{Z}$, \[
\sum_{w \in \pathset{i}{i+n+1}} wt_A(w) = \sum_{u \in \pathset{j}{j+n+1}} wt_A(u)
\]
\end{lemma}

\begin{proof}
 Recall that we define matchings for $A_{n,m}$ and $S_n$ on the infinite strip. Let $w = \poly{i+1},\ldots,\poly{i+n} \in \pathset{i}{i+n+1}$. For each $i+1 \leq k \leq i+n$, $\poly{k}$ is a lift of a polygon incident to $v_{k_n}$ in $\D$; call this subgon  $\overline{\poly{k_n}}$.  From $w$, we can construct a matching $u = Q_{j+1},\ldots,Q_{j+n} \in \pathset{j}{j+n+1}$, by, for all $j+1 \leq \ell \leq j+n$, setting $Q_\ell$ to be the lift of $\overline{\poly{\ell_n}}$ which is incident to $v_\ell$ in the infinite strip.  Then, by construction, for each $\mathcal{P}$ in the dissection $\D$, the number of times we use a lift of $\mathcal{P}$ is the same in $w$ and $u$, so that $wt_A(w) = wt_A(u)$.
 
 Moreover, this map between $w \in \pathset{i}{i+n+1}$ and $u \in \pathset{j}{j+n+1}$ is a bijection. It immediately follows that the sum over these two sets using $wt_A$ is equal.
\end{proof}

Next, we verify that $wt_A$ is unaffected by cutting ears. 

\begin{lemma}\label{lem:wtAEars}
Let $\F = \{m_{i,j}\}$ be an $n$-periodic infinite frieze pattern of Type $\Lambda_{p_1,\ldots,p_s}$ from a dissection $\D$ of $A_{n,m}$ or $S_n$ and suppose the dissection has an ear of size $p$. Let $\F' = \{m'_{i,j}\}$ be the $(n-(p-2))$-periodic frieze pattern from the surface after cutting this ear. Then, for all $i \in \mathbb{Z}$,\[
\sum_{w \in \pathset{i}{i+n+1}} wt_A(w) = \sum_{w \in \overline{\pathset{i}{i + n - (p-2)+1}}} wt_A(w)
\]
where $ \overline{\pathset{i}{i + n - (p-2)+1}}$ is the set of matchings between $v_i$ and $v_{i+n-(p-2)}$ in the surface after cutting this ear.
\end{lemma}

\begin{proof}
Let $S$ be the surface with dissection $\D$ which is the realization of $\F$. Without loss of generality, label the vertices on the outer boundary of the surface $v_1,\ldots,v_n$ such that the ear of size $p$ has vertices $v_\ell,\ldots,v_{\ell + p-1}$ with $1 \leq \ell < \ell + p \leq n$. Call this ear $\widehat{P}$. Then, every matching $w = \poly{1},\ldots,\poly{n} \in \pathset{0}{n+1}$  necessarily has $\poly{\ell+1} = \cdots = \poly{\ell+p-2} = \widehat{P}$. Moreover, in order for $w$ to have $wt_A(w) > 0$, it must be that $\poly{\ell} \neq \widehat{P}$ and $\poly{\ell + p -1} \neq \widehat{P}$. 

Let $\overline{\pathset{0}{n+(p-2)+1}}$ be the set of matchings between $v_1$ and $v_{n+(p-2)}$ in the surface $\overline{S}$ which is the result of cutting the ear $\widehat{P}$. Note that this surface now has $n - (p-2)$ vertices on the outer boundary. We have a bijection between $w = \poly{1},\ldots,\poly{n} \in \pathset{0}{n+1}$ with $\poly{\ell} \neq \widehat{P}$ and $\poly{\ell + p -1} \neq \widehat{P}$, and $u = Q_1,\ldots,Q_{n-(p-2)} \in \overline{\pathset{0}{n+(p-2)+1}}$. Given such a $w$, we set $u = Q_1,\ldots,Q_{n-(p-2)}$ as \[
Q_i = \begin{cases} \poly{i} & i \leq \ell \\ \poly{i+(p-2)} & i > \ell.
\end{cases}
\]

For any subgon $R \neq \widehat{P}$ in $\D$, the number of times $R$ appears in is the same as the number of times $R$ appears in $u$. In $w$, $N_{\widehat{P}} = p-2$ and $U_{p-2}(\lambda_p) = 1 = U_0(\lambda_p)$. Thus, $wt_A(w) = wt_A(u)$. Since we bijectively map every $w \in \pathset{0}{n+1}$ with $wt_A(w) > 0$ to $u \in \overline{\pathset{0}{n+(p-2)+1}}$, the claim follows for $i = 1$. By Lemma \ref{lem:AnnulusWeightInvariant}, the claim is true for arbitrary $i$. 
\end{proof}

With the preceding two results, we can conclude that the annulus weighting gives a combinatorial interpretation of the principal growth coefficient.

\begin{theorem}\label{thm:CombinatorialGrowthC}
Let $\F = \{m_{i,j}\}$ be an $n$-periodic infinite frieze pattern of Type $\Lambda_{p_1,\ldots,p_s}$ from a dissection $\D$ of $A_{n,m}$ or $S_n$. Let $s_1$ be the principal growth coefficient of $\F$. Then, for any $i \in  \mathbb{Z}$,\[
s_1 = \sum_{w \in \pathset{i}{i+n+1}} wt_A(w)
\]
\end{theorem}

\begin{proof}
For shorthand, let $S$ be the surface corresponding to $\F$.  Since $S$ has $n$ vertices on the outer boundary, the rows of $\F$ are necessarily $n$-periodic. Theorem \ref{thm:GrowthCoefficient} guarantees that $m_{i,i+n+1} - m_{i+1,i+n}$ is constant for all $i \in \mathbb{Z}$; by definition, $s_1$ is this constant difference. From Corollary \ref{cor:TraditionalWeight}, we know that $m_{i,i+n+1} = \sum_{w \in \pathset{i}{i+n+1}} wt_T(w)$ and similarly $m_{i+1,i+n} = \sum_{w \in \pathset{i+1}{i+n}} wt_T(w)$. Thus, it is sufficient to prove the following relation amongst the weights of matchings,\begin{equation}\label{eq:GrCo}
\sum_{w \in \pathset{i}{i+n+1}} wt_A(w) = \sum_{w \in \pathset{i}{i+n+1}} wt_T(w) - \sum_{w \in \pathset{i+1}{i+n}} wt_T(w).
\end{equation}

First, we claim that it is sufficient to consider the case of skeletal dissections. From Lemma \ref{lem:wtAEars}, we know that if $\D$ has an ear of size $p$, then $\sum_{w \in \pathset{i}{i+n+1}} wt_A(w) = \sum_{w \in \overline{\pathset{i}{i+n - (p-2) + 1}}} wt_A(w)$ where $\overline{\pathset{i}{i+n - (p-2) + 1}}$ is the set of matchings between $v_i$ and $v_{i+n-(p-2)+1}$ in $\overline{S}$, the resulting surface with dissection after cutting this ear.

Now, we claim that the right-hand side of Equation \ref{eq:GrCo} is also invariant under cutting ears, \begin{equation}\label{eq:TradWeightDiffEar}
\sum_{w \in \pathset{i}{i+n+1}} wt_T(w) - \sum_{w \in \pathset{i+1}{i+n}} wt_T(w) = \sum_{w \in \overline{\pathset{i}{i+n-(p-2)+1}}} wt_T(w) - \sum_{w \in \overline{\pathset{i+1}{i+n-(p-2)}}} wt_T(w).
\end{equation}

For convenience, we relabel the vertices of $S$ so that $i = 0$. Let $Q_0$ be an ear of size $p$ in $\D$. If the boundary edge $(v_n,v_1)$ (the edge starting at $v_n$ and traveling counterclockwise to $v_1$) is not an edge of  $Q_0$, then from Equation \ref{eq:CutTradWeight}, $\sum_{w \in \pathset{0}{n+1}} wt_T(w) =\sum_{w \in \overline{\pathset{0}{n-(p-2)+1}}} wt_T(w) $. This same equation shows $\sum_{w \in \pathset{1}{n}} wt_T(w) =\sum_{w \in \overline{\pathset{1}{n-(p-2)}}} wt_T(w) $ if neither $v_1$ nor $v_n$ are vertices of $Q_0$. 

Next, consider $\sum_{w \in \pathset{1}{n}} wt_T(w)$ where $(v_n,v_1)$ is not an edge of $Q_0$, but $Q_0$ has $v_1$ or $v_n$ as a vertex, or possibly both. Then, the vertices of $Q_0$ are $a,\ldots,a+p$ with $1 \leq a \leq n-p$. Every $w = \poly{2},\ldots,\poly{n-1} \in \pathset{1}{n}$ with $wt_T(w)$ necessarily has $\poly{a+1} = \cdots = \poly{a+p-1} = Q_0$. In order for $wt_T(w) \neq 0$, if $a > 1$, $\poly{a} \neq Q_0$ and if $a+p < n$, $\poly{a+p} \neq Q_0$. Each such matching can be mapped with a weight-preserving bijection to a matching in $\overline{\pathset{1}{n - (p-2)}}$ in $S$ after removing the ear $Q_0$. Note that if $v_1$ and $v_n$ are both vertices of $Q_0$, there is exactly one matching in $\pathset{1}{n}$, $w = Q_0,\ldots,Q_0$ which is length  $n-2 = \vert Q_0  \vert - 2$ times. It follows that $wt_L(w) = 1$. Moreover, since in this case $n = p$, we have that $\overline{\pathset{1}{n-(p-2)}} = \overline{\pathset{1}{n}} = \{\emptyset\} $. Since $wt_L(\emptyset) = 1$, the claim still holds. 

Now, assume that one of the edges of $Q_0$ is the boundary edge $(v_n,v_1)$. Let $(v_b,v_a)$ be the unique non-boundary edge from $\D$ which is an edge of $Q_0$.  Note that the vertices of $Q_0$ are, in counterclockwise order, $v_b,v_{b+1},\ldots,v_n,v_1,\ldots,v_a$ for $b \leq n$ and $a \geq 1$; in particular, $Q_0$ is a $(a + (n-b+1))$-gon. The vertices of $S$ after cutting $Q_0$ are $v_a,v_{a+1},\ldots,v_b$. 

First, let $a > 1$ and $b < n$. Necessarily, all $w = \poly{1},\ldots,\poly{n} \in \pathset{0}{n+1}$ have $\poly{1} = \cdots \poly{a-1} = Q_0$ and $\poly{b+1} = \cdots = \poly{n} = Q_0$. The subgons in these two sequences are distinct in the infinite strip $I$ but not distinct in the surface $S$.

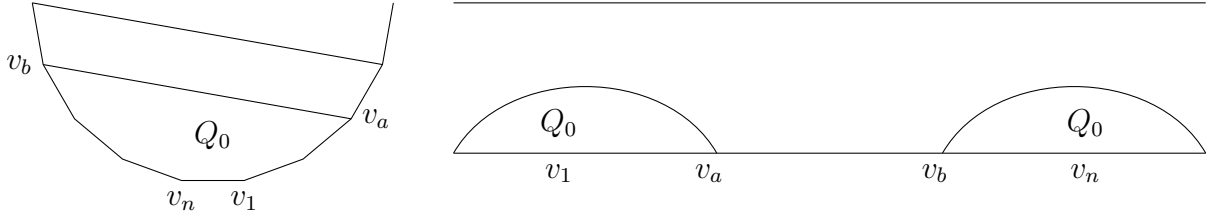
\begin{center}
\begin{tikzpicture}[scale = 2]
\draw (180:1.5\R) -- (200:1.5\R) -- (220:1.5\R) -- (240:1.5\R) -- (260:1.5\R) -- (280:1.5\R) -- (300:1.5\R) -- (320:1.5\R) -- (340:1.5\R) -- (0:1.5\R);
\draw (200:1.5\R) -- (320:1.5\R);
\draw (180:1.5\R) -- (340:1.5\R);
\node[] at (270:1.1\R){$Q_0$};
\node[below] at (260:1.5\R){$v_n$};
\node[below] at (280:1.5\R) {$v_1$};
\node[left] at (200:1.5\R) {$v_b$};
\node[right] at (320:1.5\R){$v_a$};
\draw[xshift = 2\R] (0,0) -- (5,0);
\draw[xshift = 2\R] (0,-1) -- (5,-1);
\node[below, xshift = 2\R] at (1.5,-1) {$v_1$};
\node[below,xshift = 2\R] at (2.5,-1){$v_a$};
\node[below,xshift = 2\R] at (4,-1){$v_b$};
\node[below,xshift = 2\R] at (5,-1){$v_n$};
\draw[xshift = 2\R] (1.75,-1) to [ out = 120, in = 60] (0,-1);
\draw[xshift = 2\R] (3.25,-1) to [ out = 60, in = 120] (5,-1);
\node[xshift = 2\R] at (1.5,-0.8){$Q_0$};
\node[xshift = 2\R] at (5,-0.8) {$Q_0$};
\end{tikzpicture}
\end{center}

 Let $\pathset{0}{n+1}^{0,0}$ be the set of matchings with $\poly{a} = Q_0$ and $\poly{b} = Q_0$. Let $\pathset{0}{n+1}^{0,1}$ be the set of matchings with $\poly{a} = Q_0$ and $\poly{b} \neq Q_0$. Define $\pathset{0}{n+1}^{1,0}$ and $\pathset{0}{n+1}^{1,1}$ similarly.   We define $\pathset{1}{n}^{x,y}$ similarly. (Note this is slightly different than how these indices were used in the proof of Theorem \ref{thm:EqualWeights}.)

Now we can expand the lefthand side of Equation \ref{eq:TradWeightDiffEar} using these partitions of each set of matchings and use Lemma \ref{lem:ChebyshevProduct} on like terms. First, we look at the terms which have $\poly{a} = \poly{b} = Q_0$. 

\begin{align*}
\sum_{w \in \pathset{0}{n+1}^{0,0}} wt_T(w) - \sum_{w \in \pathset{1}{n}^{0,0}} wt_T(w) \\= U_a(\lambda_{\vert Q_0\vert})U_{n-b+1}(\lambda_{\vert Q_0\vert}) \sum_{w \in \pathset{a}{b}} wt_T(w) - U_{a-1}(\lambda_{\vert Q_0\vert})U_{n-b}(\lambda_{\vert Q_0\vert}) \sum_{w \in \pathset{a}{b}} wt_T(w) \\
= U_{a+(n-b+1)}(\lambda_{\vert Q_0\vert})\sum_{w \in \pathset{a}{b}}wt_T(w) = - \sum_{w \in \pathset{a}{b}}wt_T(w)
\end{align*}

We do similar computations for the other sets $\pathset{0}{n+1}^{x,y}$ and $\pathset{1}{n}^{x,y}$. Let $\pathset{a}{b+1}^-$ be the set of matchings which do not use $Q_0$ at $v_b$. 
\begin{align*}
 \sum_{w \in \pathset{0}{n+1}^{0,1}} wt_T(w) - \sum_{w \in \pathset{1}{n}^{0,1}} wt_T(w) \\= U_a(\lambda_{\vert Q_0\vert})U_{n-b}(\lambda_{\vert Q_0\vert}) \sum_{w \in \pathset{a}{b+1}^-} wt_T(w) - U_{a-1}(\lambda_{\vert Q_0\vert})U_{n-b-1}(\lambda_{\vert Q_0\vert}) \sum_{w \in \pathset{a}{b+1}^-} wt_T(w) \\
= U_{a+(n-b)}(\lambda_{\vert Q_0\vert})\sum_{w \in w \in \pathset{a}{b+1}^-}wt_T(w) = 0
\end{align*}

We have $U_{a+(n-b)}(\lambda_{\vert Q_0 \vert}) = 0$ since $\vert Q_0 \vert = a + (n-b + 1)$. We can analogously conclude \begin{align*}
  \sum_{w \in \pathset{0}{n+1}^{1,0}} wt_T(w) - \sum_{w \in \pathset{1}{n}^{1,0}} wt_T(w) = 0   
\end{align*}

Finally, we consider $\pathset{0}{n+1}^{1,1}$ and $\pathset{1}{n}^{1,1}$. Note that the set of matchings in $\pathset{a-1}{b+1}$ which do not use $Q_0$ at $v_a$ or $v_b$ is equivalent to $\overline{\pathset{b}{a+n-(p-2)}}$, the set of matchings between $v_b$ and $v_a$ in $\overline{S}$ after cutting $Q_0$. This is true since in $\overline{S}$, $v_a$ and $v_b$ are neighbors. Then, \begin{align*}
  \sum_{w \in \pathset{0}{n+1}^{1,1}} wt_T(w) - \sum_{w \in \pathset{1}{n}^{1,1}} wt_T(w) \\= U_{a-1}(\lambda_{\vert Q_0\vert})U_{n-b}(\lambda_{\vert Q_0\vert}) \sum_{w \in \overline{\pathset{b}{a}}} wt_T(w) - U_{a-1}(\lambda_{\vert Q_0\vert})U_{n-b-1}(\lambda_{\vert Q_0\vert}) \sum_{w \in \overline{\pathset{b}{a}}} wt_T(w) \\
= U_{a+(n-b)-1}(\lambda_{\vert Q_0\vert})\sum_{w \in  \overline{\pathset{b}{a+n-(p-2)}}}wt_T(w) = \sum_{w \in  \overline{\pathset{b}{a+n-(p-2)}}}wt_T(w)  .
\end{align*}

Since $\pathset{0}{n+1}$ is partitioned by the sets $ \pathset{0}{n+1}^{x,y}$ for $x,y \in \{0,1\}$ and similarly for $\pathset{1}{n}$, we conclude that \begin{align*}
   \sum_{w \in \pathset{0}{n+1}} wt_T(w) - \sum_{w \in \pathset{1}{n}} wt_T(w) = \sum_{w \in  \overline{\pathset{b}{a+n-(p-2)}}}wt_T(w)  - \sum_{w \in \pathset{a}{b}}wt_T(w)\\
   = \sum_{w \in  \overline{\pathset{b}{a+n-(p-2)}}}wt_T(w)  - \sum_{w \in \overline{\pathset{a}{b}}}wt_T(w).
\end{align*}

where $\pathset{a}{b} = \overline{\pathset{a}{b}}$  since $Q_0$ is not incident to any vertices strictly between $v_a$ and $v_b$ when traveling counterclockwise. This is the desired result since the vertices remaining after cutting $Q_0$ are  $v_a,v_{a+1},\ldots,v_{b-1},v_b$.

Now we turn to the $a = 1$ or $b = n$ case. Suppose that $a =1$. We cannot also have $b = n$ simultaneously since $(v_b,v_a)$ is a non-boundary arc incident to $Q_0$. Note that in this case, $\vert Q_0 \vert = n-b+2$. 

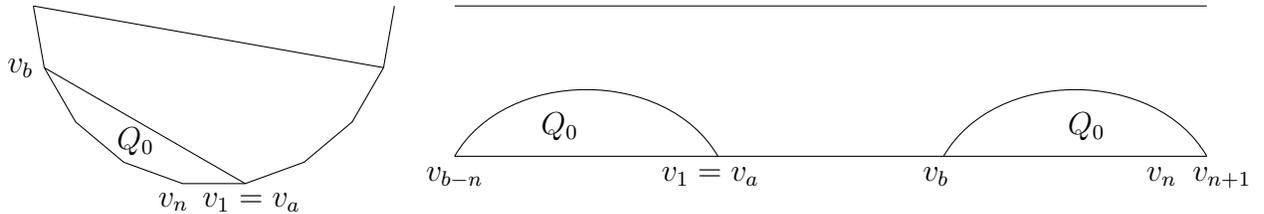
\begin{center}
\begin{tikzpicture}[scale = 2]
\draw (180:1.5\R) -- (200:1.5\R) -- (220:1.5\R) -- (240:1.5\R) -- (260:1.5\R) -- (280:1.5\R) -- (300:1.5\R) -- (320:1.5\R) -- (340:1.5\R) -- (0:1.5\R);
\draw (200:1.5\R) -- (280:1.5\R);
\draw (180:1.5\R) -- (340:1.5\R);
\node[] at (240:1.3\R){$Q_0$};
\node[below] at (257:1.5\R){$v_n$};
\node[below] at (282:1.5\R) {$v_1 = v_a$};
\node[left] at (200:1.5\R) {$v_b$};
\draw[xshift = 2\R] (0,0) -- (5,0);
\draw[xshift = 2\R] (0,-1) -- (5,-1);
\node[below,xshift = 2\R] at (2.5,-1){$v_1 = v_a$};
\node[below,xshift = 2\R] at (4,-1){$v_b$};
\node[below,xshift = 2\R] at (5.5,-1){$v_n$};
\node[below,xshift = 2\R] at (5.9,-1){$v_{n+1}$};
\node[below,xshift = 2\R] at (0.8,-1) {$v_{b - n}$};
\draw[xshift = 2\R] (1.75,-1) to [ out = 120, in = 60] (0,-1);
\draw[xshift = 2\R] (3.25,-1) to [ out = 60, in = 120] (5,-1);
\node[xshift = 2\R] at (1.5,-0.8){$Q_0$};
\node[xshift = 2\R] at (5,-0.8) {$Q_0$};
\end{tikzpicture}
\end{center}

In this case, since $v_{b+1},\ldots,v_n$ are only incident to $Q_0$, any $w = \poly{1},\ldots,\poly{n} \in \pathset{0}{n+1}$  has $\poly{b+1} = \cdots = \poly{n} = Q_0$. In order for $wt_T(w) > 0$, it must be that $\poly{b} \neq Q_0$. Thus, using the same notation as above, the only subsets of $\pathset{0}{n+1}$ which contain matchings with nonzero weight are $\pathset{0}{n+1}^{0,1}$ and $\pathset{0}{n+1}^{1,1}$. We sum over all matchings with nonzero traditional weight in $\pathset{0}{n+1}$, again letting $\pathset{1}{b+1}^-$ be the subset of $\pathset{1}{b+1}$ which does not choose $Q_0$ at $b$, and similarly for $\pathset{0}{b+1}^-$.

\begin{align*}
  \sum_{w \in \pathset{0}{n+1}} wt_T(w)  =    \sum_{w \in \pathset{0}{n+1}^{0,1}} wt_T(w) + \sum_{w \in \pathset{0}{n+1}^{1,1}} wt_T(w)\\
  = U_1(\lambda_{\vert Q_0 \vert})U_{ (n-b)}(\lambda_{\vert Q_0 \vert}) \sum_{v \in \pathset{1}{b+1}^-} wt_T(v) + U_{(n-b) }(\lambda_{\vert Q_0 \vert}) \sum_{v \in \pathset{0}{b+1}^-} wt_T(v)\\
  = \lambda_{\vert Q_0 \vert} \sum_{v \in \pathset{1}{b+1}^-} wt_T(v) + \sum_{v \in \pathset{0}{b+1}^-} wt_T(v)\\
\end{align*}

Since $a = 1$, matchings in $\pathset{1}{n}$ no longer choose a subgon at $v_a$. Thus, we will alter our notation. Let $\pathset{1}{n}^0$ be the subset of matchings $w = \poly{2},\ldots,\poly{n-1}$ with $\poly{b} = Q_0$ and we define $\pathset{1}{n}^1$ to be the subset of matchings with $\poly{b} \neq Q_0$. 

\begin{align*}
  \sum_{w \in \pathset{1}{n}} wt_T(w)  =    \sum_{w \in \pathset{1}{n}^0} wt_T(w) + \sum_{w \in \pathset{1}{n}^1} wt_T(w)\\
  = U_{(n-b)}(\lambda_{\vert Q_0 \vert}) \sum_{v \in \pathset{1}{b}} wt_T(v) + U_{(n-b) - 1}(\lambda_{\vert Q_0 \vert}) \sum_{v \in \pathset{1}{b+1}^-} wt_T(v)\\
  =  \sum_{v \in \pathset{1}{b}} wt_T(v) + \lambda_{\vert Q_0 \vert} \sum_{v \in \pathset{1}{b+1}^-} wt_T(v)\\
\end{align*}

where by Lemma \ref{lem:ChebyshevGeometric}, $ U_{ (n-b) - 1}(\lambda_{\vert Q_0 \vert}) = U_1(\lambda_{\vert Q_0\vert}) = \lambda_{\vert Q_0 \vert}$. 

Thus, in the $a = 1$ case, we conclude \begin{align*}
 \sum_{w \in \pathset{0}{n+1}} wt_T(w) - \sum_{w \in \pathset{1}{n}} wt_T(w)  =  \sum_{v \in \pathset{0}{b+1}^-} wt_T(v) - \sum_{v \in \pathset{1}{b}} wt_T(v)\\
 = \sum_{w \in  \overline{\pathset{b}{1}}}wt_T(w)  - \sum_{w \in \overline{\pathset{1}{b}}}wt_T(w).
\end{align*}

The $b = n$ case is identical. 

Now we know both sides of the desired equality in Equation \ref{eq:GrCo} are invariant under removing ears from the dissection. Hence, it suffices to consider how $wt_T$ and $wt_A$ compare in a skeletal dissection.

We will keep our vertex labels so that we consider $\sum_{w \in \pathset{0}{n+1}} wt_A(w)$. Let $\hat{P}$ be the subgon which contains the boundary edge $(v_n,v_1)$. Given a matching $w = \poly{1},\ldots,\poly{n} \in \pathset{0}{n+1}$, we only have $wt_T(w) \neq wt_A(w)$ if $\poly{1} = \poly{n} = \hat{P} $. This is the case since, for any other subgon $Q$ in $\D$, there is only one lift of $Q$ in the infinite strip between $v_1$ and $v_n$; in the notation of Definitions \ref{def:TradWeight} and \ref{def:AnnulusWeight}, $k_Q = N_Q$. However, in the infinite strip, $v_1$ and $v_n$ are adjacent to separate lifts of $\hat{P}$. 

Let $\widehat{\pathset{0}{n+1}} = \{w = \poly{1},\ldots,\poly{n} \in \pathset{0}{n+1} : \poly{1} = \poly{n} = \hat{P}\}$. Note that in some dissections $\widehat{\pathset{0}{n+1}} = \pathset{0}{n+1}$. We have a bijection, $\varphi$, between $\widehat{\pathset{0}{n+1}}$ and $\pathset{1}{n}$, given by sending $w \in\widehat{\pathset{0}{n+1}}$ to $w[2:n-1]$. This bijection is not necessarily weight-preserving under any weighting.

Recall that, given $w = \poly{i+1},\ldots,\poly{j-1}$, for $i+1 \leq a \leq b \leq j-1$, $w[a:b] = \poly{a},\ldots,\poly{b}$. For each $w \in \widehat{\pathset{0}{n+1}}$, let $a_w \geq 1$ be the largest possible integer such that all subgons in $w[1:a_w]$ are $\hat{P}$. Similarly, define $b_w \leq n$ to be the smallest possible integer such that all subgons in $w[b_w:n]$  are $\hat{P}$. Then, for each $w \in \widehat{\pathset{0}{n+1}}$, $wt_T(w) = U_{a_w}(\lambda_{\vert \hat{P}\vert}) U_{n-b_w + 1}(\lambda_{\vert \hat{P}\vert})wt_T(w[a_w+1,b_w-1])$ and $wt_T(\varphi(w)) = U_{a_w-1}(\lambda_{\vert \hat{P}\vert}) U_{n-b_w}(\lambda_{\vert \hat{P}\vert})wt_T(w[a_w+1,b_w-1])$. This allows us to reduce the following, using Lemma \ref{lem:ChebyshevProduct}, \begin{align*}
\sum_{w \in \widehat{\pathset{0}{n+1}}} wt_T(w) - \sum_{w \in \pathset{1}{n}} wt_T(w)\\
=  \sum_{w \in \widehat{\pathset{0}{n+1}}} \bigg(U_{a_w}(\lambda_{\vert \hat{P}\vert}) U_{n-b_w + 1}(\lambda_{\vert \hat{P}\vert})  - U_{a_w-1}(\lambda_{\vert \hat{P}\vert}) U_{n-b_w}(\lambda_{\vert \hat{P}\vert})\bigg) wt_T(w[a_w+1,b_w-1]) \\
=  \sum_{w \in \widehat{\pathset{0}{n+1}}} U_{a_w + (n-b_w + 1)}(\lambda_{\vert \hat{P}\vert})wt_T(w[a_w+1,b_w-1])
\end{align*}

Since $\hat{P}$ is not an ear, $U_{a_w + (n-b_w+1)}(\lambda_{\vert \hat{P}\vert}) \geq 0$.

Let $\pathset{0}{n+1}'$ be the complement of $\widehat{\pathset{0}{n+1}}$ in $\pathset{0}{n+1}$. Then, we can conclude that \begin{align*}
\sum_{w \in\pathset{0}{n+1}} wt_T(w) - \sum_{w \in \pathset{1}{n}} wt_T(w)  \\ =  \sum_{w \in \pathset{0}{n+1}'} wt_T(w) + \sum_{w \in \widehat{\pathset{0}{n+1}}} U_{a_w + (n-b_w + 1)}(\lambda_{\vert \hat{P}\vert})wt_T(w[a_w+1,b_w-1])
\end{align*}

As discussed, if $w \in \pathset{0}{n+1}'$, $wt_T(w) = wt_A(w)$. If $w \in \widehat{\pathset{0}{n+1}}$, so that for some integers $a_w\geq 1, b_w \leq n$, the submatchings $w[1,a_w]$ and $w[b_w,n]$ both only use the subgon $\hat{P}$, then $wt_A(w) = U_{a_w + (n-b_w + 1)}(\lambda_{\vert \hat{P}\vert})wt_T(w[a_w+1,b_w-1])$. Thus, we have shown Equation \ref{eq:GrCo}.
\end{proof}

Knowledge of the growth coefficient allows us to refine Step 5 of our algorithm in Section \ref{subsec:RealizabilityAlgorithm}. If a skeletal frieze pattern satisfies Proposition \ref{prop:realizability}, then it is either realizable by a dissection of $A_{n,m}$ or $S_n$. In order to determine which is the correct surface, one needed to check the conditions listed in Lemma \ref{lem:SkeletalSn}. Checking the first growth coefficient is an alternative way to differentiate between frieze patterns of these two surfaces.

\begin{proposition}\label{prop:GrCo2}
An infinite frieze pattern $\mathcal{F}$ of Type $\Lambda_{p_1,\ldots,p_s}$ which satisfies Proposition \ref{prop:realizability} is realizable by a dissection of $S_n$ if and only if its first growth coefficient is 2.
\end{proposition}

\begin{proof}
Note that if $wt_A(w) \neq 0$, then by Lemma \ref{lem:ChebyshevGeometric}, $wt_A(w) \geq 1$. Therefore, if a frieze pattern has growth coefficient 2, then it has one or two matchings which have nonzero weight under $wt_A$. We claim that every skeletal dissection has at least two matchings which satisfy $wt_A(w) > 0$. We establish $w_{CW}$ to be the matching where each vertex uses the clockwise-most subgon it is incident to. In a skeletal dissection, a subgon of size $p$ can have up to $p-1$ vertices incident to the outer boundary. However, no subgon of size $p$ would appear $p-1$ times in $w_{CW}$. Thus, $wt_A(w_{CW}) > 0$. We can similarly define $w_{CCW}$ by choosing the counterclockwise-most subgon at each vertex. Since $w_{CW}$ and $w_{CCW}$ are distinct matchings, we know that for any infinite, realizable, and skeletal frieze pattern $\mathcal{F}$, $s_1 \geq 2$. 

In order for $s_1 = 2$, it must be that $wt_A(w_{CW}) = wt_A(w_{CCW}) = 1$ and for any other matching $w \in \pathset{i}{i+n+1}, wt_A(w) = 0$. This would require that, if $P$ is a subgon of the dissection, then $P$ has $\vert P \vert  - 1$ vertices on the outer boundary. Otherwise, we would have more options for nonzero matchings. Such a dissection is only possible on $S_n$. 

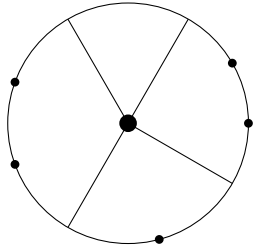
\begin{center}
\begin{tikzpicture}[scale = 1]
 \draw (0,0) circle (2\R);
\node[circle, fill = black, scale = 0.6] at (180:0\R){};
\node[circle, fill = black, scale = 0.3] at (0:2\R){};
\node[circle, fill = black, scale = 0.3] at (30:2\R){};
\node[circle, fill = black, scale = 0.3] at (285:2\R){};
\node[circle, fill = black, scale = 0.3] at (160:2\R){};
\node[circle, fill = black, scale = 0.3] at (200:2\R){};
\draw (120:2\R) to (0:0);
\draw (240:2\R) to (0:0);
\draw (60:2\R) to (0,0);
\draw(-30:2\R) to (0,0);
  \end{tikzpicture}
  \end{center}

For the converse direction, note that every skeletal dissection of $S_n$ looks as above; in the triangulation case, this is sometimes called a ``wheel''. When forming a matching of a skeletal dissection of $S_n$, once we choose whether to use the clockwise or counterclockwise most subgon at a vertex of degree 2 in a matching, the rest of the matching is forced since we must avoid using $p-1$ vertices of a size $p$ subgon. Such a matching uses a size $p$ subgon $p-2$ times, and $U_{p-2}(\lambda_p) = 1$; therefore, this matching will have weight 1.
\end{proof}

By Lemma \ref{lem:wtAEars}, the growth coefficient of a realizable frieze pattern is unaffected by cutting or gluing ears. Thus, we conclude with a more general Corollary. 

\begin{corollary}\label{cor:GrCoAllFrieze}
Let $\mathcal{F}$ be a frieze pattern of Type $\Lambda_{p_1,\ldots,p_s}$ which terminates in step 5 of the Realizability Algorithm. Then, $\mathcal{F}$ is realizable by a dissection of $S_n$ if and only if first growth coefficient of $\mathcal{F}$ is 2. 
\end{corollary}

We have thus far only provided a combinatorial interpretation of the first growth coefficient for a frieze pattern $\mathcal{F}$ realizable by a dissection $\D$ of $S_n$ or $A_{n,m}$. For higher growth coefficients, one could consider $wt_A$ on matchings in the ``$k$-th power of $\D$''.

\begin{definition}\label{def:kthpower}
Let $\D$ be a dissection of $A_{n,m}$ or $S_n$, and let $I$ be the infinite strip which is the universal cover of the surface, with fundamental domain $F$. For $k \geq 1$, we construct \emph{the $k^{th}$ power of $\D$ on $A_{kn,km}$ or $S_{kn}$} by considering the dissected surface resulting in taking $k$-consecutive copies of $F$ as a fundamental domain. 
\end{definition}

\begin{example}
The righthand side is the second power of the dissection on the lefthand side.
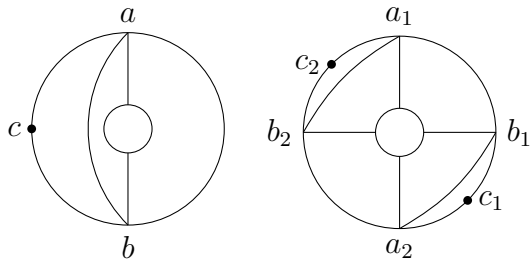
\begin{center}
\begin{tabular}{cc}
\begin{tikzpicture}[scale = 0.8]
\draw (0,0) circle (2\R);
\draw (0,0) circle (0.5\R);
\coordinate(A) at (90:2\R);
\node[above] at (90:2\R){$a$};
\node[circle, fill = black, scale = 0.3] at (180:2\R){};
\node[left] at (180:2\R){$c$};
\coordinate(C) at (-90:2\R);
\node[below] at (-90:2\R){$b$};

\coordinate(B) at (-90:0.5\R);

\coordinate(E) at (90:0.5\R);

\draw(90:2\R) -- (90:0.5\R);
\draw(C) to [out = 135, in = -135](A);
\draw(C) to [out = 90, in = -90](B);
\end{tikzpicture}
&
\begin{tikzpicture}[scale = 0.8]
\draw (0,0) circle (2\R);
\draw (0,0) circle (.5\R);
\coordinate(A1) at (90:2\R);
\coordinate(A2) at (-90:2\R);
\node[above] at (90:2\R){$a_1$};
\node[circle, fill = black, scale = 0.3] at (135:2\R){};
\node[circle, fill = black, scale = 0.3] at (-45:2\R){};
\coordinate(C1) at (-45:2\R);
\node[below] at (-90:2\R){$a_2$};
\coordinate(B1) at (0:2\R);
\node[right] at (0:2\R){$b_1$};
\node[right] at (-45:2\R){$c_1$};
\coordinate(B2) at (180:2\R);
\node[left] at (180:2\R){$b_2$};
\node[left] at (135:2\R){$c_2$};
\coordinate(C2) at (135:2\R);
\coordinate(a) at (90:.5\R);
\coordinate(b) at (0:.5\R);
\coordinate(c) at (-90:.5\R);
\coordinate(d) at (180:.5\R);
\draw (A1) -- (a);
\draw (B1) -- (b);
\draw (A2) -- (c);
\draw (B2) -- (d);
\draw(A1) to [out= -150, in = 60](B2);
\draw(A2) to [out = 30, in = -120] (B1);
\end{tikzpicture}
\end{tabular}
\end{center}
\end{example}

If we instead consider $\F$ to have period $kn$, then the first growth coefficient is the same as the $k$-th growth coefficient of $\F$ when viewing it as $n$-periodic. Thus, we can recover higher growth coefficients of $\F$, realizable by $\D$, by using the $k$-th power of $\D$ and Theorem \ref{thm:CombinatorialGrowthC}.

\begin{corollary}
Let $\mathcal{F}$ be an infinite frieze pattern of Type $\Lambda_{p_1,\ldots,p_s}$ which is realizable by a dissection $\D$ of $A_{n,m}$ or $S_n$. Then for any $i \in \mathbb{Z}$,  \[
s_k = \sum_{w \in \pathset{i}{i+kn+1}} wt_A(w)
\]
where we sum over matchings in the $k$-th power of $\D$. 
\end{corollary}

\subsection{Comparing Inner and Outer Frieze Patterns}\label{subsec:InAndOutFP}

Given a dissected annulus, we could consider the frieze pattern determined from the inner boundary instead of the outer boundary. To differentiate these, given a dissected annulus $A_{n,m}$, let $\mathcal{F}_{out}$ be the frieze pattern determined with respect to the outer boundary and let $\mathcal{F}_{in}$ be the frieze pattern determined with respect to the inner boundary. For example, for the dissection below, $\mathcal{F}_{out}$ is generated by the quiddity cycle $(1+2\sqrt{2},2 + 2\sqrt{2})$ and $\mathcal{F}_{in}$ is generated by $(1+2\sqrt{2}, \sqrt{2},2\sqrt{2},2\sqrt{2},\sqrt{2},1+\sqrt{2})$.

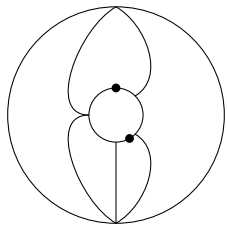
\begin{center}
\begin{tikzpicture}[scale = 0.9]
\draw (0,0) circle (2\R);
\draw (0,0) circle (0.5\R);
\coordinate(A) at (90:2\R);
\coordinate(B) at (-90:2\R);
\coordinate(a) at (180:0.5\R);
\coordinate(b) at (45:0.5\R);
\coordinate(c) at (-45:0.5\R);
\coordinate(d) at (-90:0.5\R);
\node[circle, fill = black, scale = 0.3] at (90:0.5\R){};
\node[circle, fill = black, scale = 0.3] at (-60:0.5\R){};
\draw(A) to [out = -150, in = 180](a);
\draw(B) to [out = 150, in = 180](a);
\draw(A) to [out = -30, in = 30](b);
\draw(B) to [out = 30, in = -30](c);
\draw(B)--(d);
\end{tikzpicture}
\end{center}

The goal in this section is to use our matching formula to give a combinatorial proof that, given a dissected annulus, the first growth coefficients of $\mathcal{F}_{in}$ and $\mathcal{F}_{out}$ are equal. By Theorem \ref{thm:GrowthCoefficientRecurrence}, it will follow that all growth coefficients of these two frieze patterns are equal. See \cite{3} for a module-theoretic approach to this problem for triangulated annuli. 

First, we record a useful lemma. This lemma is a consequence of the fact that all finite frieze patterns have glide symmetry, shown by Coxeter in \cite{11}. 

\begin{lemma}\label{lem:EitherDirection}
Let $\mathcal{F} = \{m_{i,j}\}$ be a finite frieze pattern of Type $\Lambda_{p_1,\ldots,p_s}$ which comes from a dissected polygon. Then, $m_{i,j} = m_{j,i+n}$, implying that \[
\sum_{w \in \pathset{i}{j}} wt_T(w) = \sum_{w \in \pathset{j}{i+n}} wt_T(w).
\]
\end{lemma}

\begin{remark}
\begin{enumerate}
    \item Another way to think about Lemma \ref{lem:EitherDirection} is that the weighted sum of matchings between vertices $v_i$ and $v_j$ is the same going clockwise or counterclockwise.
    \item We can strengthen Lemma \ref{lem:EitherDirection} to a weighted bijection between matchings in $\pathset{i}{j}$ and in $\pathset{i}{j}$ with nonzero $wt_T$. Given a dissection of a polygon $P$ into subgons $\poly{1},\ldots,\poly{k}$, form a multiset $S$ by taking $\vert \poly{i}\vert - 2$ copies of $\poly{i}$ for all $1 \leq i \leq k$. Then, the multiset of subgons used in any matching $w \in \pathset{i}{j}$ such that $wt_T(w) > 0$ is a subset of $S$. Call this subset $T_w$. One can show there is a unique matching $\bar{w} \in \pathset{j}{i}$ which uses the subgons in $S - T_w$. By repeated application of Lemma \ref{lem:ChebyshevGeometric}, these two matchings will always have the same weight. 
\end{enumerate}
\end{remark}

We are ready for the main result in this subsection.

\begin{theorem}\label{theorem:EqualGrCo}
Let $A_{n,m}$ be an annulus with dissection $\D$. Let $\mathcal{F}_{in}$ and $\mathcal{F}_{out}$ be the infinite frieze patterns with respect to the inner and outer boundaries respectively. Let $s_1^{in}$ and $s_1^{out}$ be the first growth coefficients of  $\mathcal{F}_{in}$ and $\mathcal{F}_{out}$ respectively. Then $s_1^{in}= s_1^{out}$.
\end{theorem}

\begin{proof}
From Lemma \ref{lem:wtAEars}, the sum $\sum_{w \in \pathset{i}{i+n+1}} wt_A(w)$ is unchanged by cutting ears from the surface. Therefore, we can consider only the case where $\D$ is skeletal.

Recall that in our construction for the infinite cover of a dissection of $A_{n,m}$ in Section \ref{subsec:InfStrip}, we picked a bridging arc $\tau$, cut along this arc to form a dissected $(n+m+2)$-gon, then glued copies of this subgon together along $\tau$ to form a dissected infinite strip. Call the dissected polygon $F$. Let $v_1,w_1$ be the end points of $\tau$ on the outer and inner boundary respectively, and recall our labeling of vertices of $F$ as below. 

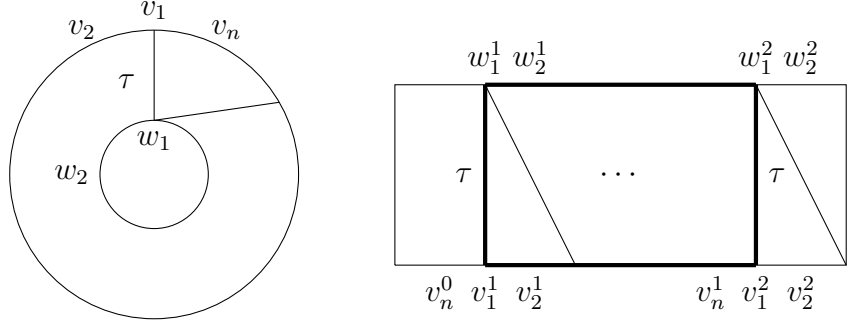
\begin{center}
\begin{tabular}{lccr}
    \begin{tikzpicture}[scale = 1.2]
\draw (0,0) circle (2\R);
\draw (0,0) circle (0.75\R);
\node[above] at (60:2\R){$v_n$};
\node[above] at (90:2\R){$v_1$};
\node[above] at (120:2\R){$v_2$};
\node[below] at (90:.75\R){$w_1$};
\node[left] at (180:.75\R){$w_2$};
\node[left] at (95:1.3\R){$\tau$};

\draw(90:2\R) -- (90:0.75\R);
\draw(30:2\R) -- (90:0.75\R);

\end{tikzpicture}
&&&
    \begin{tikzpicture}[scale = 1.2]
    \draw(0,0) -- (0,2) -- (5,2) -- (5,0) -- (0,0);
    \draw[ultra thick](1,0) -- (4,0);
    \draw[ultra thick] (1,2) -- (4,2);
    \draw[ultra thick](1,0) to node[left]{$\tau$} (1,2);
    \draw[ultra thick](4,0) to node[right]{$\tau$} (4,2);
    \draw(1,2) to (2,0);
    \draw(4,2) to (5,0);
    \node[below] at (1,0) {$v_1^1$};
    \node[below] at (1.5,0) {$v_2^1$};
    \node[below] at (0.5,0) {$v_n^0$};
    \node[below] at (1+3,0) {$v_1^2$};
    \node[below] at (1.5+3,0) {$v_2^2$};
    \node[below] at (0.5+3,0) {$v_n^1$};
    \node[above] at (1,2){$w_1^1$};
    \node[above] at (1.5,2){$w_2^1$};
    \node[above] at (1+3,2){$w_1^2$};
    \node[above] at (1.5+3,2){$w_2^2$};
    \node[] at (2.5,1){$\cdots$};
    \end{tikzpicture}
    \end{tabular}
\end{center}

By Theorem \ref{thm:CombinatorialGrowthC}, $s_1^{out} = \sum_{w \in \pathset{0}{n+1}^{out}} wt_A(w)$ and $s_1^{in} = \sum_{w \in \pathset{0}{m+1}^{in}} wt_A(w)$, where the second sum is over matchings on the inner boundary. 
We will see that the set of matchings $\pathset{0}{n+1}^{out}$ and $\pathset{0}{m+1}^{in}$  come from several sets of matchings on $F$.

Consider first the sets $\pathset{v_1^1}{w_1^2}$ and $\pathset{w_1^2}{v_1^1}$, where these matchings occur on $F$. A matching $\bar{w} \in \pathset{v_1^1}{w_1^2}$ lifts to a matching on the outer boundary of the annulus, $w \in \pathset{1}{n+2}^{out}$ such that $w[n+1:n+1]$ is a subgon incident to $v_1$ counterclockwise of $\tau$. Note that $wt_T(\bar{w}) = wt_A(w)$. 
A matching $\bar{u} \in \pathset{w_1^2}{v_1^1}$ lifts to a matching in clockwise order on the inner boundary of the annulus, $u = \poly{m},\ldots,\poly{1}$ such that $\poly{1}$ is a subgon counterclockwise of $\tau$. Since reversing the order of a matching will not affect its weight, we see that the matchings $\bar{u}$ lift to the subset of $w \in \pathset{0}{m}^{in}$ where $w[1:1]$ is a subgon counterclockwise of $\tau$, and for each $\bar{u}$, $wt_T(\bar{u}) = wt_A(u)$. 

We can similarly compare $\pathset{w_1^1}{v_1^2}$ and $\pathset{v_1^2}{w_1^1}$. A matching $\bar{w} \in \pathset{v_1^2}{w_1^1}$ will lift to a matching on the outer boundary in clockwise order. When we reverse the order, we have a matching in $w \in \pathset{0}{n+1}^{out}$ where $w[1:1]$ is a subgon clockwise of $\tau$. Moreover, by Lemma \ref{lem:AnnulusWeightInvariant}, we can cyclically shift $w$ so that $w \in \pathset{1}{n+2}^{out}$; now, $w[n+1,n+1]$ is a subgon clockwise of $\tau$. We have again $wt_T(\bar{w}) = wt_A(w)$. 

A matching $\bar{u} \in \pathset{v_1^2}{w_1^1}$ will lift to a matching $u \in \pathset{1}{n+2}^{in}$ such that  $u[n+1:n+1]$ is a subgon counterclockwise of $\tau$. We shift $u$ so that $u \in \pathset{0}{m}^{in}$, now with $u[1:1]$ a subgon clockwise of $\tau$. We have that $wt_T(\bar{u}) = wt_A(u)$.

We can moreover see that, by following this procedure, every matching $w \in \pathset{1}{n+2}^{out}$ is a lift of a unique matching in either $\pathset{v_1^2}{w_1^1}$ or $\pathset{v_1^1}{w_1^2}$, depending on which side of $\tau$ the subgon $w[n+1:n+1]$ lies. Since our lifts also preserved weight, we have that \[
\sum_{w \in \pathset{v_1^1}{w_1^2}} wt_T(w) + \sum_{w \in \pathset{v_2^1}{w_1^1}} wt_T(w) = \sum_{w \in \pathset{1}{n+2}^{out}} wt_A(w)
\]

We can similarly conclude \[
\sum_{w \in \pathset{w_1^2}{v_1^1}} wt_T(w) + \sum_{w \in \pathset{w_1^1}{v_2^1}} wt_T(w) = \sum_{w \in \pathset{0}{m}^{in}} wt_A(w)
\]

By Lemma \ref{lem:EitherDirection}, we have \[
\sum_{w \in \pathset{v_1^1}{w_1^2}} wt_T(w) + \sum_{w \in \pathset{v_2^1}{w_1^1}} wt_T(w) = \sum_{w \in \pathset{w_1^2}{v_1^1}} wt_T(w) + \sum_{w \in \pathset{w_1^1}{v_2^1}} wt_T(w) 
\]
and by our combinatorial interpretation of $s_1$ in Theorem \ref{thm:CombinatorialGrowthC}, the claim of the Theorem follows.

\end{proof}

\section{Connection To T-Paths}

In the case of triangulated surfaces, it is known that matchings are in bijection with several other combinatorial objects. One such set of objects is \emph{T-paths}. 

T-paths on polygons were first developed in \cite{18} to give a combinatorial proof of positive of cluster variables in type $A$. See \cite{12} for the definition of a cluster algebra.  Carroll and Price show a bijection between T-paths, with all edges weight 1, and BCI tuples in the case of a polygon \cite{9}, \cite{17}. The definition of T-path was expanded to arbitrary triangulated surfaces in \cite{19} and \cite{16}.

Recently, \c{C}anak\c{c}i and J{\o}rgensen extended the definition of T-paths to the case of dissected polygon. We give their definition below. 
\begin{definition}\label{def:Tpath}[Weak T-path in a Dissected Polygon \cite{8}]
Given an $n$-gon with vertices $v_1,\ldots,v_n$ and dissection $\D$, let $v_i,v_j$ be a distinct pair of vertices. A \emph{weak T-path}, $\alpha$, from $v_i$ to $v_j$ is a walk $\alpha_1,\ldots,\alpha_{2k+1}$, with the following properties. \begin{enumerate}
    \item No step $\alpha_i$ crosses an arc in $\D$.
    \item Each $\alpha_i$ is equipped with an orientation, such that
    \begin{itemize}
    \item $s(\alpha_1) = v_i$
    \item $t(\alpha_{2k+1}) = v_j$
    \item For all $1 \leq i < 2k+1$, $t(\alpha_i) = s(\alpha_{i+1})$.
    \end{itemize}
    \item The even steps, $\alpha_{2\ell}$, each cross $(v_i,v_j)$.
    \item For all $\ell$, the crossing point of $(v_i,v_j)$ and $\alpha_{2\ell}$ is closer to $v_i$ then the crossing point of $(v_i,v_j)$ and $\alpha_{2(i+1)}$.
    \item The steps $\alpha_i$, considered without orientation, are pairwise distinct.
\end{enumerate}
\end{definition}

The third condition of Definition \ref{def:Tpath} explains why a weak T-path must have odd length while the fourth and fifth prevent backtracking. From now on, we will simply refer to these as T-paths,

\begin{remark}\label{rem:completeTpath}
One can equivalently define \emph{complete T-paths}. In a polygon $P$ with vertices $v_1,\ldots,v_n$ and dissection $\D$, consider the diagonal $(v_i,v_j)$. Orienting this diagonal from $v_i$ to $v_j$, let $\tau_1,\ldots,\tau_d$ be the ordered list of arcs of $\D$ that this diagonal crosses. Then a complete T-path $\alpha$ from $v_i$ to $v_j$ is a walk $\alpha_1,\ldots,\alpha_{2d+1}$ which satisfies items 1 and 2 in Definition \ref{def:Tpath} as well as 
\begin{enumerate}
    \item[3'] The step $\alpha_{2\ell}$ goes along $\tau_\ell$.
\end{enumerate}

A complete T-path will often violate condition 5 in Definition \ref{def:Tpath} as two consecutive steps may go along the same edge in $\D$. 
\end{remark}

If we set a weight on each diagonal of a polygon (not just those in the dissection), we can weight T-paths by taking a ratio of weights on the diagonals making up the T-path. 

\begin{definition}[Weight of a T-path]
Given an $n$-gon with vertices $v_1,\ldots,v_n$ and dissection $\D$, and a T-path $\alpha$ from $v_i$ to $v_j$, $\alpha = \alpha_1,\ldots,\alpha_{2k+1}$, we define the weight of $\alpha$ as \[
wt(\alpha) = \frac{\displaystyle\prod_{j = 0}^k wt(\alpha_{2j+1})}{\displaystyle\prod_{j=1}^k wt(\alpha_{2j})}
\]
\end{definition}

Given a polygon $P$ with dissection $\D$, we will set $wt((v,w)) = U_k(\lambda_p)$ whenever $(v,w)$ is a $k$-diagonal of a subgon of size $p$. Note that $U_0(x) = 1$, so this weighting sets all diagonals in $\D$ and all boundary edges to weight 1. 

\begin{remark}
In \cite{8}, the authors define a map $f$ on diagonals of a polygon to be a \emph{frieze} if it satisfies the Ptolemy relation: given vertices on a polygon $a,b,c,d,$  which appear in this cyclic order, so that $(a,c)$ and $(b,d)$ are crossing, $f$ satisfies, \[
f((a,c))f((b,d)) = f((a,b))f((c,d)) + f((a,d))f((b,c))
\]

They define a \emph{weak frieze} to be a map which satisfies the Ptolemy relation when $(b,d) \in D$. 

One of their main results (Theorem A) is that a map $f$ on diagonals of a polygon $P$ is a weak frieze if and only if it satisfies the $T$-path formula. However, our choice of weighting will in fact be a \emph{frieze}, as it is equivalent to the frieze pattern discussed in \cite{14}.
\end{remark}

Our main result of this section is that there is a bijection between T-paths in a dissected polygon and matchings with nonzero traditional weight. Denote the set of T-paths from $v_i$ to $v_j$, where these vertices are distinct, as $\textbf{T}_{i,j}$.

\begin{proposition}\label{prop:Tpath}
Let $P$ be a polygon with vertices $v_1,\ldots,v_n$. Let $v_i,v_j$ be any pair of distinct vertices. Then, there is a bijection, $\Phi: \{w \in \pathset{i}{j}: wt_T(w) \neq 0 \} \to \textbf{T}_{i,j}$ such that, if $\Phi(w) = \pi$, then $wt_T(w) = wt(\pi)$.
\end{proposition}

\begin{proof}
Throughout this proof, we will assume that $(v_i,v_j)$ crosses all arcs in $\D$. For both a matching and a T-path, if this was not the case, we could work with the portion of $P$ which is only dissected by arcs which cross $(v_i,v_j)$, then paste the remaining subgons. Accordingly, let $\tau_1,\ldots,\tau_d$ be the arcs in $\D$, with order imposed by the order in which the arc $(v_i,v_j)$, oriented from $v_i$ to $v_j$, crosses them. By the same procedure we also index the subgons of the dissection in order, $Q_0,\ldots,Q_d$. By our assumption that $(v_i,v_j)$ crosses all arcs in $\D$, $Q_0$ and $Q_d$ are ears. 

\begin{figure}
    \centering
\begin{tikzpicture}[scale = 1.7]
\node[circle,  fill = black, scale = 0.1](P0) at (3,-1){};
\node[circle, fill = black, scale = 0.1](P1) at (4,-1){};
\node[circle, fill = black, scale = 0.1](P2) at (5,-1){};
\node[circle, fill = black, scale = 0.1](P3) at (6,-1){};
\node[left](sg) at (0,0){$v_i$};
\node[right](tg) at (9,0){$v_j$};
\draw[dashed, red] (sg) -- (tg);
\draw(0,0) -- (1,-1) -- (2,-1) -- (5,-1) -- (6,-1) -- (7,-1) -- (8,-1) -- (9,0) -- (8,1) -- (7,1) -- (5,1) -- (3,1) -- (1,1) -- (0,0);
\draw(1,1) -- (2,-1);
\draw(3,1) -- (2,-1);
\draw(5,1) -- (5,-1);
\draw(5,1) -- (6,-1);
\draw(7,1) -- (7,-1);
\draw(7,-1) -- (6,1);
\node[below] at (7,-1) {$v_a$};
\node[above] at (7,1) {$v_b$};
\node[] at (1, -0.5) {$Q_0$};
\node[] at (2,0.5) {$Q_1$};
\node[] at (4,-0.5){$Q_2$};
\node[] at (5.3,-0.5){$Q_3$};
\node[] at (7.8,-0.5) {$Q_{d}$};
\node[] at (6.2,-0.5){$\cdots$};
\node[] at (6.7, 0.5) {$Q_{d-1}$};
\end{tikzpicture}
    \caption{A typical configuration for a diagonal in a dissected polygon. We only consider the subgons which the diagonal crosses.}
    \label{fig:my_label}
\end{figure}
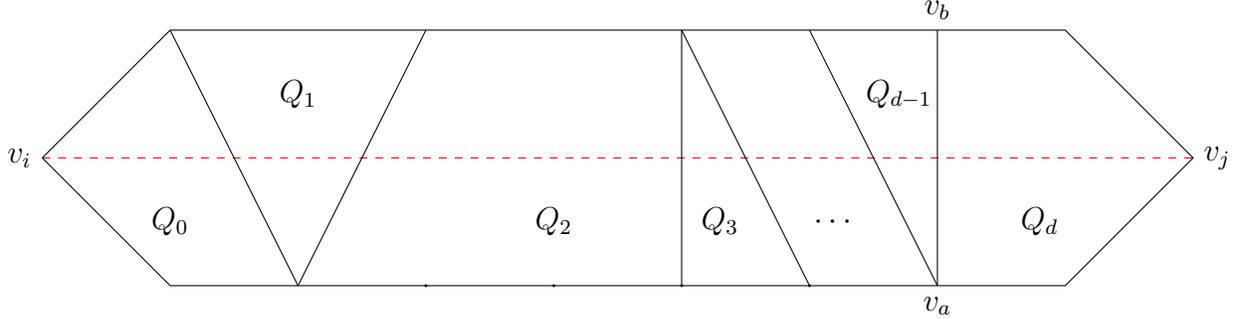

We define the image $\Phi(w)$ with the following claim.

\textbf{Claim:} Given a matching $w \in \pathset{i}{j}$, there exists a complete T-path $\alpha = \alpha_1,\ldots,\alpha_{2d+1}$ such that the number of vertices of $Q_i$ to the left of $\alpha_{2i+1}$ is equal to the number of occurrences of $Q_i$ in $w$.

\begin{proof}[Proof of Claim]
We prove the claim by induction on $d$. First, let $d = 0$. Then, our dissection is an empty dissection, with one subgon $P$. There is only one matching in $\pathset{i}{j}$, which uses $P$ $j-i-1$ times. Similarly, there is one T-path between $v_i$ and $v_j$, which simply consists of the diagonal $(v_i,v_j)$. This arc has $j-i-1$ vertices of $P$ to its left. 

For the inductive step, assume the claim is true for any arc which crosses less than $d$ arcs, and let $(v_i,v_j)$ cross $d$ arcs. Let $Q_d$ be the subgon which has vertex $v_j$. Let $v_a,v_b$ be the vertices of $\tau_d$, so that $v_a$ appears when traveling counterclockwise from $v_i$ to $v_j$. For convenience, relabel the vertices so that $i < a < j < b$.

Given any $w \in \pathset{i}{j}$, either $w[a:a] = Q_d$ or $w[a:a] \neq Q_d$. In the former case, we can consider $w[i+1:a-1]$ as a matching in $\pathset{i}{a}$. The arc $(v_i,v_a)$ crosses $k < d$ arcs. Thus, there is a T-path $\alpha' = \alpha_1,\ldots,\alpha_{2k+1}$ with the property that for all $1 \leq \ell \leq k$, the number of vertices of $Q_\ell$ to the left of $\alpha_{2\ell + 1}$ is equal to the number of occurrences of $Q_\ell$ in $w[i+1:a-1]$. We extend $\alpha'$ to a T-path $\alpha \in \textbf{T}_{i,j}$. For all $k < \ell \leq d$, we set $\alpha_{2\ell}$ to be the arc $\tau_\ell$ oriented away from $a$. If $\ell < d$, we set $\alpha_{2\ell + 1}$ to be $\tau_{\ell}$ oriented towards $a$. Finally, we set $\alpha_{2d + 1}$ to be the diagonal from $v_b$ to $v_j$. Such a T-path has zero vertices of $Q_{\ell}$ to the left of $\alpha_{2\ell + 1}$ for all $k < \ell < d$ and has $j - a > 0$ vertices of $Q_d$ to the left of $\alpha_{2d + 1}$. Together with the inductive hypothesis, we have constructed a T-path with the appropriate conditions.

Next, given any T-path $w \in \pathset{i}{j}$ with $w[a:a] \neq Q_d$, we consider $w[i+1:a]$ as a T-path between $v_i$ and $v_b$ in the polygon resulting from cutting all subgons not crossed by $(v_i,v_b)$. Since the diagonal $(v_i,v_b)$ crosses less than $d$ arcs, we can use the inductive hypothesis to construct a T-path associated to $w[i+1:a]$ with the conditions in the claim. By a parallel argument to the previous case, we can extend this to a T-path from $v_i$ to $v_j$ with the desired conditions.
\end{proof}

Given a matching $w \in \pathset{i}{j}$, we define $\Phi(w)$ to be the complete T-path described in the claim. Given two distinct matchings $w,w'$, the images $\Phi(w)$ and $\Phi(w')$ will be distinct T-paths since, for at least one value $1 \leq k \leq d$, the steps at index $2k+1$ in the two T-paths will be distinct. For surjectivity, given a complete T-path $\alpha$, we can construct a matching $w \in \pathset{i}{j}$ so that $\alpha$ satisfies the hypotheses of the claim for $w$. A proof of this would follow a similar induction as the proof of the Claim.

We check that $wt_T(w) = wt(\Phi(w))$. If $w$ uses subgon $Q_k$ $m$ times, the contribution to $wt_T(w)$ is $U_{m}(\lambda_{\vert Q_k \vert})$. In $\Phi(w)$, step $\alpha_{2k+1}$ will cross $Q_k$ and have $m$ vertices to its left. Thus, $\alpha_{2k+1}$ is a $m$-diagonal in a regular $\vert Q_k \vert$-gon, so its contribution to $wt(\Phi(w))$ is also $U_{m}(\lambda_{\vert Q_k \vert})$.
\end{proof}

\section{Positivity}\label{sec:Positivity}

Since $wt_T(w) \geq 0$ for any matching $w$, it follows from  Theorem \ref{thm:EqualWeights} and Corollary \ref{cor:TraditionalWeight} that every frieze pattern from a dissection of a once-punctured disc or annulus has all positive entries.

\begin{corollary}\label{cor:TradWeightPositive}
If $\mathcal{F} = \{m_{i,j}\}$ is an infinite frieze pattern from a dissected annulus or once-punctured disc, then for all $i < j$, $m_{i,j} > 0$.
\end{corollary}

\begin{proof}
We must show that, for all $i < j$, there is at least one matching $w \in \pathset{i}{j}$ with $wt_T(w) > 0$. From Equation  \ref{eq:CutTradWeight}, we know that if any edge from $\mathcal{D}$ is strictly between $v_i$ and $v_j$, so that there are ears between the vertices, we can remove these and the sum over matchings using $wt_T$ will be unchanged. From the proof of Equation \ref{eq:TradWeightDiffEar}, we know the same is true for edges weakly between $v_i$ and $v_j$. Thus, it suffices to consider the case where there does not exist an edge $(v_\ell,v_m) \in \mathcal{D}$ with $i \leq \ell < m \leq j$.

  In this case, for any subgon $\mathcal{P}$ with vertices strictly between $v_i$ and $v_j$, a matching can choose $\mathcal{P}$ at most $\vert \mathcal{P} \vert - 2$ times. For a concrete example, one could form a matching $\poly{i+1},\ldots,\poly{j-1}$ by setting $\poly{k}$ as the subgon which has edge $(v_k, v_{k+1})$ for all $i+1 \leq k \leq j-1$. This is the matching which uses the counterclockwise-most subgon incident to each vertex. 

\end{proof}

We consider the positivity of the other frieze patterns of Type $\Lambda_{p_1,\ldots,p_s}$ - either those realizable by a quotient dissection or unrealizable frieze patterns. The following results determine a larger class of frieze patterns which are guaranteed to be positive.

\begin{proposition}\label{prop:GreaterThan2}
Let $\mathcal{F}$ be a frieze pattern where all entries in the quiddity sequence are at least 2. Then, all entries of $\mathcal{F}$ are positive.
\end{proposition}

\begin{proof}
Baur, Parsons, and Tschabold proved this in the case of positive integers in \cite{1}. Their proof can be modified for frieze patterns with more general entries. In particular, we can combine the fact that the 1-periodic frieze pattern determined by quiddity row $\ldots, 2,2, \ldots$ has all positive entries with Theorem 2.1 of \cite{1} applied to $b \in \mathbb{R}_{\geq 0}$.
\end{proof}

\begin{corollary}
Any sequence $\{m_{i-1,i+1}\}$ such that $m_{i-1,i+1} = \sum_{p \in A_i}\lambda_p$ and $\vert A_i \vert > 1$ for all $i$ determines a frieze pattern with all positive entries.
\end{corollary}

\begin{proof}
Recall that $\lambda_p = 2\cos(\frac{\pi}{p})$, which implies that $\lambda_p \geq 1$, with equality when $p = 3$.
\end{proof}

Note that this immediately implies there are unrealizable frieze patterns with all positive entries. One example is the frieze pattern with quiddity cycle $(2, 2 \sqrt{2})$. There are also unrealizable frieze patterns with negative entries; we provide an example below.

$\begin{array}{cccccccccccccccccccc}
  0&&0&&0&&0&&0\\
 &1& &1&&1&&1&\\
 1&&\sqrt{2}&&1&&\sqrt{2}&&1 \\
&\sqrt{2}-1& &\sqrt{2}-1&&\sqrt{2}-1&&\sqrt{2}-1& \\
  2 - 2\sqrt{2}&& \sqrt{2} - 2&&2 - 2\sqrt{2}&&\sqrt{2} - 2&&2 - 2\sqrt{2}\\
  &&&&\vdots&&&&&& \\
 \end{array}$

Next, we take advantage of the Progression Formulas in \cite{13} to provide another way to check for positivity of an arbitrary frieze pattern.

\begin{proposition}\label{prop:PositivityCheck}
Let $\mathcal{F} = \{m_{i,j}\}$ be an $n$-periodic, infinite frieze pattern. Suppose that the first $n$ nontrivial rows are positive; that is $m_{i,j} > 0$ whenever $\vert j - i\vert \leq n+1$. Moreover, suppose that the first growth coefficient, $s_1 = m_{0,n+1} - m_{1,n}$, satisfies $s_1 \geq 2$. Then, all entries of $\mathcal{F}$ are positive.
\end{proposition}
\begin{proof}
First, note by Lemma 4.3 of \cite{2} that if $s_1 \geq 2$, then $s_k \geq 2$ where $s_k$ is the $k$-th growth coefficient. 

Next, we recall a special case of Theorem 5.4 of \cite{13}, known as the Progression Formulas. Let $1 \leq i,j \leq n$ and $k \in Z_{\geq 1}$. Moreover, first assume $i < j$. Then,

\[
m_{i,j+(k-1)n} = s_{k-1}m_{i,j} + m_{j, i+k(n-1)}
\]

Similarly, suppose $i \geq j$. Then, \[
m_{i,j+kn} = s_{k-1}m_{i,j+n} + m_{j, i + kn}
\]

We see we can write entries which are below the $n$-th row in terms of entries which are in higher rows and growth coefficients. Thus, given that the first $n$ rows are positive, we can conclude by induction that all entries in $\mathcal{F}$ are positive. 

\end{proof}

Note that in the example of an unrealizable frieze pattern with negative entries, the first growth coefficient is $\sqrt{2}$, which is positive but not larger than 2. 

We can use Proposition \ref{prop:PositivityCheck} to check for positivity of frieze patterns. Based on experiments, we conclude with the following conjecture.

\begin{conjecture}\label{conj:Positive}
The only frieze patterns of Type $\Lambda_{p_1,\ldots,p_s}$ with negative entries are those which fail the realizability test ( Proposition \ref{proposition:unrealizability1}). In particular, all quotient dissections yield frieze patterns with positive entries. 
\end{conjecture}

\section{Acknowledgements}
This research was largely conducted at the 2020 University of Minnesota Twin Cities REU, supported by NSF RTG grant DMS-1745638. We thank Gregg Musiker for feedback throughout the process and Libby Farrell for help with the manuscript and support during the REU.


\end{document}